\pdfoutput=1
\documentclass[11pt]{article}
\usepackage[T1]{fontenc}
\usepackage{lmodern}

\usepackage{amsmath,amsthm,amssymb}
\usepackage[usenames,dvipsnames]{xcolor}
\usepackage{enumerate}
\usepackage{graphicx}
\usepackage{cite}
\usepackage{comment}
\usepackage{oands}
\usepackage{tikz}
\usepackage{changepage}
\usepackage{bbm}
\usepackage{mathtools}
\usepackage[margin=1in]{geometry}
\usepackage[pagewise,mathlines]{lineno}
\usepackage{appendix}
\usepackage{stmaryrd}
\usepackage{multicol}
\usepackage{microtype}
\usepackage[colorinlistoftodos]{todonotes}
\usepackage{dsfont}

\usepackage[pdftitle={On the geometry of uniform meandric systems},
  pdfauthor={Jacopo Borga, Ewain Gwynne, and Minjae Park},
colorlinks=true,linkcolor=NavyBlue,urlcolor=RoyalBlue,citecolor=PineGreen,bookmarks=true,bookmarksopen=true,bookmarksopenlevel=2,unicode=true,linktocpage]{hyperref}

\theoremstyle{plain}
\newtheorem{thm}{Theorem}[section]
\newtheorem{cor}[thm]{Corollary}
\newtheorem{lem}[thm]{Lemma}
\newtheorem{prop}[thm]{Proposition}
\newtheorem{conj}[thm]{Conjecture}

\def\@rst #1 #2other{#1}
\newcommand\MR[1]{\relax\ifhmode\unskip\spacefactor3000 \space\fi
  \MRhref{\expandafter\@rst #1 other}{#1}}
\newcommand{\MRhref}[2]{\href{http://www.ams.org/mathscinet-getitem?mr=#1}{MR#2}}

\theoremstyle{definition}
\newtheorem{defn}[thm]{Definition}
\newtheorem{remark}[thm]{Remark}

\newtheorem{ques}[thm]{Question}

\numberwithin{equation}{section}

\newcommand{\dsb}{\begin{adjustwidth}{2.5em}{0pt}
\begin{footnotesize}}
\newcommand{\dse}{\end{footnotesize}
\end{adjustwidth}}

\newcommand{\ssb}{\begin{adjustwidth}{2.5em}{0pt}}
\newcommand{\sse}{\end{adjustwidth}}

\newcommand{\aryb}{\begin{eqnarray*}}
\newcommand{\arye}{\end{eqnarray*}}
\def\alb#1\ale{\begin{align*}#1\end{align*}}
\def\allb#1\alle{\begin{align}#1\end{align}}
\newcommand{\eqb}{\begin{equation}}
\newcommand{\eqe}{\end{equation}}
\newcommand{\eqbn}{\begin{equation*}}
\newcommand{\eqen}{\end{equation*}}

\newcommand{\BB}{\mathbbm}
\newcommand{\ol}{\overline}

\newcommand{\op}{\operatorname}

\newcommand{\frk}{\mathfrak}
\newcommand{\eqD}{\overset{d}{=}}
\newcommand{\ep}{\varepsilon}
\newcommand{\rta}{\rightarrow}

\newcommand{\wt}{\widetilde}
\newcommand{\wh}{\widehat}
\newcommand{\mcl}{\mathcal}

\newcommand{\bdy}{\partial}

\newcommand{\ms}{{\mathfrak S}}

\let\originalleft\left
\let\originalright\right
\renewcommand{\left}{\mathopen{}\mathclose\bgroup\originalleft}
\renewcommand{\right}{\aftergroup\egroup\originalright}

\title{On the geometry of uniform meandric systems}
 \date{ }
 \author{
\begin{tabular}{c} Jacopo Borga\\[-5pt]\small Stanford University \end{tabular}
\begin{tabular}{c} Ewain Gwynne\\[-5pt]\small University of Chicago \end{tabular} 
\begin{tabular}{c} Minjae Park\\[-5pt]\small University of Chicago \end{tabular} 
}

\begin{document}

\maketitle

\begin{abstract}
A meandric system of size $n$ is the set of loops formed from two arc diagrams (non-crossing perfect matchings) on $\{1,\dots,2n\}$, one drawn above the real line and the other below the real line.   
A uniform random meandric system can be viewed as a random planar map decorated by a Hamiltonian path (corresponding to the real line) and a collection of loops (formed by the arcs). 
Based on physics heuristics and numerical evidence, we conjecture that the scaling limit of this decorated random planar map is given by an independent triple consisting of a Liouville quantum gravity (LQG) sphere with parameter $\gamma=\sqrt 2$, a Schramm-Loewner evolution (SLE) curve with parameter $\kappa=8$, and a conformal loop ensemble (CLE) with parameter $\kappa=6$.

We prove several rigorous results which are consistent with this conjecture. 
In particular, a uniform meandric system admits loops of nearly macroscopic graph-distance diameter with high probability.  
Furthermore, a.s., the uniform infinite meandric system with boundary has no infinite path. 
But, a.s., its boundary-modified version has a unique infinite path whose scaling limit is conjectured to be chordal SLE$_6$.  
\end{abstract}

 

\tableofcontents

\bigskip
\noindent\textbf{Acknowledgments.} We thank two anonymous referees for helpful comments on an earlier version of this article. We thank Ahmed Bou-Rabee, Valentin Féray, Gady Kozma, Ron Peled, and Xin Sun for helpful discussions. E.G.\ was partially supported by a Clay research fellowship. M.P.\ was partially supported by an NSF grant DMS 2153742.

\section{Introduction}
\label{sec-intro}

\subsection{Meandric systems}
\label{sec-system}
Throughout this paper, we identify $\BB R$ (resp.\ $\BB Z$) with the set $\BB R\times \{0\} $ (resp.\ $\BB Z\times\{0\}$).

\begin{defn} \label{def-system}
A \textbf{meandric system} of size $n\in\BB N$ is a configuration $\ms_n$ consisting of a finite collection of simple loops in $\BB R^2$ with the following properties:
\begin{itemize}
\item No two loops of $\ms_n$ intersect each other.
\item Each loop of $\ms_n$ intersects the real line $\BB R$ at least twice, and does not intersect $\BB R$ without crossing it. 
\item The total number of intersection points between the loops of $\ms_n$ and $\BB R$ is equal to $2n$. 
\end{itemize}
We view such configurations as being defined modulo orientation-preserving homeomorphisms from $\BB R^2$ to $\BB R^2$ which take $\BB R$ to $\BB R$.
\end{defn}

\begin{figure}[ht!]
\begin{center}
\includegraphics[width=0.75\textwidth]{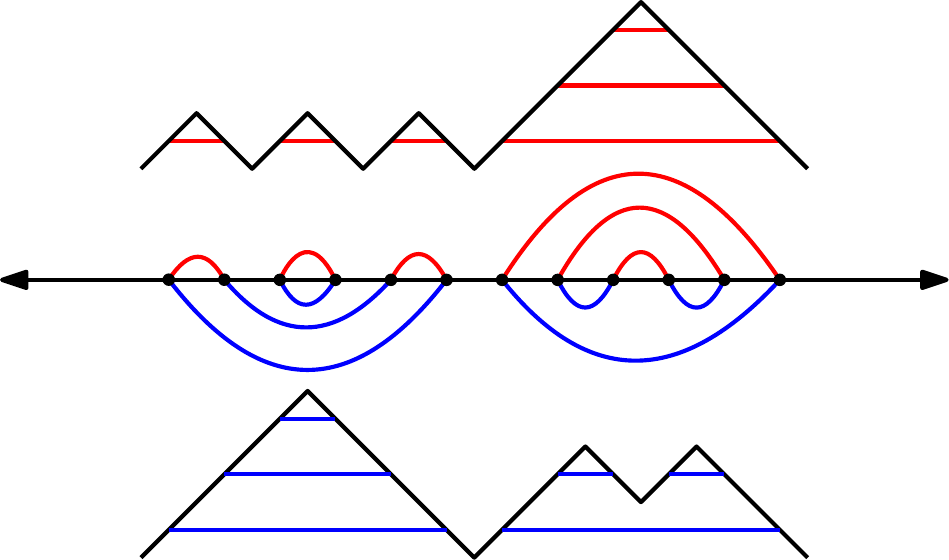}  
\caption{\label{fig-system} A meandric system of size $n=6$. The two corresponding arc diagrams are shown in red and blue. The graphs of the simple walk excursions corresponding to these arc diagrams are shown in black. The walk excursions are translated (in time) by a $1/2$-factor to the right so that each step of the walks is centered with the corresponding integer point in the real line. We made this choice (here and in all our graphical representations of meandric systems) to highlight better the correspondence between walks and arcs. See~\eqref{eqn-arc-walk} and the surrounding text for further explanation.
}
\end{center}
\vspace{-3ex}
\end{figure}

See Figure~\ref{fig-system} for an illustration of a meandric system. 
If $\ms_n$ is a meandric system of size $n$, then by applying a homeomorphism, we can always arrange the set of intersection points of arcs in $\ms_n$ with $ \BB R$ so that it is equal to $[1,2n]\cap\BB Z$. We will make this assumption throughout the paper. 
There have been several recent works in probability and combinatorics which studied meandric systems (see, e.g.,~\cite{ckst-noodle,fn-meander-system,gnp-meander-system,kargin-meander-system,fv-noodle}). 
These works were in part motivated by the study of decorated planar maps and by the connection between meandric systems, non-crossing partitions, and meanders, as we discuss just below. 
Additional motivations come from the fact that random meandric systems are equivalent to a certain percolation-type model on a random planar map (Section~\ref{sec-perc}) and to a version of the fully packed $O(0 \times 1)$ loop model on a random planar map (Section~\ref{sec-physics}). 

A \textbf{meander} of size $n$ is a meandric system of size $n$ with a single loop. 
Each of the loops in a meandric system can be viewed as a meander by forgetting the other loops. However, a typical loop in a uniformly sampled meandric system of size $n$ is not the same as a uniformly sampled meander~\cite[Section 4]{fv-noodle}. 
The study of meanders dates back to at least the work of Poincar\'e in 1912~\cite{poincare-meander} and is connected to a huge number of different areas of math and physics. 
See~\cite{lacroix-meander-survey, zvonkin-meander-survey} for surveys of results on meanders. 

Most features of meanders are notoriously difficult to analyze mathematically. For example, determining the $n\rta\infty$ asymptotics of the total number of meanders of size $n$ is a long-standing open problem (but see~\cite{dgg-meander} for a conjecture). 

Meandric systems are significantly more tractable than meanders. The main reason for this is that meandric systems are in bijection with pairs of arc diagrams (non-crossing perfect matchings).
An \textbf{arc diagram} of size $n\in\BB N$ is a collection of arcs in the upper half-plane $\BB R \times [0,\infty)$, each of which joint two points in $[1,2n]\cap\BB Z$, subject to the condition that no two of the arcs cross. If $\ms_n$ is a meandric system of size $n$, then the segments of loops in $\ms_n$ above (resp.\ below) the real line form an arc diagram. Conversely, any two arc diagrams of size $n$ give rise to a meandric system by drawing one above and one below the real line, and considering the set of loops that they form.  
It is well known that arc diagrams of size $n$ are counted by the Catalan number $\op{Cat}_n = \frac{1}{n+1} \binom{2n}{n}$. Consequently, the number of meandric systems of size $n$ is $\op{Cat}_n^2$.

Furthermore, arc diagrams of size $n$ are in bijection with $2n$-step simple walks on $\BB Z_{\ge 0}$ from $0$ to $0$, often called (non-negative) simple walk excursions or Dyck paths.
If $\mcl X : [0,2n]\cap\BB Z \rta \BB Z$ is a $2n$-step simple walk excursion, then the corresponding arc diagram is defined as follows.
Two points $x_1,x_2\in[1,2n]\cap\BB Z$ with $x_1 < x_2$ are joined by an arc if and only if 
\eqb \label{eqn-arc-walk}
 \mcl X_{x_1-1} = \mcl X_{x_2} < \min_{y \in [x_1 ,x_2-1]\cap\BB Z} \mcl X_y ;
\eqe 
see Figure~\ref{fig-system}.
So, one can sample a uniform random meandric system of size $n$ by sampling two independent simple random walk excursions with $2n$ steps, drawing one of the two corresponding arc diagrams above the real line and the other below the real line, then looking at the loops formed by the union of the two arc diagrams. 

Let $\frk S_n$ be a uniform meandric system of size $n$. There are a number of natural questions about the large-scale geometry of $\frk S_n$, e.g., the following:  
\begin{enumerate}
\item \label{item-loop-count} How many loops does $\frk S_n$ typically have?  
\item \label{item-largest} What is the size of the largest loop of $\frk S_n$, in terms of the number of intersection points with $\BB R$? What about for other notions of size, e.g., graph-distance diameter in the 4-regular graph whose vertices are the intersection points of loops with $\BB R$, and whose edges are the segments of loops and the segments of $\BB R$ between these vertices?
\item \label{item-macro} Is there typically a single loop of $\frk S_n$ which is much larger (in some sense) than the other loops, or are there multiple large loops of comparable size?
\item \label{item-limit} Is there some sort of scaling limit of $\frk S_n$ as $n\rta\infty$? 
\end{enumerate}
Due to the bijection between meandric systems and pairs of $2n$-step simple walk excursions, questions of the above type, in principle, can be reduced to questions about simple random walks on $\BB Z$. However, the encoding of the meandric system loops in terms of the pair of walk excursions is complicated, so the answers to the above questions are far from trivial.  

Question~\ref{item-loop-count} was largely solved by F\'eray and Th\'evenin~\cite{fv-noodle}, who showed that there is a constant $c > 0$ (expressed in terms of a sum over meanders) such that the number of loops in $\frk S_n$ is asymptotic to $c n$ as $n\rta\infty$.  
Regarding Question~\ref{item-largest}, Kargin~\cite{kargin-meander-system} showed that the number of intersection points with $\BB R$ of the largest loop is at least constant times $\log n$, and presented some numerical simulations which suggested that this quantity in fact behaves like $n^\alpha$ for $\alpha\approx 4/5$. 
Question~\ref{item-macro} is closely related to the question of whether there exists a so-called ``infinite noodle'', i.e., an infinite path in the infinite-volume limit of a uniform meandric system. It was shown in~\cite{ckst-noodle} that there is at most one such path. The existence is still open, but it is conjectured in~\cite{ckst-noodle} that such a path does not exist. See Section~\ref{sec-infinite} for further discussion.  

In this paper, we present conjectures for the answers to each of Questions~\ref{item-largest}, \ref{item-macro}, and~\ref{item-limit} (see Conjectures~\ref{conj-system} and~\ref{conj-largest}). In particular, if $k\in\BB N$ is fixed, then as $n\rta\infty$ the number of intersection points with $\BB R$ of the $k$th largest loop should grow like $n^\alpha$ where $\alpha = \frac12 (3-\sqrt 2) \approx 0.7929$. Moreover, the scaling limit of $\frk S_n$ should be described by a $\sqrt 2$-Liouville quantum gravity sphere, a Schramm-Loewner evolution curve with parameter $\kappa=8$, and a conformal loop ensemble with parameter $\kappa=6$.

We also prove several rigorous results in the direction of the above questions. 
We show that a uniform meandric system admits loops of nearly macroscopic graph-distance diameter (Theorem~\ref{thm-macro}). 
This leads to an explicit power-law lower bound for the size of the largest loop in a uniform meandric system (Corollary~\ref{cor-largest}).
We also construct the \emph{uniform infinite half-plane meandric system (UIHPMS)} and show that it does not admit any infinite paths of arcs (Theorem~\ref{thm-infinite-path-half}).
But, a minor modification of the UIHPMS admits a unique infinite path of arcs which should converge to SLE$_6$ (Proposition~\ref{prop-alternating}). 

Most of our proofs use only elementary discrete arguments, but we need to use the theory of Liouville quantum gravity at one step in the proof, namely in Section \ref{sec-lqg}.

\subsection{Conjectures for scaling limit and largest loop exponent}
\label{sec-conj}
 
We want to state a conjecture for the scaling limit of the uniformly sampled meandric system $\frk S_n$ as $n\rta\infty$. 
To formulate this conjecture, we let $\mathcal M_n$ be the planar map whose vertices are the $2n$ intersection points of the loops $\ell$ in $\frk S_n$ with the real line, whose edges are the segments of the loops or the line $\BB R$ between these intersection points (we consider the two infinite rays of $\BB R$ as being an edge from the leftmost to the rightmost intersection point), and whose faces are the connected components of $\BB R^2 \setminus \left( \bigcup_{\ell \in \frk S_n} \ell  \cup \BB R \right)$. 

The planar map $\mathcal M_n$ is equipped with a Hamiltonian path $P_n : [1,2n] \cap \BB Z \rta \{\text{vertices of $\mathcal M_n$}\}$ which traverses the vertices of $\mathcal M_n$ and the segments of $\BB R$ between these vertices in left-right numerical order. The planar map is also equipped with a collection of loops $\Gamma_n$ (simple cycles in $\mathcal M_n$) corresponding to the loops in $\frk S_n$. See Figure~\ref{fig-cle-sim} for a simulation of $(\mathcal M_n,P_n,\Gamma_n)$. 
 
To talk about convergence, we can, e.g., view $(\mathcal M_n,P_n,  \Gamma_n)$ as a metric measure space (equipped with the graph metric and the counting measure on vertices) decorated by a path and a collection of loops. 
We can then ask whether this decorated metric measure space has a scaling limit with respect to the generalization of the Gromov-Hausdorff topology for metric measure spaces decorated by curves and/or loops~\cite{gwynne-miller-uihpq,ghs-metric-peano}.
Alternatively, we could embed $(\mathcal M_n,P_n,\Gamma_n)$ into $\BB C$ in some manner (e.g., Tutte embedding~\cite{gms-tutte} as in Figure~\ref{fig-cle-sim} or circle packing~\cite{stephenson-circle-packing}) and ask whether the resulting metric, measure, curve, and collection of loops in $\BB C$ have a joint scaling limit in law. 

We now state a conjecture for the scaling limit of $(\mathcal M_n,P_n,\Gamma_n)$. The limiting object is described in terms of several random objects whose definitions we will not write down explicitly. 
\begin{itemize}
\item The \textbf{Liouville quantum gravity (LQG)} sphere with parameter $\gamma \in (0,2]$ is a random fractal surface with the topology of the sphere first introduced in~\cite{wedges,dkrv-lqg-sphere}. A $\gamma$-LQG sphere can be described by a random metric and a random measure on the Riemann sphere $\BB C\cup \{\infty\}$. LQG spheres (and other types of LQG surfaces) describe the scaling limits of various types of random planar maps. See~\cite{gwynne-ams-survey,sheffield-icm,bp-lqg-notes} for expository articles on LQG.
\item \textbf{Schramm-Loewner evolution (SLE$_\kappa$)} with parameter $\kappa > 0$ is a random fractal curve introduced in~\cite{schramm0}. The curve is simple for $\kappa \in (0,4]$, has self-intersections but not self-crossings for $\kappa \in (4,8)$, and is space-filling for $\kappa \geq 8$. 
\item The \textbf{conformal loop ensemble (CLE$_\kappa$)} with parameter $\kappa \in (8/3,8)$ is a random countable collection of loops which do not cross themselves or each other and which locally look like SLE$_\kappa$ curves~\cite{shef-cle}. We allow our CLE loops to be nested (i.e., we do not restrict attention to the outermost loops). CLE on the whole plane was first defined in~\cite{mww-nesting} for $\kappa\in (4,8)$ and in~\cite{werner-sphere-cle} for $\kappa \in (8/3,4]$. 
\end{itemize}

\begin{figure}[ht!]
\begin{center}
    \begin{minipage}[c]{.38\textwidth}
    \includegraphics[width=\textwidth]{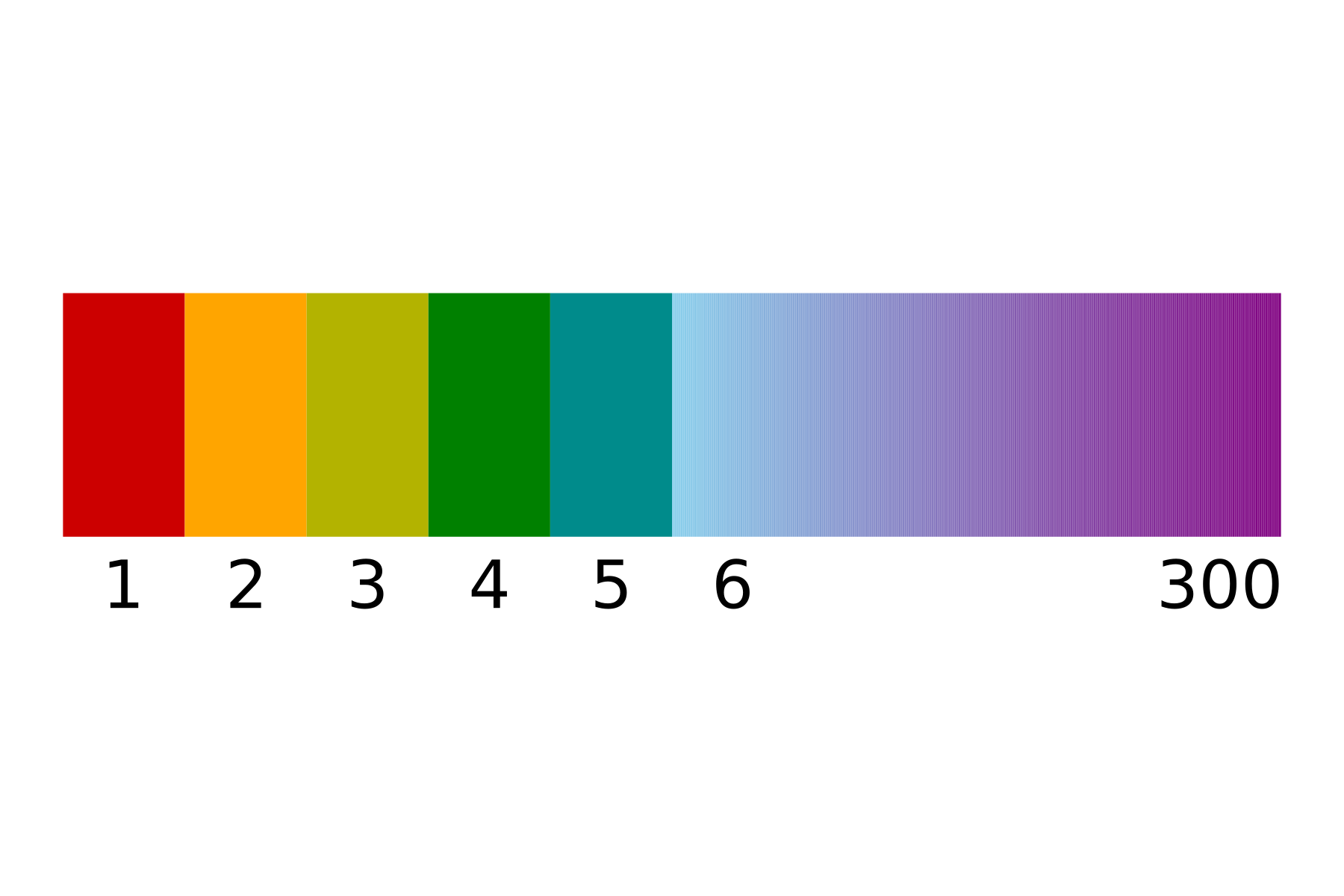}\\
    \includegraphics[width=\textwidth]{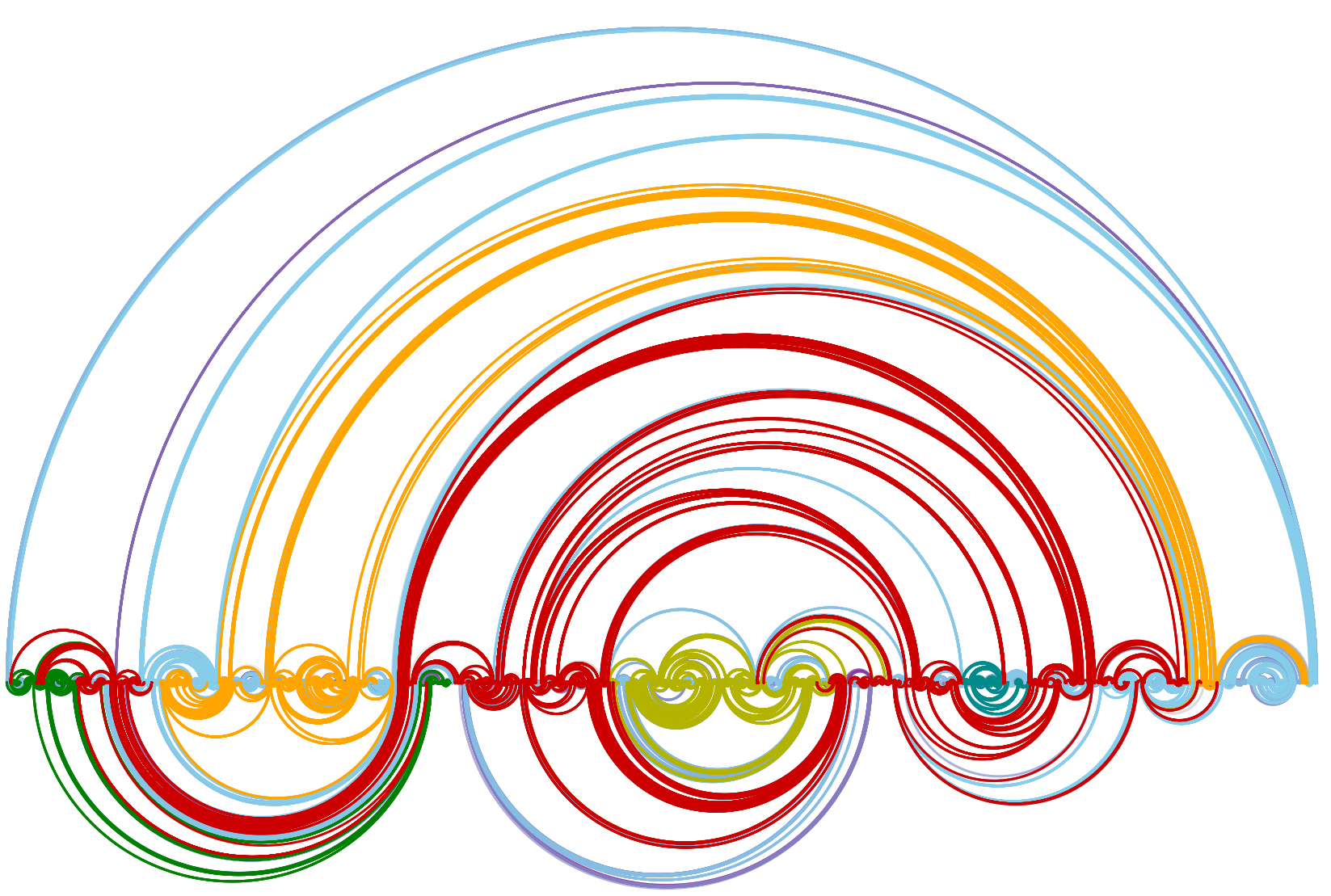}
    \end{minipage}
    \begin{minipage}[c]{.58\textwidth}
    \includegraphics[width=\textwidth]{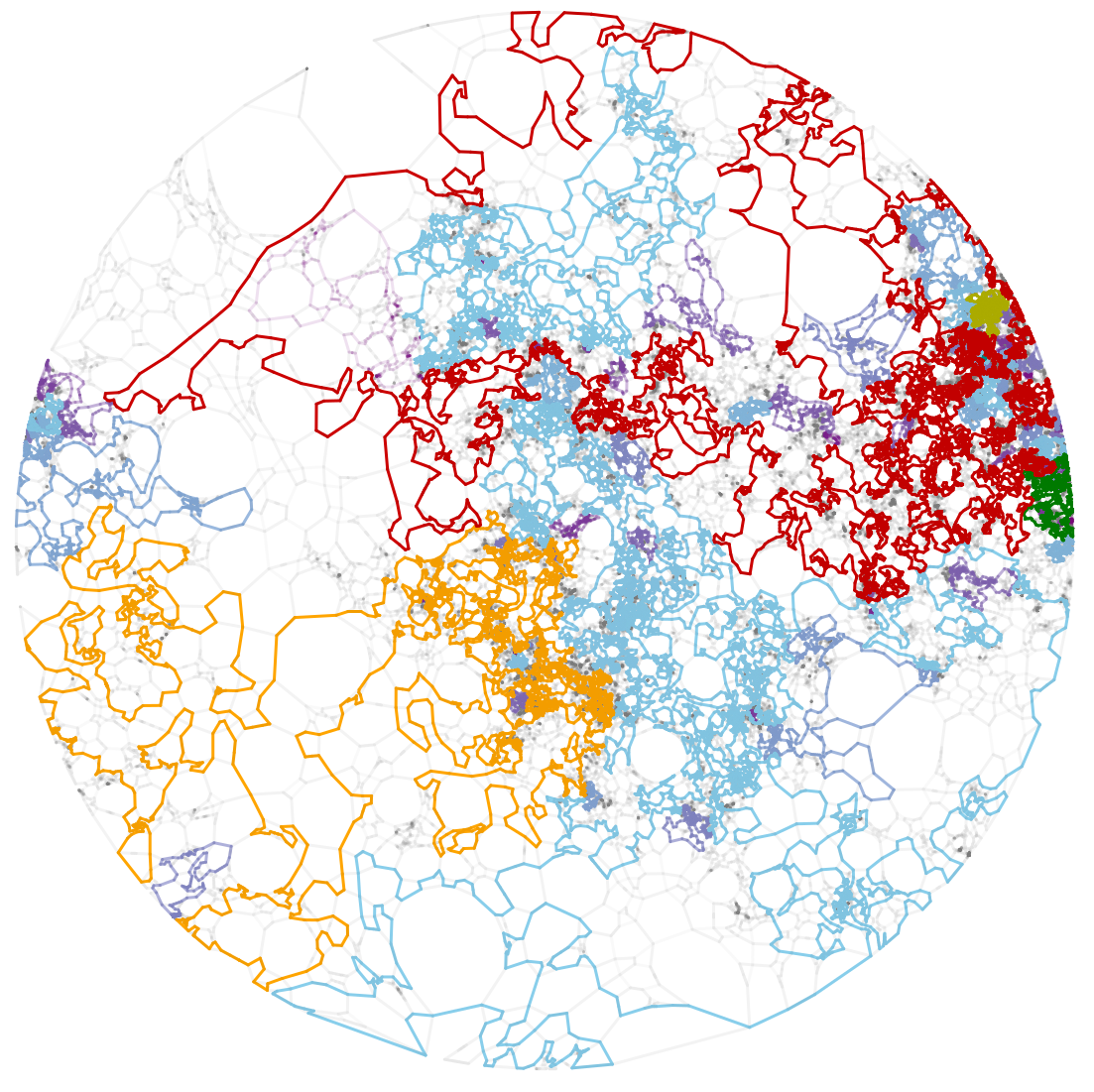}
    \end{minipage}
\caption{\label{fig-cle-sim} Simulation of a uniform meandric system with boundary of size $n = 10^6$ (see Section~\ref{sec-sim} for a precise definition and for the details of simulations). The left picture shows the corresponding arc diagrams. The right picture shows the associated planar map $\mathcal M_n$, embedded in the disk via the Tutte embedding~\cite{gms-tutte}, together with some of the loops in $\Gamma_n$.  The largest 300 loops in $\Gamma_n$ (in terms of number of vertices) are each shown in color, as indicated by the color bar. Smaller loops and edges between consecutive vertices of $\BB R$ are shown in gray. Note that the distribution of colors in the arc diagram picture is rather chaotic -- this is consistent with the fact that the meandric system loops are a complicated functional of the arc diagrams. According to Conjecture~\ref{conj-system}, the embedded planar map $\mathcal M_n$ together with the path $P_n$ and the loops in $\Gamma_n$ should converge to $\sqrt 2$-LQG decorated by SLE$_8$ and CLE$_6$.
}
\end{center}
\vspace{-3ex}
\end{figure}

\begin{conj} \label{conj-system}
Let $(\mathcal M_n,P_n ,   \Gamma_n)$ be the random planar map decorated by a Hamiltonian path and a collection of loops associated to a uniform meandric system of size $n$, as described just above. 
Then $(\mathcal M_n,P_n,\Gamma_n)$ converges under an appropriate scaling limit to an independent triple consisting of a $\sqrt 2$-LQG sphere, a whole-plane SLE$_8$ from $\infty$ to $\infty$, and a whole-plane CLE$_6$. 
\end{conj}

In the setting of Conjecture~\ref{conj-system}, the metric and measure on $\BB C$ corresponding to the $\sqrt 2$-LQG sphere, the SLE$_8$ curve (viewed modulo time parametrization), and the CLE$_6$ are independent. At first glance, this may be surprising since $(\mathcal M_n,P_n)$ and $(\mathcal M_n,\Gamma_n)$ each determine each other. However, we expect that the function which goes from $(\mathcal M_n,P_n)$ to $(\mathcal M_n,\Gamma_n)$ depends only on microscopic features of $(\mathcal M_n,P_n)$ which are not seen in the scaling limit, and the same is true for the function which goes in the opposite direction. This independence is numerically justified by Figure~\ref{fig-plots} (Right), using the discussion in Section~\ref{sec-kpz}. 

An equivalent formulation of the conjecture is that $(\mathcal M_n,P_n , \Gamma_n)$ should be in the same universality class as a uniform triple consisting of a planar map decorated by a spanning tree (represented by its associated discrete Peano curve) and a critical Bernoulli percolation configuration (represented by the loops which describe the interfaces between open and closed clusters). 

Conjecture~\ref{conj-system} is based on a combination of rigorous results, physics heuristics, and numerical simulations. We will explain the reasoning leading to the conjecture in Section~\ref{sec-justification}. A similar scaling limit conjecture for meanders, rather than meandric systems, is stated as~\cite[Conjecture 1.3]{bgs-meander} (building on~\cite{dgg-meander}). In the meander case, one has $\gamma = \sqrt{\frac13 \left( 17 - \sqrt{145} \right)}$ instead of $\gamma=\sqrt 2$ and there are two SLE$_8$ curves instead of an SLE$_8$ and a CLE$_6$. The heuristic justification for Conjecture~\ref{conj-system} is similar to the argument leading to this meander conjecture. In fact, as we will explain in Section~\ref{sec-physics}, Conjecture~\ref{conj-system} may be viewed as a special case of a conjecture for the $O(n\times m)$ loop model on a random planar map from~\cite{dgg-meander}. See also~\cite{ddgg-hamiltonian} for additional related predictions.

As we will explain in Section~\ref{sec-kpz}, Conjecture~\ref{conj-system} together with the KPZ formula~\cite{kpz-scaling} leads to the following conjectural answer to Questions~\ref{item-largest} and~\ref{item-macro} above.

\begin{conj} \label{conj-largest}
Let $\frk S_n$ be a uniform meandric system of size $n$. For each fixed $k\in\BB N$, it holds with probability tending to 1 as $n\rta\infty$ that 
\eqb \label{eqn-largest}
\# \{\text{vertices of $k$th largest loop of $\frk S_n$}\} =  n^{\alpha + o(1)} ,\quad \text{where} \quad \alpha =  \frac12 \left(3 - \sqrt 2 \right)\approx 0.7929.
\eqe
\end{conj}

Conjecture~\ref{conj-largest} for $k=1$ is consistent with the numerical study of~\cite[Section 3]{kargin-meander-system}, which suggested that the size of the largest loop in $\frk S_n$ is of order $n^\alpha$ for $\alpha$ close to $4/5$. 
We have also run some numerical simulations of our own which are consistent with Conjecture~\ref{conj-largest} for $k=1, 2, 3, 4, 5$. See Section~\ref{sec-sim} for more details.

\subsection{Macroscopic loops in finite meandric systems}
\label{sec-finite-results}

Let $d = d_{\sqrt 2}$ be the Hausdorff dimension of the $\sqrt 2$-Liouville quantum gravity metric space (this quantity is well-defined thanks to~\cite[Corollary 1.7]{gp-kpz}).
A reader not familiar with Liouville quantum gravity can simply think of $d$ as a certain constant. 
The number $d$ is not known explicitly, but fairly good rigorous upper and lower bounds are available. In particular, it was shown in~\cite[Corollary 2.5]{gp-lfpp-bounds}, building on~\cite[Theorem 1.2]{dg-lqg-dim}, that
\eqb \label{eqn-d-bound}
3.5504 \approx 2 ( 9 +3\sqrt 5 - \sqrt 3) ( 4 - \sqrt{15} ) \leq d \leq \frac23(3 + \sqrt 6) \approx 3.6330.
\eqe
The following proposition can be proven using previously known techniques for bounding distances in random planar maps~\cite{ghs-map-dist,gp-dla}. 

\begin{prop} \label{prop-map-diam}
Let $\mathcal M_n$ be the planar map associated with a uniform meandric system of size $n$. 
For each $\zeta \in (0,1)$, it holds except on an event of probability decaying faster than any negative power of $n$ that the graph-distance diameter of $\mathcal M_n$ is between $n^{1/d-\zeta}$ and $n^{1/d+\zeta}$. 
\end{prop} 

Proposition~\ref{prop-map-diam} is proven via a coupling with a so-called \textbf{mated-CRT map}, a certain type of random planar map which is directly connected to Liouville quantum gravity. See Section~\ref{sec-macro-proof} for details. It is possible to prove Proposition~\ref{prop-map-diam} via exactly the same argument as in~\cite[Theorem 1.9]{gp-dla}, which gives an analogous statement for spanning-tree decorated random planar maps. But, we will give a more self-contained proof in Section~\ref{sec-macro-proof}. 

Our first main result tells us that a uniform meandric system admits loops whose graph-distance diameter is nearly of the same order as the graph-distance diameter of $\mathcal M_n$, in the following sense, c.f.\ Proposition~\ref{prop-map-diam}. 

\begin{thm}  \label{thm-macro}
Let $\mathcal M_n$ be the planar map associated with a uniform meandric system of size $n$ and let $\Gamma_n$ be the associated collection of loops on $\mathcal M_n$.
For each $\zeta \in (0,1)$, 
it holds except on an event of probability decaying faster than any negative power of $n$ that the following is true.
There is a loop in $\Gamma_n$ which has $\mathcal M_n$-graph-distance diameter at least $n^{1/d-\zeta}$.
\end{thm}

The proof of Theorem~\ref{thm-macro} is based on a combination of two results. The first input is a purely discrete argument, based on a parity trick, which shows that the infinite-volume analog of $(\mathcal M_n,\Gamma_n)$ admits with positive probability loops which are ``macroscopic'' in a certain sense (Theorem~\ref{thm-loop-dichotomy}). 
The second input is a lower-bound for certain graph distances in $\mathcal M_n$ which is proven via a combination of discrete arguments and SLE/LQG techniques (Proposition~\ref{prop-rpm-estimate}). The continuum part of the argument, given in Section~\ref{sec-lqg}, is short and simple, but very far from elementary since it relies on both the mating of trees theorem~\cite{wedges} and the existence of the LQG metric~\cite{dddf-lfpp,gm-uniqueness}.

The reason why we have an error of order $n^\zeta$ in Proposition~\ref{prop-map-diam} and Theorem~\ref{thm-macro} is that we are only able to estimate graph distances in $\mathcal M_n$ up  to $o(1)$ errors in the exponent. If we had up-to-constants bounds for graph distances in $\mathcal M_n$, then our arguments would show that with high probability, there exist loops in $\Gamma_n$ whose $\mathcal M_n$-graph-distance diameter is comparable, up to constants, to the graph-distance diameter of $\mathcal M_n$. 
Hence, Theorem~\ref{thm-macro} suggests that the scaling limit of $(\mathcal M_n,\Gamma_n)$ should be non-degenerate, in the sense that the loops of $\Gamma_n$ do not collapse to points. This is consistent with Conjecture~\ref{conj-system}.

Theorem~\ref{thm-macro} is similar in spirit to the recent work~\cite{dgps-macroscopic-loops}, which proves the existence of macroscopic loops for the critical $O(n)$ loop model on the hexagonal lattice when $1\leq n\leq 2$ (see also~\cite{cghp-macroscopic-loops} for a similar result for a different range of parameter values). However, our proof of Theorem~\ref{thm-macro} is very different from the arguments in~\cite{dgps-macroscopic-loops,cghp-macroscopic-loops}. Part of the reason for this is that we are not aware of any positive association (FKG) inequality in our setting (see Question~\ref{question:FKG}), in addition to the fundamental difference that we are working on a random lattice.

Loops in $\Gamma_n$ are connected subsets of $\mathcal M_n$, so a loop of graph-distance diameter at least $n^{1/d-\zeta}$ must hit at least $n^{1/d-\zeta}$ vertices of $\mathcal M_n$. 
We therefore have the following corollary of Theorem~\ref{thm-macro}. 

\begin{cor}  \label{cor-largest}
Let $\frk S_n$ be a uniform meandric system of size $n$. 
For each $\zeta \in (0,1)$, it holds except on an event of probability decaying faster than any negative power of $n$ that there is a loop in $\frk S_n$ which crosses the real line at least $n^{1/d-\zeta}$ times.
\end{cor}

We note that the bounds for $d$ from~\eqref{eqn-d-bound} show that
\eqb \label{eqn-exponent-bounds}
0.2753 \approx \frac12 (3 - \sqrt 6) \leq \frac{1}{d} \leq   \frac{1}{  72 - 38 \sqrt 3 + 30 \sqrt 5 - 18 \sqrt{15} } \approx 0.2817.
\eqe 
Corollary~\ref{cor-largest} gives a power-law lower bound for the number of vertices of the largest loop in a typical meandric system. To our knowledge, the best lower bound for this quantity prior to our work is~\cite[Theorem 3.4]{kargin-meander-system}, which shows that the number of vertices in the largest loop is typically at least a constant times $\log n$. 

If Conjecture~\ref{conj-largest} is correct, then the lower bound of Corollary~\ref{cor-largest} is far from optimal. 
Nevertheless, it would require substantial new ideas to get any exponent larger than $1/d$ for the number of vertices in the largest loop of $\frk S_n$. 
Indeed, to do this one would need to show that the largest loop in $\frk S_n$ is much longer than an $\mathcal M_n$-graph distance geodesic.

\subsection{The uniform infinite meandric system (UIMS)} 
\label{sec-infinite}

The \textbf{uniform infinite meandric system (UIMS)} is the local limit (in the sense of Benjamini-Schramm~\cite{benjamini-schramm-topology}) of a uniform meandric system of size $n$ based at a uniform vertex. It is shown in~\cite[Proposition 5]{fv-noodle} that this local limit exists (in a quenched sense; see~\cite[Section 1.2]{fv-noodle} for further explanations) and is the same as the object studied in~\cite{ckst-noodle}.

\begin{figure}[ht!]
\begin{center}
\includegraphics[scale=1]{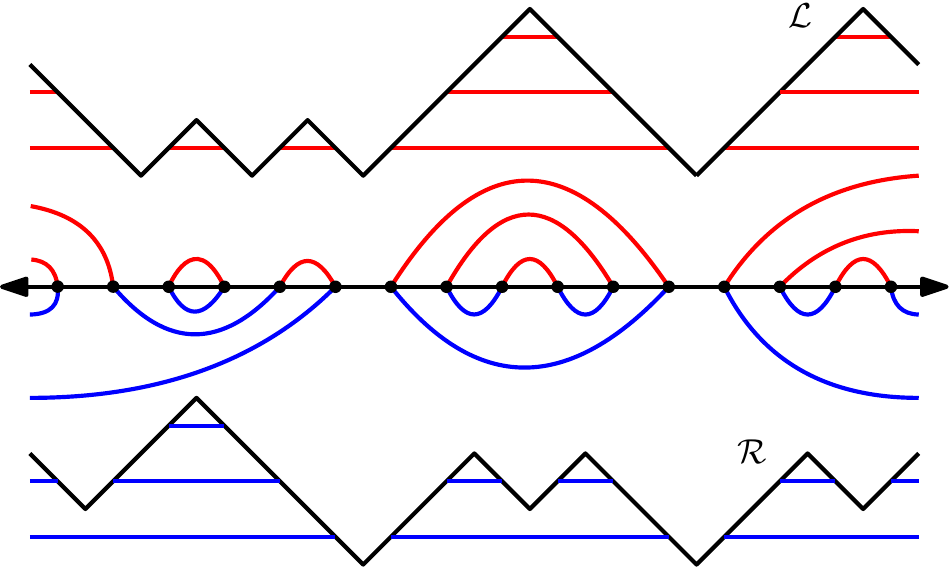}  
\caption{\label{fig-uniform-system} A subset of the uniform infinite meandric system (UIMS) together with the graphs of the corresponding increments of the encoding walks $\mcl L$ and $\mcl R$. Each horizontal line below the graph of $\mcl L$ (resp.\ $\mcl R$) corresponds to an arc above (resp.\ below) the real line. See~\eqref{eqn-rw-inf-adjacency} and the surrounding text for further explanation.
}
\end{center}
\vspace{-3ex}
\end{figure}

We now define the UIMS, following~\cite{ckst-noodle}. Let\footnote{
The reason why we write $\mcl L$ and $\mcl R$ for the walks is that $\mcl L$ (resp.\ $\mcl R$) describes the arcs which lie to the left (resp.\ right) of the real line when we traverse the real line from left to right. This notation is chosen to be consistent with the mating of trees literature~\cite{wedges,ghs-mating-survey,ghs-map-dist}.} 
$\mcl L$ and $\mcl R$ be independent two-sided simple random walks on $\BB Z$ with $\mcl L_0 = \mcl R_0 = 0$. 
See Figure~\ref{fig-uniform-system} for an illustration. 
We define two infinite arc diagrams (non-crossing perfect matchings of $\BB Z$), one above $\BB R$ and one below $\BB R$, as follows. 
For $x_1,x_2\in \BB Z$ with $x_1 <x_2$, we draw an arc above (resp.\ below) the real line joining $x_1$ and $x_2$ if and only if 
\eqb \label{eqn-rw-inf-adjacency}
 \mcl L_{x_1-1} = \mcl L_{x_2} < \min_{y \in [x_1 ,x_2-1]\cap\BB Z} \mcl L_y \qquad \text{(resp.}  \:  \mcl R_{x_1-1} = \mcl R_{x_2} < \min_{y \in [x_1 ,x_2-1]\cap\BB Z} \mcl R_y \: \text{)} ;
\eqe
c.f.~\eqref{eqn-arc-walk}.
It is easy to see that there is exactly one arc above (resp.\ below) the real line incident to each $x\in \BB Z$, and that the arcs above (resp.\ below) the real line can be taken to be non-intersecting. 
We then define the UIMS to be the set of the loops and bi-infinite paths formed by the union of the arcs in the above two arc diagrams.

It is also possible to recover the walks $\mcl L$ and $\mcl R$ from the arc diagrams. Indeed, for each $x\in\BB Z$, the increment $\mcl L_x - \mcl L_{x-1}$ is $+1$ (resp.\ $-1$) if and only if the arc above $\BB R$ which is incident to $x$ has its other endpoint greater than (resp.\ less than) $x$. A similar statement holds for $\mcl R$. 

The UIMS is easier to work with than a finite uniform meandric system since there is no conditioning on the walks. Most of our proofs will be in the setting of the UIMS. 
 
It is shown in~\cite[Theorem 1]{ckst-noodle} that either a.s.\ the collection of loops (and possibly bi-infinite paths) in the UIMS has a unique infinite path, or a.s.\ it has no infinite paths. This infinite path, if it exists, is called the \textbf{infinite noodle} in~\cite{ckst-noodle}. 
It is not known rigorously which of these two possibilities holds. But, the following is conjectured in~\cite{ckst-noodle}.

\begin{conj}[\!\cite{ckst-noodle}] \label{conj-infinite-path}
Almost surely, the UIMS has no infinite path.
\end{conj}


Exactly as in the case of a finite meandric system, we can associate to the UIMS an infinite planar map $\mathcal M$ decorated by a bi-infinite Hamiltonian path $P$ and a collection of loops (and possibly bi-infinite paths) $\Gamma$.
Namely, the vertex set of $\mathcal M$ is $\BB Z$; the edges of $\mathcal M$ are the arcs above and below the real line together with the segments $[x-1,x]$ for $x\in\BB Z$; the path $P$ traverses the vertex set $\BB Z$ in left-right numerical order; and $\Gamma$ is the set of loops (and possibly bi-infinite paths) formed by the two arc diagrams as above. 

We will now state the infinite-volume analog of Conjecture~\ref{conj-system}. 
The \textbf{$\sqrt 2$-quantum cone} is the most natural $\sqrt 2$-LQG surface with the topology of the plane. 
It arises as the local limit of the $\sqrt 2$-quantum sphere based at a point sampled from its associated area measure~\cite[Proposition 4.13(ii)]{wedges}.

\begin{conj} \label{conj-system-infty}
Let $(\mathcal M , P , \Gamma)$ be the infinite random planar map decorated by a bi-infinite Hamiltonian path and a collection of loops (and possibly bi-infinite paths) associated to the UIMS, as described just above.
Then $(\mathcal M , P , \Gamma)$ converges under an appropriate scaling limit to an independent triple consisting of a $\sqrt 2$-LQG cone, a whole-plane SLE$_8$ from $\infty$ to $\infty$, and a whole-plane CLE$_6$. 
\end{conj}

Conjecture~\ref{conj-system-infty} is consistent with Conjecture~\ref{conj-infinite-path}, since all of the loops in a whole-plane CLE$_6$ are compact sets~\cite{mww-nesting}.
Our main result concerning the UIMS gives a polynomial lower tail bound for the graph-distance diameter (and hence also the number of vertices) of the loop containing the origin. It turns out to be an easy consequence of one of the intermediate results in the proof of Theorem~\ref{thm-macro} (see, in particular, Proposition~\ref{prop-loop-dist}).

\begin{thm} \label{thm-origin-loop}
Let $(\mathcal M,\Gamma)$ be the planar map and collection of loops associated with the UIMS and let $\ell_0$ be the loop (or infinite path) in $\Gamma$ which passes through $0\in\BB Z$. Then with $d$ as in~\eqref{eqn-d-bound},
\eqb \label{eqn-origin-loop}
\BB P\left[ \left(\text{$\mathcal M$-graph-distance diameter of $\ell_0$} \right) \geq k \right] \geq k^{-(d-1) - o(1)} , \quad \text{as $k\rta\infty$}.
\eqe
\end{thm}

Unlike in Theorem~\ref{thm-macro}, we expect that the exponent in Theorem~\ref{thm-origin-loop} is not optimal. Rather, we have the following conjecture. 

\begin{conj} \label{conj-origin-loop}
Let $\ell_0$ be as in Theorem~\ref{thm-origin-loop}. Then with $\alpha = \frac12 (3-\sqrt 2)$ as in Conjecture~\ref{conj-largest}, 
\eqb
\BB P\left[ \left(\text{$\mathcal M$-graph-distance diameter of $\ell_0$} \right) \geq k \right] \geq k^{-d(1-\alpha) - o(1)} , \quad \text{as $k\rta\infty$},
\eqe
and
\eqb
\BB P\left[ \left(\text{number of vertices in $\ell_0$} \right) \geq k \right] \geq k^{-(1/\alpha - 1) - o(1)} , \quad \text{as $k\rta\infty$}.
\eqe
\end{conj}
 
If Conjecture~\ref{conj-largest} were true, then, at least heuristically, the same arguments as in the proof of Theorem~\ref{thm-origin-loop} would lead to a proof of Conjecture~\ref{conj-origin-loop}. 

\subsection{The uniform infinite half-plane meandric system (UIHPMS)} 
\label{sec-half-plane}

The \textbf{uniform infinite half-plane meandric system (UIHPMS)} is a natural variant of the UIMS which is half-plane like in the sense that it has a bi-infinite sequence of distinguished ``boundary vertices'' (points in $\BB Z$ which are not disconnected from $\infty$ below the real line by any path in the UIHPMS), but ``most'' vertices are not boundary vertices. We will now give two equivalent definitions of this object, see Figure~\ref{fig-cutting}.

\medskip

\noindent \emph{By cutting:}
Start with a sample of the UIMS on $\mathbb Z$, represented by a pair of arc diagrams as in Section~\ref{sec-infinite}. 
We assume that the arcs of the two arc diagrams are drawn in $\mathbb R^2$ in such a way that the arcs do not cross each other and each arc with endpoints $x <y$ is contained in the vertical strip $[x,y] \times \BB R$ (the arc diagrams in all of the figures in this paper have this property).
We cut all of the arcs below the real line which intersect the vertical ray $\{1/2\}\times (-\infty,0]$, leaving two unmatched ends corresponding to each arc.
Then, we rewire successive pairs of unmatched ends to form new arcs. That is, we enumerate the unmatched ends from left to right as $\{e_j\}_{j\in\BB Z}$, with positive indices corresponding to ends to the right of $\{1/2\}\times (-\infty,0]$ and non-positive indices corresponding to ends to the left. Then, we link up $e_{2j-1}$ and $e_{2j }$ for each  $ j\in\BB Z$. See Figure~\ref{fig-cutting} (Left). The resulting infinite collection of the loops (and possibly bi-infinite paths) is the UIHPMS.

A point $x\in \BB Z$ is called a \textbf{boundary point} if it can be connected to the ray $\{1/2\}\times (-\infty, 0]$ by some continuous path without crossing any arc or the real line. Each endpoint of arcs which were cut is a boundary point, but there are also other boundary points. See Figure~\ref{fig-cutting} (Left) for an illustration.
It is clear from the random walk description of the meandric system~\eqref{eqn-rw-inf-adjacency} that each time $x\in\BB Z$ at which the walk $\mcl R$ attains a running infimum when run forward or backward from time 0 gives rise to a boundary point of the UIHPMS (but not every boundary point arises in this way).
Hence, there are infinitely many boundary points. 
Denote by $J_k  \in \BB Z $ (resp.\ $J_{-k+1} \in \BB Z$) the $k$th positive (resp.\ non-positive) boundary point.
Note that in the UIHPMS we always have that $J_0=0$ and $J_1=1$ and that $J_0$ and $J_1$ are not joined by an arc.

\medskip
\noindent \emph{By random walks:} We also have a random walk description for the UIHPMS similar to that of the UIMS. The only difference is that we use a \emph{reflected} simple random walk instead of an ordinary simple random walk for the arc diagram below the real line.
Let $\mcl L$ and $\mcl R$ be independent two-sided simple random walks on $\mathbb Z$ with $\mcl L_0 = \mcl R_0 = 0$.
For $x_1,x_2\in \BB Z$ with $x_1 <x_2$, we draw an arc above (resp.\ below) the real line joining $x_1$ and $x_2$ if and only if 
\eqb \label{eqn-rw-inf-adjacency-half}
 \mcl L_{x_1-1} = \mcl L_{x_2} < \min_{y \in [x_1 ,x_2-1]\cap\BB Z} \mcl L_y \quad \text{(resp.}  \:  |\mcl R|_{x_1-1} = |\mcl R|_{x_2} < \min_{y \in [x_1 ,x_2-1]\cap\BB Z} |\mcl R|_y \: \text{)} ;
\eqe
c.f.~\eqref{eqn-rw-inf-adjacency}.

A point $x\in \BB Z$ is called a \textbf{boundary point} if $|\mcl R|$ crosses $1/2$ in between time $x-1$ and $x$, i.e., $\mcl R_{x-1}, \mcl R_x \in \{-1,0,1\}$.
As before, denote by $J_k \in \BB Z$ (resp.\ $J_{-k+1} \in \BB Z$) the $k$th positive (resp.\ non-positive) boundary point. See Figure~\ref{fig-cutting} (Right) for an illustration. Since $\mcl R_0=0$, we always have $J_0=0$, $J_1=1$, and $J_0$ not matched with $J_1$ as in the previous definition. 

\medskip

\begin{figure}[ht!]
    \begin{center}
    \begin{minipage}[c]{.53\textwidth}
    \includegraphics[width=\textwidth]{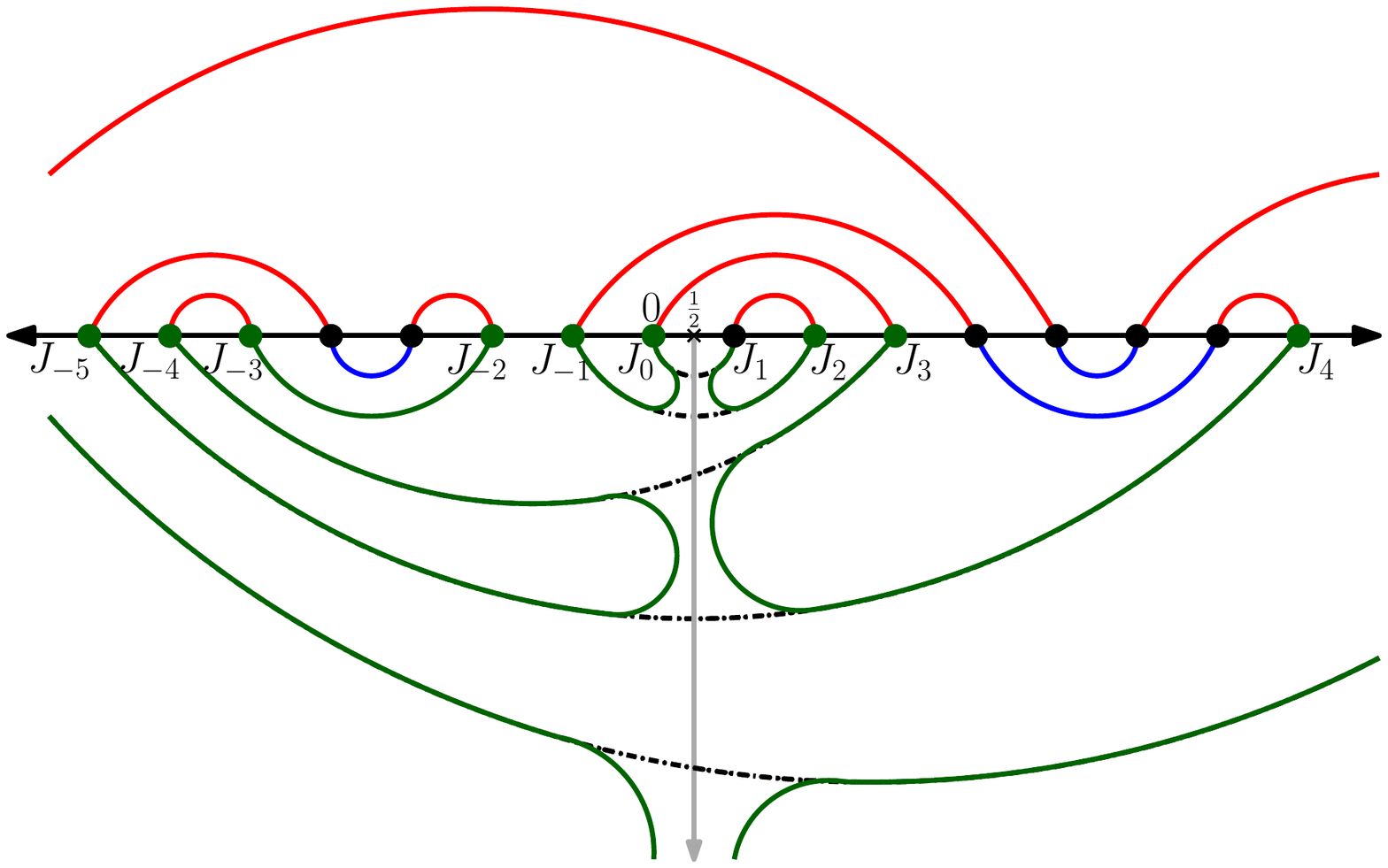}
    \end{minipage}
    \hspace{.02\textwidth}
    \begin{minipage}[c]{.43\textwidth}
    \includegraphics[width=\textwidth]{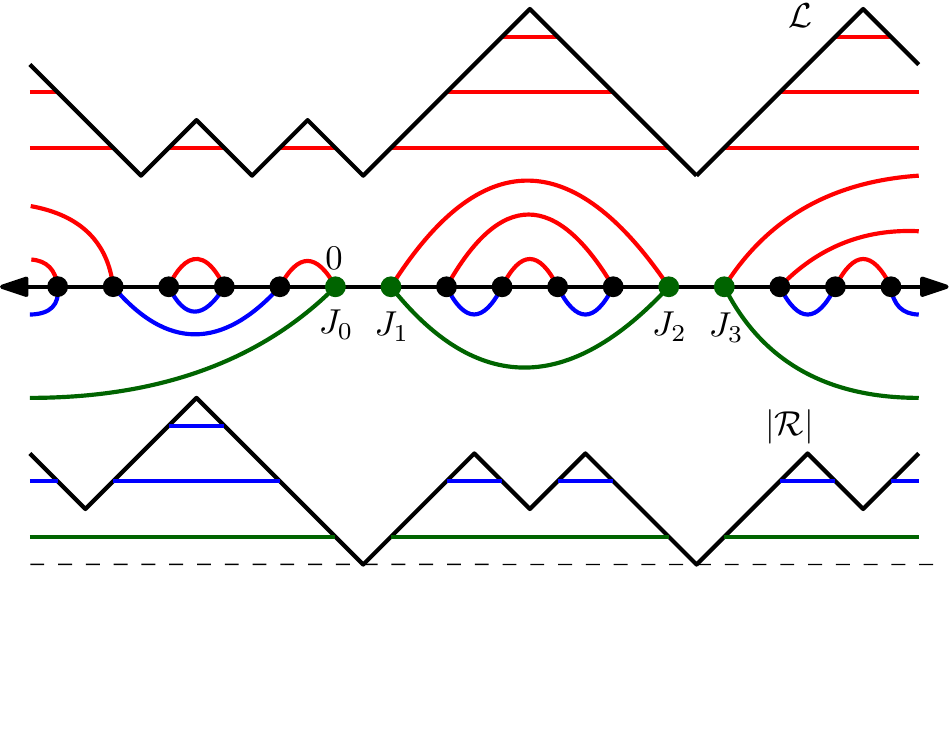}
    \end{minipage}
    \caption{\label{fig-cutting} \textbf{Left:} The UIHPMS constructed by cutting.
    The dash-dotted arcs straddling $0$ below the real line in the UIMS (defined in Section~\ref{sec-infinite}) are cut by the gray ray $\{1/2\}\times (-\infty,0]$, and these arcs are rewired successively to result in the UIHPMS.
    Boundary points are labeled as $\{J_k\}_{k\in\BB Z}$ and are colored green together with their connecting arcs. Note that the points $J_{-2}$ and $J_{-3}$ are boundary points but are not endpoints of any cut arcs. \textbf{Right:} The UIHPMS constructed by random walks. We show a subset of the UIHPMS together with the graphs of the corresponding increments of the encoding walks $\mcl L$ and $|\mcl R|$. Each horizontal line below the graph of $\mcl L$ (resp.\ |$\mcl R$|) corresponds to an arc above (resp.\ below) the real line. The boundary points and arcs incident to them are colored green.
    }
    \end{center}
    \vspace{-3ex}
\end{figure}

The following lemma will be proven in Section~\ref{sec-uihpms-equiv} using a discrete version of L\'evy's theorem~\cite{simons1983discrete} which relates (roughly speaking) the zero set and the running minimum times of a random walk.
 
\begin{lem} \label{lem-uihpms-equiv}
The law of the UIHPMS is the same under the above two definitions.
\end{lem}

Just as in the case of other types of meandric systems, we can view the UIHPMS as a planar map $ \mathcal M'$ decorated by a collection of loops (plus possibly bi-infinite paths) $\Gamma'$ and a Hamiltonian path $P'$. 
The law of the UIHPMS is invariant under even translations along the boundary. That is, if $\{J_k\}_{k\in\BB Z}$ is the set of boundary points as above, then
\eqb \label{eqn-bdy-translate}
(\mathcal M' - J_{2k}, P' - J_{2k} , \Gamma'  -J_{2k}) \eqD (\mathcal M' , P', \Gamma') ,\quad \forall k \in \BB Z .
\eqe
This is clear from the random walk description. 

As in Conjecture~\ref{conj-infinite-path}, we can also ask if there exists an infinite path in the UIHPMS. The answer should be no if the scaling limit conjecture is true.
Unlike in the whole-plane setting, we are able to prove this in the half-plane setting.

\begin{thm} \label{thm-infinite-path-half}
    Almost surely, the set $\Gamma'$ for the UIHPMS has no infinite path.
\end{thm}

Theorem~\ref{thm-infinite-path-half} will be proven in Section~\ref{sec-half-plane-proof} via a purely discrete argument. See the beginning of Section~\ref{sec-half-plane-proof} for an outline.
The proof is completely independent of the proofs of our theorems for finite meandric systems and for the UIMS (stated in Sections~\ref{sec-finite-results} and~\ref{sec-infinite}). So, the proofs can be read in any order.

In Conjecture~\ref{conj-system-infty}, we discussed how the UIMS is conjecturally related to the $\sqrt{2}$-quantum cone.
Another natural $\sqrt 2$-LQG surface -- with the half-plane topology instead of whole-plane topology -- is the \textbf{$\sqrt{2}$-quantum wedge}. See~\cite[Section 1.4]{wedges} or~\cite[Section 3.4]{ghs-mating-survey} for a rigorous definition.  
In the continuum setting, cutting a $\sqrt{2}$-quantum cone by an independent whole-plane SLE$_2$ from 0 to $\infty$ yields the $\sqrt{2}$-quantum wedge~\cite[Theorem 1.5]{wedges}. This is a continuum analog of the above cutting description of the UIHPMS. That is, in the scaling limit, the gray ray in Figure~\ref{fig-cutting} should converge to some simple fractal curve from 0 to $\infty$. Cutting the previous loop-decorated whole-plane along this curve yields the half-plane (or whole-plane with a slit) topology. The points located immediately on the left or right side of the curve become boundary points.
Thus, Conjecture~\ref{conj-system-infty} has a natural half-plane version as follows.

\begin{conj} \label{conj-system-infty-half}
    Let $(\mathcal M', P', \Gamma')$ be the infinite random planar map decorated by a bi-infinite Hamiltonian path and a collection of loops associated to the UIHPMS, as described just above.
    Then $(\mathcal M', P', \Gamma')$ converges under an appropriate scaling limit to an independent triple consisting of a $\sqrt 2$-LQG wedge, a space-filling SLE$_8$ from $\infty$ to $\infty$ in the half-plane, and a CLE$_6$ in the half-plane. 
\end{conj}

In the setting of Conjecture~\ref{conj-system-infty-half}, the boundary vertices of $\mathcal M'$ should correspond to boundary points for the $\sqrt 2$-LQG wedge. For example, if $(\mathcal M', P', \Gamma')$ is embedded into the half-plane appropriately (e.g., via some version of Tutte embedding or circle packing), in such a way that the boundary vertices of $\mcl M'$ are mapped to points on the real line, then one should have the convergence of the embedded objects toward the $\sqrt 2$-LQG wedge together with SLE$_8$ and CLE$_6$.

\bigskip

Starting from a CLE$_6$ in the half-plane, one can construct a chordal SLE$_6$ curve from 0 to $\infty$ by concatenating certain arcs of boundary-touching CLE$_6$ loops (see the proof of~\cite[Theorem 5.4]{shef-cle}). 
The analogous path can be also constructed in the UIHPMS by leaving exactly one boundary point unmatched as follows. 

We define the \textbf{pointed infinite half-plane meandric system (PIHPMS)} by rewiring lower arcs between non-positive boundary points in the UIHPMS as follows.
Start with the UIHPMS and recall that $J_0=0$ is not matched with $J_1=1$. Remove the arcs below the real line which join $J_{2k-1}$ and $J_{2k}$ for each $k \leq 0$.
Then, add new arcs joining $J_{2k-2}$ and $J_{2k-1}$ for each $k\leq 0$, in such a way that the new arcs do not cross any other arcs. Note that now $J_{0}=0$ is unmatched. This new configuration is defined to be the PIHPMS and we think of it as pointed at $J_{0}=0$.
One can also describe the PIHPMS in terms of random walks by replacing $|\mcl R
|_{\cdot}$ in~\eqref{eqn-rw-inf-adjacency-half} with the modified walk
\eqb\label{eqn-sle-rw}
\mcl R^\circ_{\cdot}:=|\mcl R|_{\cdot} - \textbf{1}_{\{\cdot\ge 0\}}.
\eqe
As there is no $x<0$ such that $\mcl R^\circ_{x-1} = \mcl R^\circ_{0} = -1$, we have a special boundary point $J_0=0$ which is not incident to any arc below the real line as desired. See Figure~\ref{fig-sle-sim} (Top Left) for an illustration. 

As per usual, we view the PIHPMS as a planar map $ \mathcal M^\circ$ decorated by a Hamiltonian path $ P^\circ$ and a collection of loops (plus possibly infinite paths) $\Gamma^\circ$. Note that the path started from $J_0 = 0$ must be an infinite path because no arc below the real line is incident to the origin, so the path cannot come back to its starting point to form a loop.
We prove that this is the unique infinite path.

\begin{prop} \label{prop-alternating}
Consider the PIHPMS as defined above. Almost surely, there is a unique infinite path $\gamma^\circ \in \Gamma^\circ$. This path starts at $J_0 = 0$ and hits infinitely many points in each of $\{J_k\}_{k < 0}$ and $\{J_k\}_{k > 0}$.
\end{prop}

The proof of Proposition~\ref{prop-alternating} is given in Section~\ref{sec-half-plane-proof}, and uses many of the same ideas as in the proof of Theorem~\ref{thm-infinite-path-half}. Another conjecture follows accordingly.

\begin{figure}[ht!]
    \begin{center}
    \begin{minipage}[c]{.38\textwidth}
    \includegraphics[width=\textwidth]{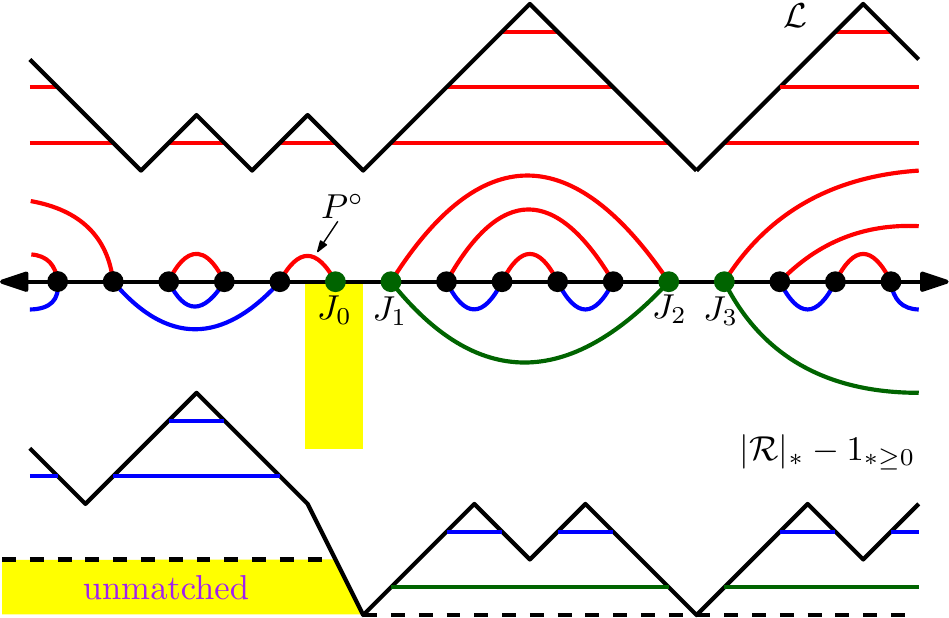}\\
    \includegraphics[width=\textwidth]{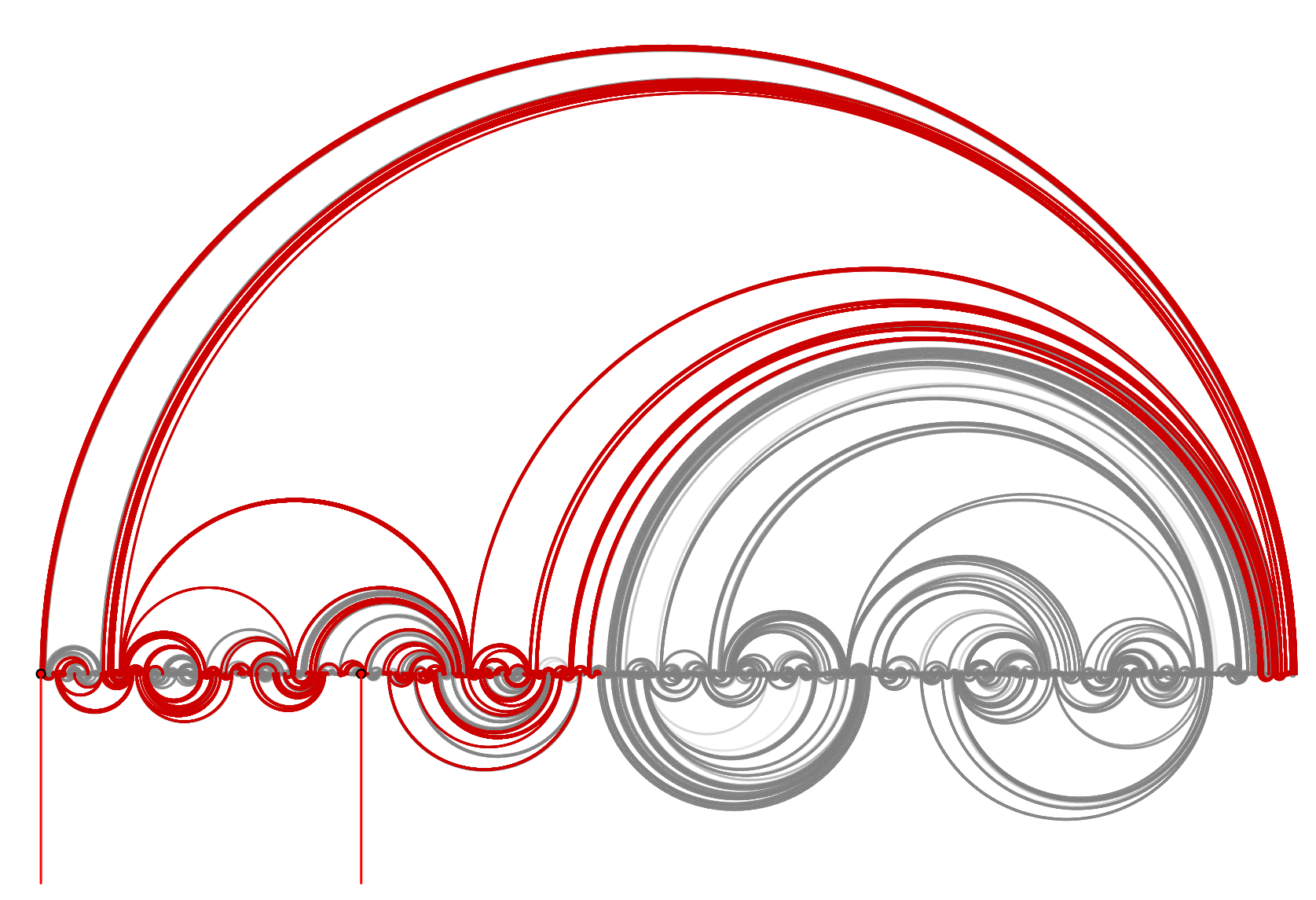}
    \end{minipage}
    \begin{minipage}[c]{.58\textwidth}
    \includegraphics[width=\textwidth]{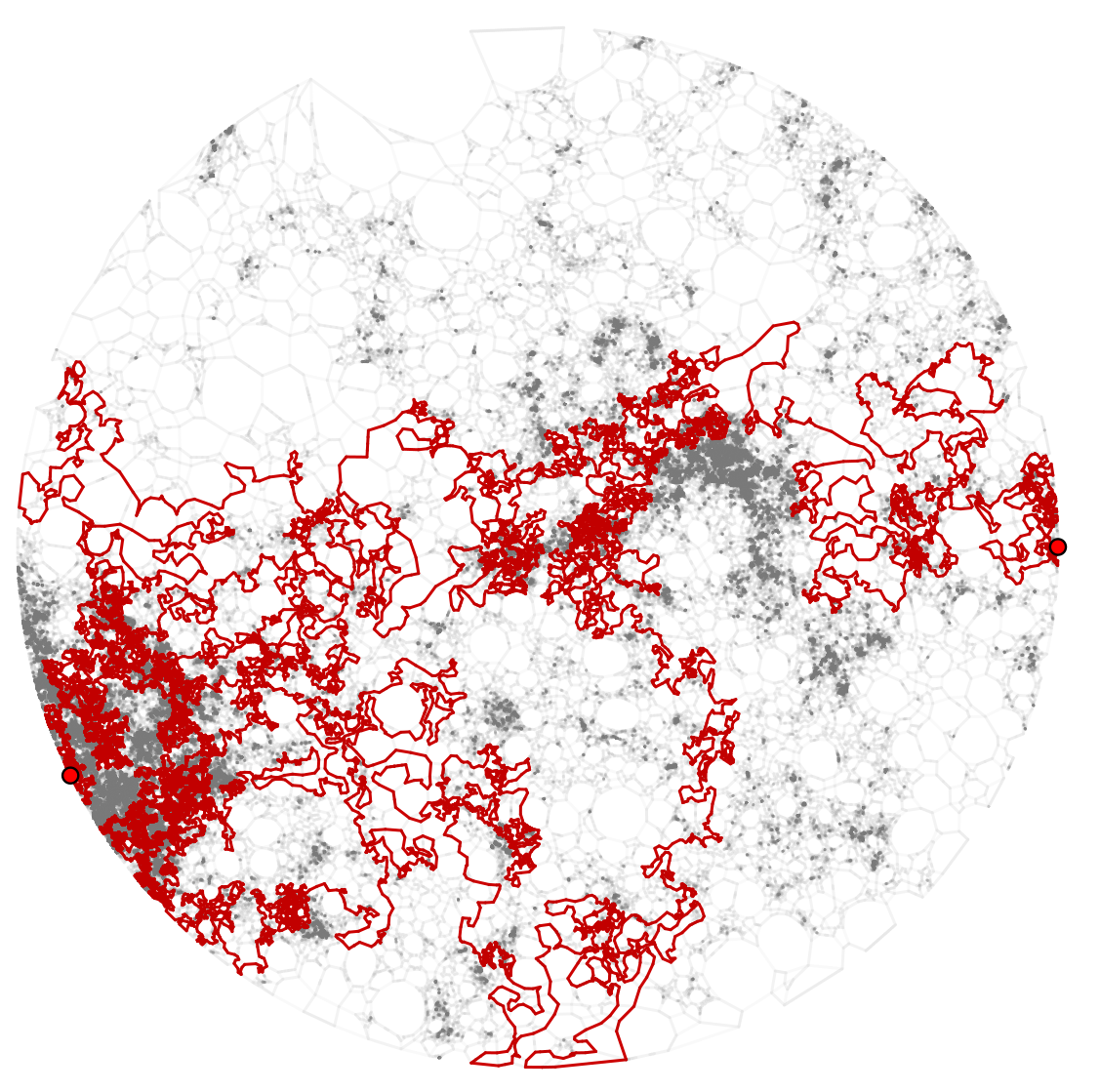}
    \end{minipage}
    \caption{\label{fig-sle-sim} \textbf{Top left:} The random walks encoding a sample of the PIHPMS obtained from the previous example in Figure~\ref{fig-cutting} (Right). The yellow strip in the PIHPMS highlights the special boundary point $J_0=0$ which is not incident to any arc below the real line as a consequence of the fact that there is no $x<0$ such that $\mcl R^\circ_{x-1} = \mcl R^\circ_{0} = -1$; see the yellow strip in the walk. \textbf{Bottom left/right:} A simulation of a finite-volume version (i.e., with the topology of the disk) of the PIHPMS with two marked points, of size $n = 9\cdot 10^6$. The bottom left picture shows the arc diagrams and the right picture shows the associated planar map, drawn in the disk via the Tutte embedding~\cite{gms-tutte}. The red curve is conjectured to converge to a chordal SLE$_6$ between the two marked boundary points. See Section~\ref{sec-sim} for the details of simulations.
    }
    \end{center}
    \vspace{-3ex}
\end{figure}

\begin{conj} \label{conj-sle}
    Let $(\mcl M^\circ, P^\circ , \Gamma^\circ )$ be the infinite random planar map decorated by a bi-infinite Hamiltonian path and a collection of loops (plus possibly infinite paths) associated to the PIHPMS, as described just above.
    Also, let $\gamma^\circ$ be the path started from 0 in $\Gamma^\circ$ (which is a.s. the unique infinite path by Proposition~\ref{prop-alternating}).
    Then $(\mcl M^\circ, P^\circ,  \gamma^\circ , \Gamma^\circ \setminus \gamma^\circ)$ converges under an appropriate scaling limit to a $\sqrt 2$-LQG wedge, a space-filling SLE$_8$ from $\infty$ to $\infty$ in the half-plane, a chordal SLE$_6$ from $0$ to $\infty$ in the half-plane, and a CLE$_6$ in the complement of the SLE$_6$ curve (i.e., a union of conditionally independent CLE$_6$s in the complementary connected components).
\end{conj}

\begin{remark} \label{remark-correlated}
One can also consider finite or infinite meandric systems constructed from a pair of random walks which are correlated rather than independent, i.e., the step distribution of the pair of walks assigns different weights to steps in $\{ (-1,-1) , (1,1)\}$ and steps in $\{(-1,1) ,(1,-1)\}$. We do not have a simple combinatorial description for the random meandric systems obtained in this way, so we do not emphasize them in this paper. Based on mating of trees theory~\cite{wedges,ghs-mating-survey}, it is natural to conjecture that the associated decorated planar map converges to $\gamma$-LQG decorated by a space-filling SLE$_{16/\gamma^2}$ and an independent CLE$_6$, where $\gamma \in (0,2)$ is chosen so that the correlation of the encoding walks is $-\cos(4\pi/\gamma^2)$. We have run some numerical simulations (similar to Section~\ref{sec-sim}) which are consistent with this.
Our proofs of Theorems~\ref{thm-macro} and~\ref{thm-origin-loop} work verbatim in the case of correlated walks, with $d$ replaced by the dimension of $\gamma$-LQG. However, the upper bound for the diameter of $\mcl M_n$ in Proposition~\ref{prop-map-diam} uses estimates from~\cite{gp-dla} which are only proven for $\gamma=\sqrt 2$. Most of our proofs for the UIHPMS do not work in the case of correlated walks since we cannot apply the discrete L\'evy theorem~\eqref{eqn-discrete-levy} to one coordinate of the walk independently from the other.
\end{remark}

\section{Percolation interpretation}
\label{sec-perc}

The main goal of this section is to explain how meandric systems can be viewed as a model of (critical) percolation on planar maps. Throughout the section we assume that the reader is familiar with the basic terminology and results in percolation theory. This section is included only for intuition and motivation: it is not needed for the proofs of our main results.

\subsection{Non-crossing perfect matchings and non-crossing integer partitions}

We start by recalling a classical bijection between non-crossing perfect matchings and non-crossing integer partitions, see for instance~\cite[Section 3]{gnp-meander-system}.

We write a partition $\pi$ of a finite set $A\subset \BB R$
as $\pi = \{V_1, \dots, V_k\}$, where $V_1, \dots, V_k$ are called the blocks of $\pi$ and are non-empty, pairwise disjoint sets,
with $\cup_i V_i=A$.  
We say that a partition $\pi$ is non-crossing when it is not possible to find two
distinct blocks $V, W \in \pi$ and numbers $a < b < c < d$ in $A$ such that $a, c \in V$ and
$b, d \in W$. Equivalently, $\pi$ is non-crossing if it can be plotted as a planar graph with the vertices in $A$
arranged on the real line, so that the blocks of $\pi$ are the connected components
of the planar graph drawn in the upper-half plane; see the orange planar graph in Figure~\ref{fig-bij_ms_ncp} for an example.

\begin{figure}[ht!]
	\begin{center}
		\includegraphics[scale=0.6]{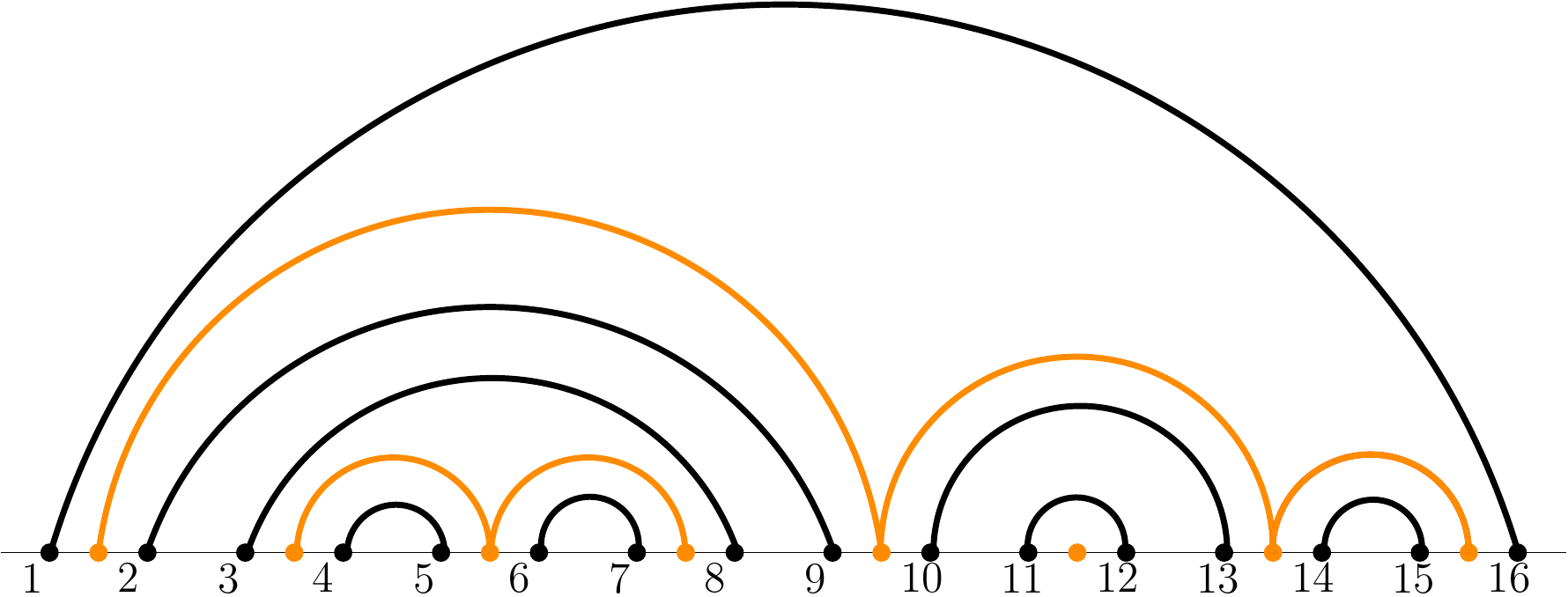}  
		\caption{\label{fig-bij_ms_ncp}  We plot in orange the planar graph associated with the non-crossing partition  $\pi=\{\{2-1/2,10-1/2,14-1/2,16-1/2\},\{4-1/2,6-1/2,8-1/2\},\{12-1/2\}\}$ of the points $2\BB Z \cap [1,16]-1/2$ and in black the corresponding planar graph of the non-crossing perfect matching $\psi(\pi)$ of the points $[1,16] \cap \BB Z$. Note that the black arcs are chosen so that they ``block'' new potential orange edges.
		}
	\end{center}
	\vspace{-3ex}
\end{figure}

We can now describe the aforementioned bijection. 
Given a non-crossing partition $\pi=\{V_1, \dots, V_k\}$ of the points $2\BB Z \cap [1,2n]-1/2$, we consider the unique non-crossing perfect matching $\psi(\pi)$ of $[1,2n] \cap \BB Z$ such that for every pair of points $z<w\in 2\BB Z \cap [1,2n]-1/2$ not contained in the same block $V_i$, there exists a pair of matched points
$x,y\in [1,2n] \cap \BB Z$ in  $\psi(\pi)$ such that either $x<z<y<w$ or $z<x<w<y$. See Figure~\ref{fig-bij_ms_ncp} for a graphical interpretation.

Conversely, given a non-crossing perfect matching $m$ of $[1,2n] \cap \BB Z$, we construct a non-crossing partition $\psi^{-1}(m)$ of the points $2\BB Z \cap [1,2n]-1/2$ as follows: for every pair of points $z,w\in 2\BB Z \cap [1,2n]-1/2$ we say that $z\sim_m w$ if there is no pair of matched points $x,y$ in $m$ with either $x<z<y<w$ or $z<x<w<y$. We then set the equivalent classes of $(2\BB Z \cap [1,2n]-1/2,\sim_m)$ to be the blocks of $\psi^{-1}(m)$.

Note that the fact that $\psi$ and $\psi^{-1}$ are inverse maps is immediate.

\subsection{Meandric systems as boundaries of clusters of open edges}\label{sect-inf-meand-perco}

Given a non-crossing perfect matching $m$ of $[1,2n] \cap \BB Z$, one can also consider (in the previous description of the inverse map $\psi^{-1}$) the points $2\BB Z \cap [0,2n]+1/2$ (see the green points in Figure~\ref{fig-meand-as-perco}). Then, as before, one determines a second non-crossing integer partition $\widetilde\psi^{-1}(m)$ of the points $2\BB Z \cap [0,2n]+1/2$ (see for instance the green planar graph on the upper-half plane in Figure~\ref{fig-meand-as-perco}).

We further notice that any element of the triple $(m,\psi^{-1}(m),\widetilde\psi^{-1}(m))$ determines the  two other elements of the triple. 

\begin{figure}[ht!]
	\begin{center}
		\includegraphics[scale=0.5]{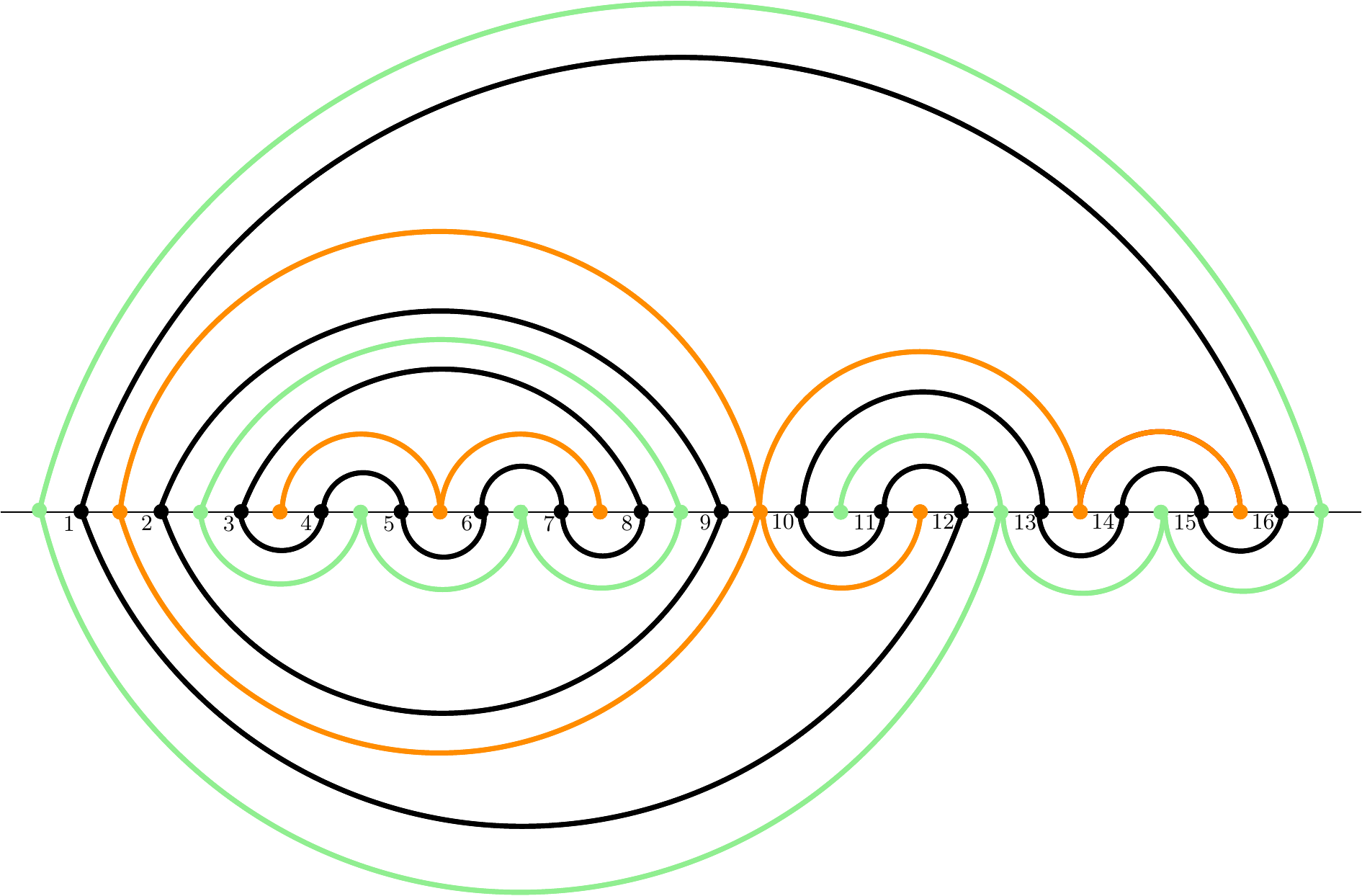}  
		\caption{\label{fig-meand-as-perco} A meandric system $m=(m_1,m_2)$ plotted in black together with its two pairs of non-crossing integer partitions $(\psi^{-1}(m_1),\psi^{-1}(m_2))$ and $(\widetilde\psi^{-1}(m_1),\widetilde\psi^{-1}(m_2))$ in orange and green, respectively.
		}
	\end{center}
	\vspace{-3ex}
\end{figure}

Consider now a meandric system of size $n$, which is a pair $(m_1,m_2)$ of non-crossing perfect matchings of $[1,2n] \cap \BB Z$; the first above and the second below the real line. Using the bijections $\psi$ and $\widetilde\psi$ described above, one can associate to the meandric system $(m_1,m_2)$ two pairs of non-crossing integer partitions of $2\BB Z \cap [1,2n]-1/2$ and $2\BB Z \cap [0,2n]+1/2$: the first pair (resp.\ the second pair) is obtained applying the map $\psi^{-1}$ (resp.\ the map $\widetilde\psi^{-1}$) to $m_1$ and $m_2$, obtaining the orange (resp.\ the green) planar graph in Figure~\ref{fig-meand-as-perco}. From now on we will refer to these two pairs of non-crossing integer partitions as the orange and the green planar graphs (associated to a meandric system).

\medskip

We now present a new way of thinking about a meandric system of size $n$ closer to percolation models (in what follows we explain what are the classical quantities in percolation models in our setting; c.f. Figure~\ref{fig-meand-as-perco}).

\medskip

\noindent \emph{Main lattice:} This is the planar map whose vertices are the union of the green and orange vertices, and whose edges are the union of the green and orange edges and the edges corresponding to the real line (between consecutive vertices of different colors).

\medskip

\noindent \emph{Dual lattice:} This is the planar map that we previously associated with a meandric system, i.e. the map whose vertices are the black vertices, and whose edges are the black edges and the edges corresponding to the real line (between consecutive black vertices).

\medskip

\noindent \emph{Open edges in the main lattice:} These are the orange edges.

\medskip

\noindent \emph{Closed edges in the main lattice:} These are the green edges.

\medskip

\noindent \emph{Boundary of clusters of open edges:} These are the black loops in the meandric system.

\medskip

Note that if one considers a uniform meandric system (or equivalently a uniform pair of non-crossing integer partitions), then one obtains (through the interpretation above) a model of percolation on random planar maps. Analogously, if one considers the UIMS or the UIHPMS (Sections~\ref{sec-infinite} and~\ref{sec-half-plane}), then one obtains (through the obvious adaptations of the constructions above\footnote{We do not give the details of these constructions since they are not needed to continue our discussion.}) some natural model of percolation on infinite random planar maps.

One of the main difficulties in the study of these models is that the randomness for the environment (i.e.\ the planar map) and the randomness for the percolation are strongly coupled. See also the beginning of Section~\ref{sect-no-fkg} for a second major difficulty.

\subsection{Meandric systems and critical percolation on random planar maps}

Our interpretation of meandric systems as a percolation model described in the previous section, gives also some natural new interpretations of our results. In particular, Theorem~\ref{thm-macro}, which states that a uniform meandric system admits loops whose graph distance diameter is nearly of the same order as the graph distance diameter of the associated planar map, implies that our percolation model admits macroscopic clusters. Hence, it does not behave like subcritical Bernoulli percolation on a fixed lattice.

Due to Corollary~\ref{cor-path-dichotomy} below, a.s.\ the UIMS has a bi-infinite path if and only if our above percolation model has an infinite open cluster. Theorem~\ref{thm-infinite-path-half} implies that there is no infinite cluster for the percolation model on the UIHPMS. 
Due to the construction of the UIHPMS by cutting the UIMS along an infinite ray (Section~\ref{sec-half-plane}), this is roughly analogous to the statement that there is no infinite cluster for, say, Bernoulli percolation on $\BB Z^2 \setminus (\{0\}\times \BB Z_+)$. 
In particular, our percolation model does not behave like supercritical Bernoulli percolation on a fixed lattice. 

By combining the preceding two paragraphs, we see that our percolation model is in some sense ``critical'' (see also Proposition~\ref{prop-crossings} for further evidence). 
Determining whether there is a bi-infinite path in the UIMS (Conjecture~\ref{conj-infinite-path}) is therefore analogous to determining whether there is percolation at criticality.
 
We note that the interpretation of meandric systems as a critical percolation model is also consistent with Conjectures~\ref{conj-system}, \ref{conj-system-infty} and \ref{conj-system-infty-half}, which state that the scaling limit of the loops of a meandric system should be the conformal loop ensemble with parameter $\kappa=6$. Indeed, the latter is also conjectured to be the scaling limit of the cluster interfaces for critical percolation models on various deterministic discrete lattices. This conjecture has been proved in the case of critical site percolation on the triangular lattice \cite{smirnov-cardy
	,camia-newman-sle6
	,camia-newman-cle6
}. 
 
\subsection{Box crossings}\label{sect-box-cross}

A classical result for critical Bernoulli ($p=1/2$) bond percolation on $\mathbb{Z}^2$ is that given a box $B_n = \BB Z^2 \cap \{[0,n]\times[0,n+1]\}$, then the probability that there exists a path of open edges connecting the top side of $B_n$ to the bottom side of $B_n$ is always equal to $1/2$, independently of the size $n$ of the box.

We prove the analogous result in the context of the UIMS. We start by defining a natural notion of box for the UIMS. 
Given the UIMS together with its orange and green planar graphs as in the left-hand side of Figure~\ref{fig-inf-meand-box}, we define the \textbf{box} of size $n$ rooted at $x\in\mathbb Z+\frac 1 2$ to be the collection of (black/orange/green) vertices in $[x,x+2n-1]$ and (black/orange/green) edges with at least one endpoint in $(x,x+2n-1)$. We denote such box by $B_n(x)$. See Figure~\ref{fig-inf-meand-box} for an example. Note that with this definition any box of size $n$ contains the same number of green and orange vertices.

\begin{figure}[ht!]
	\begin{center}
		\includegraphics[scale=0.6]{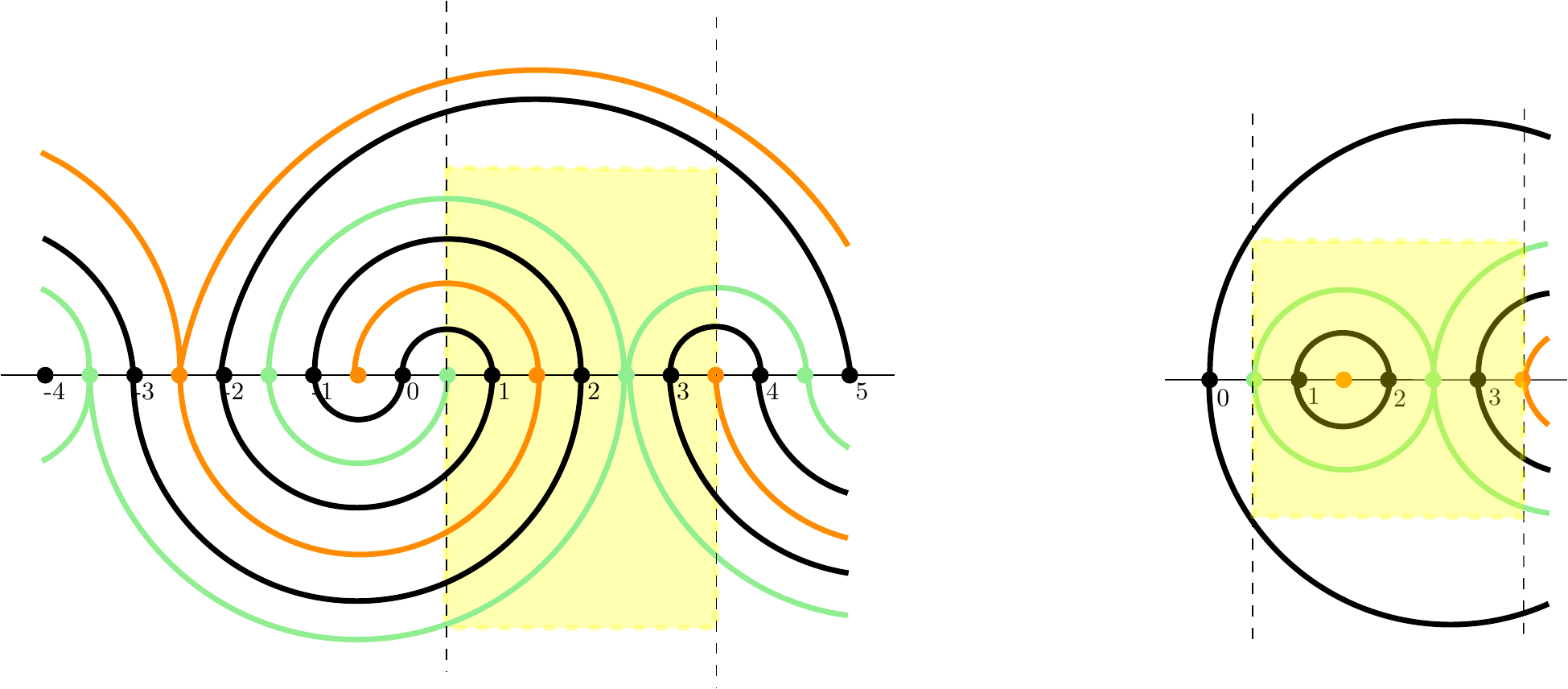}  
		\caption{\label{fig-inf-meand-box} 
		\textbf{Left:} 
		A portion of the UIMS together with its orange and green planar graphs. A box of size $2$ rooted at $1/2$ is highlighted in yellow. The bottom-green edge between $-3.5$ and $2.5$ and the top-green edge between $2.5$ and $4.5$ form a bottom-to-top green crossing of the box $B_2(1/2)$. Note that $B_2(1/2)$ contains also a top-to-bottom green crossing of the box.
		\textbf{Right:} An example of another box. Note that both the top and bottom green edges from $0.5$ to $2.5$ crosses both the top-left boundary and the bottom-left boundary of the box $B_2(1/2)$, according to our definition below.
		}
	\end{center}
	\vspace{-3ex}
\end{figure}

Given a box $B_n(x)$, we say that an edge crosses the top-left (resp.\ bottom-left, top-right, bottom-right) boundary of the box, if that edge touches the vertical line $\{x\}\times[0,\infty)$ (resp.\ $\{x\}\times (-\infty,0]$, $\{x+2n-1\}\times[0,\infty)$, $\{x+2n-1\}\times(-\infty,0]$), where we highlight that the point $(x,0)$ (resp.\ $(x+2n-1,0)$) is included in the line. See Figure \ref{fig-inf-meand-box} for an example. We then say that there exists a \textbf{top-to-bottom orange crossing} (resp.\ bottom-to-top orange crossing) of $B_n(x)$ if there exists a path of all orange edges such that the first edge in the path crosses the top-left (resp.\ bottom-left) boundary of $B_n(x)$, the last edge in the path crosses the bottom-right (resp.\ top-right) boundary of $B_n(x)$, and all the other edges have both extremities in $B_n(x)$. We similarly define top-to-bottom and bottom-to-top green crossings.

We can now state our analogous result for ``the probability that there exists a top-to-bottom crossing of open edges in a box is 1/2''. 

\begin{prop}\label{prop-crossings}
	Consider the UIMS together with its green and orange planar graphs. For all $x\in\mathbb Z+\frac 1 2$ and all $n>0$, the probability that there exists a top-to-bottom orange crossing of the box $B_n(x)$ is 1/2.
\end{prop}

\begin{proof}
	Fix $x\in\mathbb Z+\frac 1 2$ and $n>0$. 
	Note that by the construction of the orange and green planar graphs (recall from Section~\ref{sect-inf-meand-perco} how the orange arcs determine the green arcs), exactly one of the following two events holds: either 
	\[O=\{\text{There exists a top-to-bottom orange crossing of the box } B_n(x)\}\] 
	or 
	\[G=\{\text{There exists a bottom-to-top green crossing of the box } B_n(x)\}.\]
	Moreover, since $B_n(x)$ contains the same number of green and orange vertices and the events $O$ and $G$ are symmetric, then they must must have the same probability. Hence $\mathbb P(O)=\mathbb P(G)=1/2$.
\end{proof}

\subsection{The lack of positive association and two open problems}\label{sect-no-fkg}

Proposition~\ref{prop-crossings} is a further evidence that meandric systems behave like a critical percolation model on a planar map and also suggests that it might be possible to prove analogs of other standard results in the theory of percolation models. Unfortunately, as already mentioned in Section \ref{sec-finite-results}, we are not aware of any positive association (FKG) inequality in our setting and for the moment all the notions of monotonicity that we tried do not satisfy FKG. This discussion naturally leads to the following open problem.

\begin{ques}\label{question:FKG}
	Is there a natural notion of monotonicity in infinite meandric configurations such that some type of positive association (FKG) inequality holds (for instance) for our notion of top-to-bottom orange crossings of proper boxes?
\end{ques}

Another interesting question, which might be relevant to Conjecture~\ref{conj-system} (or one of its variants), is to compute the following loop crossing probability.

Recall the definition of the box $B_{n}(x)$ given at the beginning of Section~\ref{sect-box-cross}. We define a \textbf{top-to-bottom loop crossing} of $B_{n}(x)$ in the same way as we defined a top-to-bottom orange crossing of $B_{n}(x)$, where in the definition orange edges are replaced by black edges.

\begin{ques}\label{question:crossing}
Consider the UIMS. What is the asymptotic probability as $n\to\infty$ that there exists a top-to-bottom loop crossing of $B_{n}(0)$? 
\end{ques}

Conjecture~\ref{conj-system-infty} gives a natural candidate for the answer to the question above, which we now explain (assuming a certain familiarity with SLEs).
Building on Conjecture~\ref{conj-system-infty}, the points in $B_{n}(0)$ are expected to converge under an appropriate scaling limit to the points visited by the whole-plane SLE$_8$ $\eta$ between time $0$ and $1$.\footnote{Here we assume that $\eta$ is parametrized by $\sqrt 2$-LQG area with respect to an independent unit area quantum cone, so that $\eta : \BB R\rta \BB R$.}

The set $\eta((-\infty,0]) \cap \eta([0,\infty))$ is the union of two disjoint SLE$_2$-type curves~\cite[Footnote 4]{wedges}, one called the left-boundary of $\eta((-\infty,0])$ and the other one called the right-boundary of $\eta((-\infty,0])$. The same holds for $\eta((-\infty,1]) \cap \eta([1,\infty))$. Moreover, the left-boundaries (resp.\ right-boundaries) of $\eta((-\infty,0])$ and  $\eta((-\infty,1])$ a.s.\ merge into each other. As a consequence, the set $\eta([0,1])$ forms a topological rectangle. We refer to the piece of the left boundary of  $\eta((-\infty,1])$ not in common with the left boundary of $\eta((-\infty,0])$ (resp.\ the piece of the right boundary of  $\eta((-\infty,0])$ not in common with the left boundary of $\eta((-\infty,1])$)  as the top-left boundary (resp.\ bottom-right boundary) of $\eta([0,1])$.


We expect that the probability in Question~\ref{question:crossing} converges to the probability that there exists a continuous portion of a loop in the whole-plane CLE$_6$ of  Conjecture~\ref{conj-system-infty} crossing $\eta([0,1])$ from its top-left boundary to its bottom-right boundary, without leaving $\eta([0,1])$.

We highlight that the latter crossing probability is not known explicitly, and that computing it is an interesting problem in its own right.

\section{Existence of macroscopic loops in the UIMS}
\label{sec-loop-dichotomy}

Throughout this section, we let $(\mathcal M,\Gamma)$ be the infinite planar map (with vertex set $\BB Z$) and collection of loops (and possibly bi-infinite paths) associated with the UIMS, as defined in Section~\ref{sec-infinite}.

\begin{defn} \label{def-disconnect}
Let $A,B\subset \BB R^2$ and let $\ell$ be a loop or a path $\BB R^2 \cup \{\infty\}$. We say that $\ell$ \textbf{disconnects} $A$ from $B$ if every path from $A$ to $B$ hits $\ell$.
\end{defn}

The loops (and possibly bi-infinite paths) of the UIMS can be viewed as loops in $\BB R^2$ which hit $\BB R$ only at integer points and which are defined modulo orientation-preserving homeomorphisms from $\BB R^2$ to $\BB R^2$ which fix $\BB R$. 
If $A,B\subset \BB R \cup \{\infty\}$, such a homeomorphism does not alter whether a loop disconnects $A$ from $B$. Hence it makes sense to talk about loops in $\Gamma$ disconnecting subsets of $\BB R\cup \{\infty\}$.

The goal of this section is to prove the following theorem, which can roughly speaking be thought of as saying that $\Gamma$ has a positive chance to admit macroscopic loops or paths at all scales. This theorem will eventually be combined with estimates for distances in the infinite planar map $\mcl M$ (which we prove in Sections~\ref{sec-mated-crt} and~\ref{sec-lqg-mcrt}) to prove Theorems~\ref{thm-macro} and~\ref{thm-origin-loop}, see Section~\ref{sec-macro-proof}.

\begin{thm} \label{thm-loop-dichotomy}
For each sufficiently large $n\in\BB N$, it holds for each $C>1$ that with probability at least $ 5-2\sqrt 6 \approx 0.1010$, at least one of the following two conditions is satisfied:
\begin{enumerate}[A.] 
\item \label{item-di-around} There is a loop in $\Gamma$ which disconnects $[-n,n]$ from $\infty$ and which hits a point of $[n , (2C+1)n + 3]\cap\BB Z$. 
\item \label{item-di-across} There is a loop or an infinite path in  $\Gamma$ which hits a point in each of $[-n,n]\cap\BB Z$ and $\BB Z\setminus [-C n , C n] $. 
\end{enumerate}
\end{thm}

When we apply Theorem~\ref{thm-loop-dichotomy}, we will take $C$ large but fixed independently of $n$. For such a choice of $C$, a loop satisfying either of the two conditions of Theorem~\ref{thm-loop-dichotomy} should be thought of as being ``macroscopic'', in the sense that it should give rise to a non-trivial loop when we send $n\rta\infty$ and pass to an appropriate scaling limit (c.f.\ Conjecture~\ref{conj-system-infty}). 


Theorem~\ref{thm-loop-dichotomy} implies the following independently interesting corollary. We note that a similar statement for the $O(n)$ loop model on the hexagonal lattice, for a certain range of parameter values, is proven in~\cite[Theorem 1]{cghp-macroscopic-loops}.

\begin{cor} \label{cor-path-dichotomy}
Exactly one of the following two conditions occurs with probability one, and the other occurs with probability zero:
\begin{enumerate}[A$'$.] 
\item There is an infinite path in $\Gamma$. \label{item-infinite-path}
\item For each $x \in \BB Z$, there are infinitely many loops in $\Gamma$ which disconnect $x$ from $\infty$. \label{item-disconnect}
\end{enumerate}
\end{cor}

It is possible to prove Corollary~\ref{cor-path-dichotomy} via a more direct argument which does not use Theorem~\ref{thm-loop-dichotomy}, but we will deduce it from Theorem~\ref{thm-loop-dichotomy} for convenience. 

\begin{proof}[Proof of Corollary~\ref{cor-path-dichotomy}]
By the zero-one law for translation invariant events, the event that $\Gamma$ has an infinite path has probability zero or one (see also~\cite[Theorem 1]{ckst-noodle}). Loops and infinite paths in $\Gamma$ cannot cross each other, so for any $x\in\BB Z$ which is hit by an infinite path in $\Gamma$, there is no loop in $\Gamma$ which disconnects $x$ from $\infty$.
Therefore, to prove the corollary it suffices to assume that $\Gamma$ a.s.\ has no infinite paths and show that this implies that for each $x\in\BB Z$, a.s.\ there are infinitely many loops in $\Gamma$ which disconnect $x$ from $\infty$. The definition~\eqref{eqn-rw-inf-adjacency} of $(\mathcal M,\Gamma)$ implies that translating by a fixed $x\in\BB Z$ preserves the law of $(\mathcal M,\Gamma)$. So, we can restrict attention to the case $x=0$. 

Let $n\in\BB N$ be large enough so that the conclusion of Theorem~\ref{thm-loop-dichotomy} is satisfied. Since we are assuming that $\Gamma$ has no infinite paths, as $C\rta\infty$ ($n$ fixed) the probability that condition~\ref{item-di-across} of Theorem~\ref{thm-loop-dichotomy} tends to zero. Hence, the theorem implies that with probability at least $5-2\sqrt 6$, there is a loop in $\Gamma$ which disconnects $[-n,n]$ from $\infty$. Let $G_n$ be the event that this is the case, so that $\BB P[G_n] \geq 5-2\sqrt 6$ for each large enough $n\in\BB N$. Then
\eqbn
\BB P\left[ \bigcap_{m=1}^\infty \bigcup_{n=m}^\infty G_n \right] \geq 5-2\sqrt 6 , 
\eqen
i.e., with probability at least $5-2\sqrt 6$, there are infinitely many loops in $\Gamma$ which disconnect 0 from $\infty$. By the zero-one law for translation invariant events, this in fact holds with probability one.
\end{proof}

We now give an overview of the proof of Theorem~\ref{thm-loop-dichotomy}. 
If $\Gamma$ admits an infinite path, it is straightforward to check that condition~\ref{item-di-across} in the proposition statement holds with high probability when $n$ is large. So, we can assume without loss of generality that there is no infinite path. 

The key idea of the proof is that a meandric system satisfies rather rigid parity properties. 
In particular, any distinct $x,y\in \BB Z$ such that $x-1/2$ and $y-1/2$ are not separated by a loop or an infinite path have to have the same parity (Lemma~\ref{lem-parity}). 
Under the assumption that there is no infinite path in $\Gamma$, this allows us to force the existence of a macroscopic loop as follows. 

Fix $C>1$ and let $E_n$ be the event that there is a loop in $\Gamma$ which hits only vertices of $[-C n , C n]\cap\BB Z$ and which disconnects $[-n,n]$ from $\infty$ (see Figure~\ref{fig-around-event}). If $\BB P[E_n]$ is bounded below by an $n$-independent constant, then condition~\ref{item-di-around} in the theorem statement holds with uniformly positive probability. So, we can assume that $\BB P[E_n]$ is small, i.e., $\BB P[E_n^c]$ is large. 

The event $E_n$ depends only on the restriction of the encoding walk $(\mcl L , \mcl R)$ to $[-C n , C n] \cap\BB Z$ (Lemma~\ref{lem-disconnect-msrble}). 
Therefore, if $x \in \BB Z$ with $x > 2 C n$, then the probability that $E_n^c$ occurs, and also $E_n^c$ occurs with the translated map $\mathcal M-x$ in place of $\mathcal M$, is $\BB P[E_n^c]^2$. If we choose $x$ to be odd, then, using the definition of $E_n$, we can show that with probability at least $\BB P[E_n^c]^2/8$, there are pairs of points $(y,y_x)$ with $y \in [-n,n]$ and $y_x \in [x-n,x+n]$ which have opposite parity and which are not disconnected from $\infty$ by loops whose vertex sets are contained in $[-Cn ,C n]$ and $[x-Cn,x+Cn]$, respectively. 
Since $y$ and $y_x$ have opposite parity, there has to be a loop which disconnects $y$ from $y_x$ (Lemma~\ref{lem-loop-exists} and Figure~\ref{fig-loop-exists}).
It is then straightforward to check that this loop has to satisfy one of the two conditions of Theorem~\ref{thm-loop-dichotomy}. 

The probability $5-2\sqrt 6$ in the proposition statement comes from considering the ``worst case'' possibility for $\BB P[E_n]$. 
 
We now commence with the proof, starting with the requisite parity lemma. 

\begin{figure}[ht!]
\begin{center}
\includegraphics[width=0.75\textwidth]{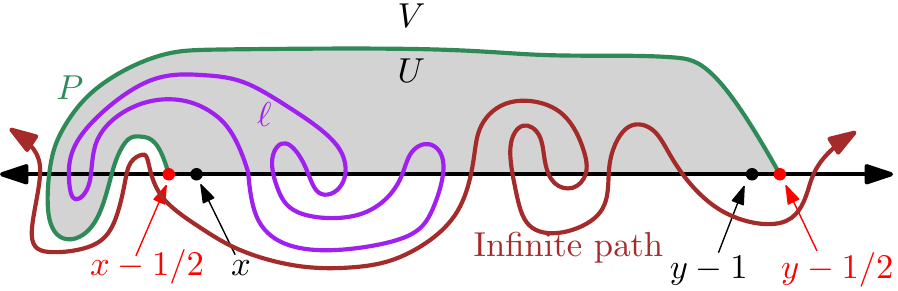}  
\caption{\label{fig-parity} Illustration of the proof of Lemma~\ref{lem-parity}. We prove the following statement: suppose $x-1/2$ and $y-1/2$ are not disconnected by any loop or infinite path in $\Gamma$, and each point $z-1/2$ with $z \in [x+1,y-1]\cap\BB Z$ is disconnected from $x$ and $y$ by a loop or an infinite path in $\Gamma$. 
Then any loop or infinite path in $\Gamma$ must hit $[x , y-1]\cap\BB Z$ an even number of times (examples of a loop and an infinite path are shown in purple and brown, respectively). 
This implies that $y-x$ is even, which gives the contrapositive of the lemma statement.
}
\end{center}
\vspace{-3ex}
\end{figure}

\begin{lem} \label{lem-parity}
Let $x \in \BB Z$ be even and let $y \in\BB Z$ be odd. 
There is a loop or an infinite path in $\Gamma$ which disconnects $x-1/2$ from $y-1/2$. 
\end{lem}
\begin{proof}
See Figure~\ref{fig-parity} for an illustration. 
Define an equivalence relation on $\BB Z$ by $x \sim y$ if there is no loop or infinite path in $\Gamma$ which disconnects $x-1/2$ from $y-1/2$. 
Let $X$ be any equivalence class. 
If suffices to show that every element of $X$ has the same parity. 
By considering the elements of $X$ in left-right numerical order, it suffices to show the following. 
If $x$ and $y$ are two consecutive elements of $X$ (i.e., $x,y\in X$, $x<y$, and there is no element of $X$ in $[x,y]$) then $x$ and $y$ have the same parity. 

Since $x\sim y$, there exists a path $P$ in $\BB R^2$ from $x-1/2$ to $y-1/2$ which does not hit any loop or infinite path in $\Gamma$. 
By erasing loops made by $P$, we can take $P$ to be a simple path. Since loops in $\Gamma$ do not intersect $\BB R$ except at integer points, we can also arrange that $P$ hits $(x-1/2,x)$ and $(y-1/2,y)$ only at their endpoints.

Since $x$ and $y$ are consecutive elements of $X$, the path $P$ does not hit $[x,y-1]$ (otherwise, there would be an element of $X$ between $x$ and $y$). 
Therefore, $[x-1/2,y-1/2]\cup P$ is a simple closed loop in $\BB R^2$.
By the Jordan curve theorem, there are exactly two connected components of $\BB R^2 \setminus ([x-1/2,y-1/2] \cup P)$ whose common boundary is $[x-1/2,y-1/2]\cup P$. 
Let $U$ and $V$ be these two connected components.  
  
Consider a loop $\ell \in \Gamma$ which hits a point of $[x ,y-1]\cap\BB Z$. 
We traverse $\ell$ counterclockwise, say, started from a point of $\ell \setminus \BB R$. 
Since $\ell$ cannot intersect $\BB R$ without crossing it (by the definition of a meandric system) and $\ell$ cannot hit $P$ (by our choice of $P$), the number of times that $\ell$ intersects $[x,y-1]\cap\BB Z$ is equal to the number of times that $\ell$ crosses from $U$ to $V$ or from $V$ to $U$. 
Since $\ell$ starts and ends at the same point, the number of times that $\ell$ crosses from $U$ to $V$ is equal to the number of times that $\ell$ crosses from $V$ to $U$.
Hence the number of times that $\ell$ intersects $[x,y-1]\cap\BB Z$ is even.

Similarly, if $\Gamma$ has an infinite path, then the number of times that this infinite path intersects $[x,y-1]\cap\BB Z$ is even. 
Since every point of $[x,y-1]\cap\BB Z$ is hit by either a loop or an infinite path of $\Gamma$, we get that $[x,y-1]\cap\BB Z$ is even. 
Hence, $y-x$ is even. 
\end{proof}

\begin{figure}[ht!]
\begin{center}
\includegraphics[width=0.75\textwidth]{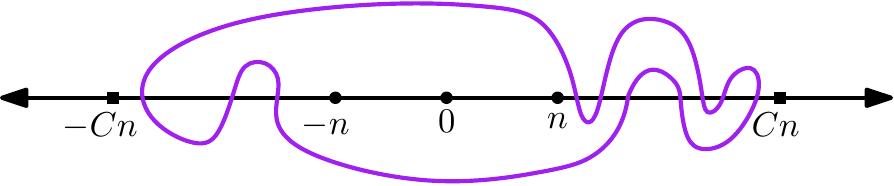}  
\caption{\label{fig-around-event} Illustration of the event $E_n$, which depends on the constant $C>1$. 
}
\end{center}
\vspace{-3ex}
\end{figure}

Fix $C  > 1$ and for $n\in\BB N$, let $E_n$ be the event that the following is true:
\begin{enumerate} \label{eqn-around-event} 
\item[$(\dagger)$] There exists a loop in $\Gamma$ which hits only vertices in $[-C n , C n ]\cap\BB Z$ and which disconnects $[-n,n]$ from $\infty$.
\end{enumerate}
\newcommand{\refE}{{(\hyperref[eqn-around-event]{$\dagger$})}}
Note that $E_n$ depends on $C$. See Figure~\ref{fig-around-event} for an illustration.

\begin{lem} \label{lem-disconnect-msrble}
The event $E_n$ is determined by the encoding walk increment $(\mcl L  , \mcl R)|_{[-C n , C n ] \cap\BB Z}$.
\end{lem}
\begin{proof}
The event $E_n$ depends only on the arcs of the upper and lower arc diagrams for $\mathcal M$ which join points of $[-C n , C n]\cap\BB Z$. 
These arcs are determined by $(\mcl L  , \mcl R)|_{[-C n , C n ] \cap\BB Z}$ by the relationship between the arc diagrams and the walks~\eqref{eqn-rw-inf-adjacency}. 
\end{proof}

It is clear that the conclusion of Theorem~\ref{thm-loop-dichotomy} is satisfied if $\BB P[E_n] \geq 5-2\sqrt 6$. 
So, we need to show that the conclusion of the theorem is also true if $\BB P[E_n^c] \geq 1 - (5-2\sqrt 6)$.
The following elementary topological lemma, in conjunction with Lemma~\ref{lem-parity}, will help us do so. 
 
\begin{lem} \label{lem-disconnect-exposed}
If $E_n^c$ occurs, then there exists $y \in [-n  + 1,n]\cap\BB Z$ such that $y-1/2$ is not disconnected from $\infty$ by any loop in $\Gamma$ which hits only vertices in $[-C n , C n ]\cap\BB Z$.
\end{lem}
\begin{proof}
Let $\Gamma_{C n}$ be the set of loops in $\Gamma$ which hit only vertices in $[-C n , C n] \cap\BB Z$.  
Then $\Gamma_{C n} $ is a finite collection of simple loops in $\BB R^2$ which do not intersect each other.  
Let $\Gamma_{Cn}'$ be the set of outermost loops in $\Gamma_{C n}$ (i.e., those which are not disconnected from $\infty$ by any other loop in $\Gamma_{C n}$).  
For $\ell \in \Gamma_{Cn}'$, let $U_\ell$ be the open region disconnected from $\infty$ by $\ell$.
Then the closures $\ol U_\ell$ of the sets $U_\ell$ for $\ell \in \Gamma_{C n}'$ are disjoint (since the loops $\ell \in \Gamma_{Cn}'$ are disjoint and non-nested) and their union is the same as the set of points which are disconnected from $\infty$ by the union of the loops in $\Gamma_{C n}$.

By the definition~\refE\ of $E_n$, if $E_n^c$ occurs then $[-n,n]$ is not contained in $\ol U_\ell$ for any $\ell \in \Gamma_{C n}'$. 
Since $[-n,n]$ is connected and the sets $\ol U_\ell$ for $\ell \in \Gamma_{C n}'$ are closed and disjoint, it follows that $[-n,n]$ is not contained in the union of the sets $\ol U_\ell$ for $\ell \in \Gamma_{C n}'$.
Hence, there must be a point $z\in [-n,n]$ which is not contained in $\ol U_\ell$ for any $\ell \in \Gamma_{Cn}'$. 
The set of such $z$ is an open subset of $[-n,n]$, so we can take $z \in [-n,n]\setminus \BB Z$.  
The point $z$ is not disconnected from $\infty$ by any loop in $\Gamma_{C n}$. 
Since loops in $\Gamma$ hit $\BB R$ only at integer points, if $y \in [-n+1,n]\cap\BB Z$ is chosen so that $z\in (y-1,y)$, then also $y-1/2$ is not disconnected from $\infty$ by any loop in $\Gamma_{C n}$.
\end{proof}

The following lemma is the main step in the proof of Theorem~\ref{thm-loop-dichotomy}.

\begin{figure}[ht!]
\begin{center}
\includegraphics[width=0.75\textwidth]{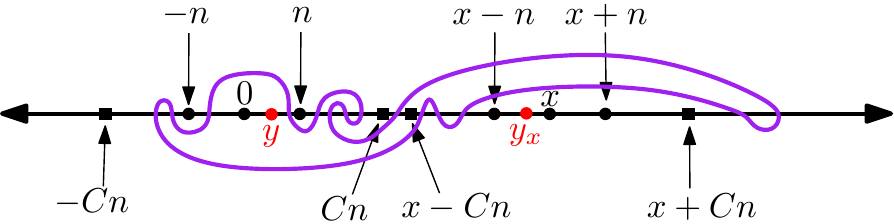}  
\caption{\label{fig-loop-exists} Illustration of the proof of Lemma~\ref{lem-loop-exists}. Let $x\in [2Cn+1,2Cn+3]$, so that $[-Cn,Cn]\cap[x-Cn,x+Cn]=\emptyset$. Note that we are considering two symmetric intervals of size $2Cn$ centered at $0$ and $x$ (recall Figure~\ref{fig-around-event}). If (i) $E_n^c$ occurs, (ii) the analogous event for $ [x-Cn,x+Cn]$ occurs, and (iii) the ``exposed'' points $y$ and $y_x$ given by Lemma~\ref{lem-disconnect-exposed} have opposite parity, then Lemma~\ref{lem-parity} (and our assumption that there is no infinite path) tells us that there is a loop in $\Gamma$ which disconnects $y-1/2$ from $y_x - 1/2$ (purple). By symmetry, we can take this loop to disconnect $y-1/2$ from $\infty$. By our choice of $y$ (from Lemma~\ref{lem-disconnect-exposed}), this loop must hit a vertex outside of $[-Cn,Cn]\cap\BB Z$. Since the loop disconnects $y$ from $y_x$, it must also hit a vertex in $[y , y_x]\cap\BB Z \subset [y,(2C+1)n + 3]\cap\BB Z$.  We note that the loop given by Lemma~\ref{lem-loop-exists} can have many possible behaviors besides what is shown in the figure. For example, it could disconnect $[-n,n]$ from $\infty$ and/or it could hit $(-\infty,-Cn]$. 
}
\end{center}
\vspace{-3ex}
\end{figure}

\begin{lem} \label{lem-loop-exists}
Assume that there is no infinite path in $\Gamma$.  
With probability at least $\BB P[E_n^c]^2/8$, there is a loop $\ell \in \Gamma$ with the following properties:
\begin{enumerate}[$(i)$]
\item \label{item-loop-dc} $\ell$ disconnects  $y-1/2$ from $\infty$ for some $y\in [-n + 1,n] \cap\BB Z$.
\item \label{item-loop-big} $\ell$ hits a point of $\BB Z\setminus [-C n , C n]$. 
\item \label{item-loop-close} $\ell$ hits a point of $[y, (2C+1) n + 3]\cap\BB Z$. 
\end{enumerate}
\end{lem}
\begin{proof}
See Figure~\ref{fig-loop-exists} for an illustration. 
Write $p : = \BB P[E_n^c]  $. 
Recall from Lemma~\ref{lem-disconnect-msrble} that $E_n$ is determined by $(\mcl L  , \mcl R)|_{[-C n , C n ] \cap\BB Z}$. 
On $E_n^c$, let $y \in [-n + 1 , n] \cap\BB Z$ be a point as in Lemma~\ref{lem-disconnect-exposed}, chosen in some manner which depends only on $(\mcl L  , \mcl R)|_{[-C n , C n ] \cap\BB Z}$ (on $E_n$, we arbitrarily set $y = 0$).
Then one of the events
\eqb
E_n^c\cap\left\{\text{$y$ is even}\right\} \quad \text{or} \quad E_n^c\cap\left\{\text{$y$ is odd}\right\}
\eqe
has probability at least $p/2$. We will assume that 
\eqb \label{eqn-even-assumption}
\BB P\left[ E_n^c\cap\left\{\text{$y$ is even}\right\} \right] \geq \frac{p}{2} .
\eqe
The other case is treated in an identical manner.

Let $x \in [2 C n + 1 , 2 C n + 3]\cap\BB Z$ be an odd integer (so that $[-Cn,Cn]\cap[x-Cn,x+Cn]=\emptyset$).
Define the event $E_n(x)$ in the same manner as the event $E_n$ from just above Lemma~\ref{lem-disconnect-msrble}, but with the translated meandric system $\mathcal M - x$ in place of $\mathcal M$.
Also let $y_x \in [x-n+1,x-n]\cap\BB Z$ be defined in the same manner as the point $y$ above but with $\mathcal M-x$ in place of $\mathcal M$. 
By~\eqref{eqn-even-assumption} and since $x$ is odd,
\eqb \label{eqn-odd-translate}
\BB P\left[ E_n^c(x) \cap\left\{\text{$y_x$ is odd}\right\} \right] \geq \frac{p}{2}  .
\eqe

The translated meandric system $\mathcal M -x$ is encoded by the translated pair of walks $j\mapsto (\mcl L_{j+x} -\mcl L_x, \mcl R_{j+x} -\mcl R_x)$ in the same manner that $\mcl M$ is encoded by $(\mcl L , \mcl R)$.
By Lemma~\ref{lem-disconnect-msrble} and the definition of $y_x$, the event $E_n(x)$ and the point $y_x$ are determined by the walk increment $(\mcl L - \mcl L_x  , \mcl R - \mcl R_x)|_{[x -C n , x +  C n ] \cap\BB Z}$. Since $[-C n , C n ] \cap [x-C n , x + C n] =\emptyset$, the pairs $(E_n , y)$ and $(E_n(x) , y_x)$ are independent. 
From this together with~\eqref{eqn-even-assumption} and~\eqref{eqn-odd-translate}, we obtain
\eqb \label{eqn-intersect-event}
\BB P\left[E_n^c\cap\left\{\text{$y$ is even}\right\}  \cap  E_n^c(x) \cap\left\{\text{$y_x$ is odd}\right\} \right] \geq \frac{p^2}{4}  .
\eqe

Recall that we are assuming that there is no infinite path in $\Gamma$. 
By Lemma~\ref{lem-parity}, if the event in~\eqref{eqn-intersect-event} occurs, then there is a loop $\ell \in \Gamma$ which disconnects $y -1/2$ from $y_x -1/2$. 
The loop $\ell$ has precisely two complementary connected components (by the Jordan curve theorem), so $\ell$ must disconnect exactly one of $y-1/2$ or $y_x-1/2$ from $\infty$.
Since the law of $\mathcal M$ is invariant under integer translation and reflection about the origin, symmetry considerations show that the probability that $\ell$ disconnects $y-1/2$ from $\infty$ is at least $p^2/8$. 

By our choice of $y$ (recall Lemma~\ref{lem-disconnect-exposed}), the loop $\ell$ must hit a point of $\BB Z\setminus [-C n , C n]$.  
Since $\ell$ disconnects $y$ from $y_x$, $\ell$ must hit an integer point in the interval
\eqbn
[y , y_x]\subset [y , x + n] \subset [y , (2C + 1) n + 3] .
\eqen
Therefore, $\ell$ satisfies the conditions in the lemma statement. 
\end{proof}

\begin{proof}[Proof of Theorem~\ref{thm-loop-dichotomy}]
First assume that there is an infinite path $\frk P$ in $\Gamma$. 
It holds with probability tending to 1 as $n\rta\infty$ that $\frk P$ intersects $[-n,n]\cap\BB Z$, so since $\frk P$ is infinite we get that condition~\ref{item-di-across} is satisfied with probability\footnote{Here and throughout this paper, if $f , g : \BB N \to (0,\infty)$ we write $g(n) = o_n(f(n))$ (resp.\ $g(n) = O_n(f(n))$) if $g(n) / f(n) \to 0$ (resp.\ $g(n) / f(n)$ remains bounded above) as $n\to\infty$.}
$1-o_n(1)$. We choose $n$ large enough so that this probability is at least $ 5-2\sqrt 6$. 

Now assume that there is no infinite path in $\Gamma$. 
Since $\BB P[E_n^c] = 1 - \BB P[E_n]  \in [0,1]$,  
\eqb \label{eqn-min-prob}
\max\left\{ \BB P[E_n] , \BB P[E_n^c]^2/8 \right\} \geq 5 - 2\sqrt 6. 
\eqe
By the definition~\refE\ of $E_n$, if $E_n$ occurs then condition~\ref{item-di-around} is satisfied, so condition~\ref{item-di-around} is satisfied with probability at least $\BB P[E_n]$.
In light of~\eqref{eqn-min-prob}, it remains to show that the probability that either~\ref{item-di-around} or~\ref{item-di-across} is satisfied is at least $\BB P[E_n^c]^2/8$. 
 
By Lemma~\ref{lem-loop-exists}, it holds with probability at least $\BB P[E_n^c]^2/8$ that there is a loop $\ell \in \Gamma$ satisfying the three properties in that lemma. 
If $\ell$ disconnects $[-n,n]$ from $\infty$, then property~\eqref{item-loop-close} of Lemma~\ref{lem-loop-exists} shows that condition~\ref{item-di-around} in the theorem statement is satisfied.
If $\ell$ does not disconnect $[-n,n]$ from $\infty$, then since $\ell$ disconnects some point of $[-n,n]\cap\BB Z$ from $\infty$ by property~\eqref{item-loop-dc}, we get that $\ell$ must hit a point of $[-n,n]\cap\BB Z$. 
By property~\eqref{item-loop-big}, this implies that condition~\ref{item-di-across} in the theorem statement is satisfied.
\end{proof}

\section{Bounding distances via the mated-CRT map}
\label{sec-mated-crt}

Recall that $\mathcal M$ denotes the planar map associated with the UIMS.
In this section, we will prove a lower bound for certain graph distances in $\mathcal M$ (Proposition~\ref{prop-rpm-estimate} just below), conditional on an estimate (Proposition~\ref{prop-mcrt-estimate}) which will be proven in Section~\ref{sec-lqg} using Liouville quantum gravity techniques.   Then, in Section~\ref{sec-macro-proof}, we will combine Proposition~\ref{prop-rpm-estimate} and Theorem~\ref{thm-loop-dichotomy} to deduce our main results on the sizes of loops in meandric systems (Proposition~\ref{prop-map-diam}, Theorem~\ref{thm-macro}, and Theorem~\ref{thm-origin-loop}).

To state our lower bound for distances in $\mathcal M$, we introduce some notation. 
We define the submaps
\eqb \label{eqn-rpm-submap}
\mathcal M[a,b] := \left(\text{submap of $\mathcal M$ induced by $[a,b]\cap\BB Z$}\right) ,\quad \forall -\infty < a < b < \infty . 
\eqe 
For a graph $G$, we also define
\eqb \label{eqn-graph-dist}
\op{dist}_G(x,y):= \left(\text{$G$-graph distance from $x$ to $y$} \right) ,\quad \text{$\forall$ vertices $x,y\in G$}. 
\eqe
If $H$ is a subgraph of $G$ and $x,y\in H$, then $\op{dist}_G(x,y) \leq \op{dist}_H(x,y)$. The inequality can be strict since the minimal-length path in $G$ between $x$ and $y$ may not be contained in $H$. We will frequently use this fact without comment, often with $H = M[a,b]$ as in~\eqref{eqn-rpm-submap}. 
 
\begin{prop} \label{prop-rpm-estimate}
For each $\zeta \in (0,1)$ and each $p\in (0,1)$, there exists $C = C(p , \zeta) > 3$ such that for each large enough $n\in\BB N$, it holds with probability at least $p$ that the following is true:
\begin{enumerate}[$(i)$]
\item \label{item-rpm-in} $\op{dist}_{\mathcal M} \left( [-n,n]\cap\BB Z , \BB Z\setminus [-C n , C n] \right) \geq n^{1/d-\zeta}$. 
\item \label{item-rpm-out} $\op{dist}_{\mathcal M} \left( [-3 C n, 3 C n]\cap\BB Z , \BB Z\setminus [-C^2 n , C^2 n] \right) \geq n^{1/d-\zeta}$. 
\item \label{item-rpm-out'} $\op{dist}_{\mathcal M} \left( [-C^2 n, C^2 n]\cap\BB Z , \BB Z\setminus [-C^3 n , C^3 n] \right) \geq n^{1/d-\zeta}$. 
\item \label{item-rpm-around} There are two paths $\Pi_1,\Pi_2$ in $\mathcal M$ each going from a vertex of $[-n  , n ]\cap\BB Z$ to a vertex of $\BB Z\setminus [- C^2 n , C^2 n] $ which lie at $\mathcal M[-C^3 n , C^3 n]$-graph distance at least $  n^{1/d-\zeta}$ from each other.
\end{enumerate}
\end{prop}

For non-uniform random planar maps, it appears to be quite difficult to estimate graph distances directly. 
So, to prove Proposition~\ref{prop-rpm-estimate}, we will use an indirect approach which was introduced in~\cite{ghs-map-dist}.
The idea is as follows. In Section~\ref{sec-mcrt}, we will define the mated-CRT map $\mcl G$, a random planar map constructed from a pair of independent two-sided Brownian motions via a semicontinuous analog of the construction of $\mathcal M$ from a pair of independent two-sided random walks.
We will also state a comparison result for distances in $\mathcal M$ and distances in $\mcl G$ (Theorem~\ref{thm-ghs-coupling}) which follows from a more general result in~\cite{ghs-map-dist}. 

This comparison result allows us to reduce Proposition~\ref{prop-rpm-estimate} to a similar estimate for the mated-CRT map (Proposition~\ref{prop-mcrt-estimate}).
The proof of this latter estimate is given in Section~\ref{sec-lqg}. 
Due to the results of~\cite{wedges}, the mated-CRT map admits an alternative description in terms of $\sqrt 2$-LQG decorated by SLE$_8$ (see Section~\ref{sec-lqg-mcrt}). 
Using this description, the needed estimate for the mated-CRT map turns out to be an easy consequence of known results for SLEs and LQG. 

\subsection{The mated-CRT map} 
\label{sec-mcrt}

%

Let $(L,R)$ be a pair of standard linear two-sided Brownian motions, with $L_0 = R_0 = 0$. 
The \textbf{mated-CRT map} associated with $(L,R)$ is the graph $\mcl G$ with vertex set $ \BB Z$ and edge set defined as follows. 
Two integers $x_1,x_2 \in \BB Z$ with $x_1<x_2$ are joined by an edge if and only if either
\allb \label{eqn-mcrt-def}
&\max\left\{ \inf_{t\in [x_1-1 , x_1]} L_t   ,\, \inf_{t\in [x_2 - 1 , x_2]} L_t  \right\} \leq \inf_{t\in [x_1  , x_2-1]} L_t \quad\op{or}\quad \notag \\
&\max\left\{ \inf_{t\in [x_1-1 , x_1]} R_t   ,\, \inf_{t\in [x_2 - 1 , x_2]} R_t  \right\} \leq \inf_{t\in [x_1  , x_2-1]} R_t  .
\alle
We note that a.s.\ both conditions in~\eqref{eqn-mcrt-def} hold whenever $x_2 = x_1+1$. 
If both conditions in~\eqref{eqn-mcrt-def} hold and $ x_2 \geq x_1+2$, we declare that $x_1$ and $x_2$ are joined by two edges. 

The edge set of $\mcl G$ naturally splits into three subsets:
\begin{itemize}
\item \textbf{Trivial edges}, which join $x$ and $x+1$ for $x\in\BB Z$.
\item \textbf{Upper edges}, which join $x_1  , x_2\in\BB Z$ with $x_2 \geq x_1+2$ and arise from the first condition (the one involving $L$) in~\eqref{eqn-mcrt-def}.
\item \textbf{Lower edges}, which join $x_1  , x_2\in\BB Z$ with $x_2 \geq x_1+2$ and arise from the second condition (the one involving $R$) in~\eqref{eqn-mcrt-def}.
\end{itemize}
We can assign a planar map structure to $\mcl G$ by associating each trivial edge with the line segment from $x$ to $x+1$ in $\BB R^2$, each upper edge $\{x_1,x_2\}$ with an arc from $x_1$ to $x_2$ in the upper half-plane, and each lower edge $\{x_1,x_2\}$ with an arc from $x_1$ to $x_2$ in the lower half-plane. In fact, $\mcl G$ is a triangulation when equipped with this planar map structure. See Figure~\ref{fig-mated-crt-map}.

\begin{figure}[ht!]
 \begin{center}
\includegraphics[scale=.75]{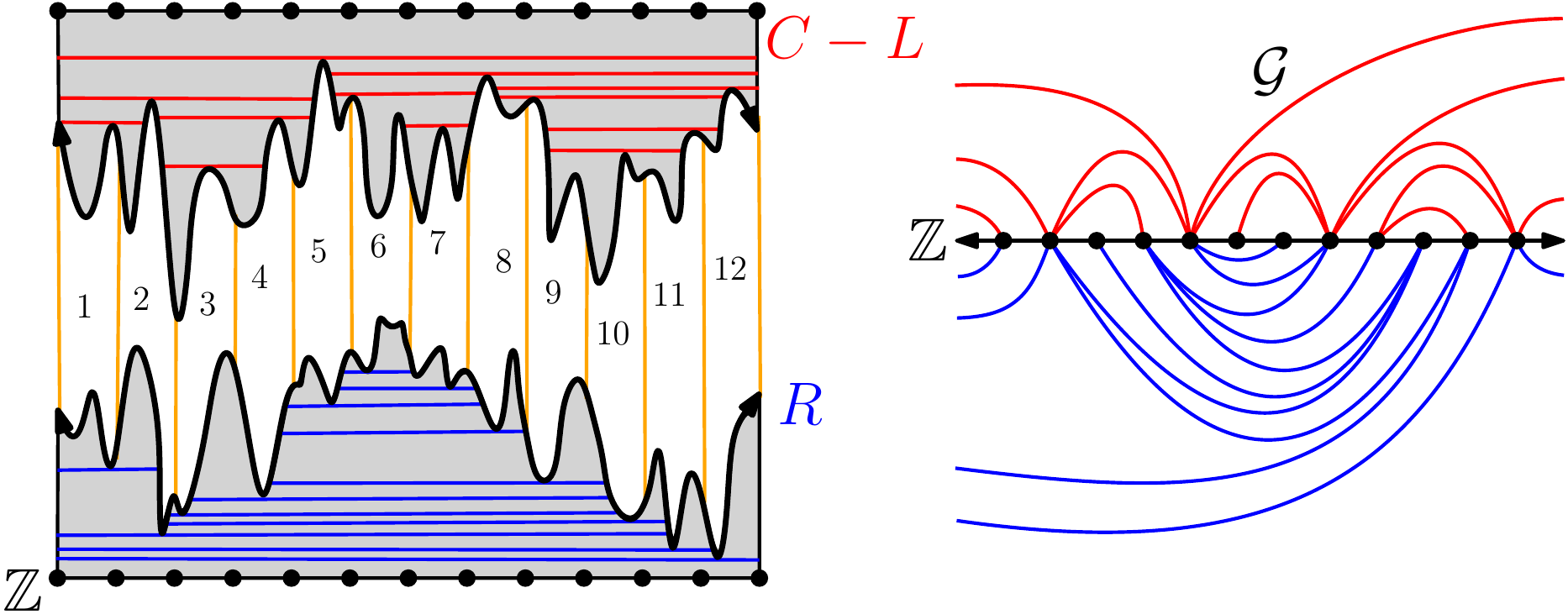} 
\caption{\textbf{Left:} To construct the mated-CRT map $\mcl G $ geometrically, one can draw the graph of $C - L$ (red) and the graph of $ R$ (blue) for some large constant $C > 0$ chosen so that the parts of the graphs over some time interval of interest do not intersect. Here, this time interval is $[0,12 ]$. One then divides the region between the graphs into vertical strips (boundaries shown in orange). Each vertical strip corresponds to the vertex $x\in   \BB Z$ which is the horizontal coordinate of its rightmost points. Vertices $x_1,x_2\in  \BB Z$ are connected by an edge if and only if the corresponding vertical strips are connected by a horizontal line segment which lies above the graph of $C - L$ or below the graph of $ R$. For each pair of vertices for which the condition holds for $C - L$ (resp.\ $ R$), we have drawn the highest (resp.\ lowest) segment which joins the corresponding vertical strips in red (resp.\ blue). Equivalently, for each $x\in  \BB Z$, we let $t_x$ be the time in $[x-1,x]$ at which $L$ attains its minimum value and we draw in red the longest horizontal segment above the graph of $C - L$ which contains $(t_x , C - L_{t_x})$; and we perform a similar procedure for $R$. 
Note that consecutive vertices are always joined by an edge.
\textbf{Right:} One can draw the graph $\mcl G $ in the plane by connecting two non-consecutive vertices $x_1,x_2 \in  \BB Z$ by an arc above (resp.\ below) the real line if the corresponding vertical strips are connected by a horizontal segment above (resp.\ below) the graph of $C - L$ (resp.\ $ R$); and connecting each pair of consecutive vertices of $ \BB Z$ by an edge. This gives $\mcl G $ a planar map structure under which it is a triangulation. A similar figure and caption appeared in~\cite{gms-harmonic}. 
}\label{fig-mated-crt-map}
\end{center}
\end{figure}

Analogously to~\eqref{eqn-rpm-submap}, we define
\eqb \label{eqn-mcrt-submap}
\mcl G[a,b] := \left(\text{submap of $\mcl G$ induced by $[a,b]\cap\BB Z$}\right) ,\quad \forall -\infty < a < b < \infty . 
\eqe

The definition of $\mcl G$ is similar to the construction of the UIMS $\mathcal M$ from the pair of bi-infinite simple random walks $(\mcl L , \mcl R)$. 
In particular, by~\eqref{eqn-rw-inf-adjacency}, two vertices $x_1,x_2\in\BB Z$ are joined by an edge of $\mathcal M$ if and only if 
\allb \label{eqn-rw-map-def}
 x_2 = x_1 + 1 \quad \text{or} \quad
 \mcl L_{x_1-1} = \mcl L_{x_2} < \min_{y \in [x_1 ,x_2-1]} \mcl L_y \quad\op{or}\quad  
 \mcl R_{x_1-1} = \mcl R_{x_2} < \min_{y \in [x_1 ,x_2-1]} \mcl R_y  
\alle 
which is a continuous time analog of~\eqref{eqn-mcrt-def}. 
We note, however, that for $x\in\BB Z$ the number of arcs of $\mcl G$ in the upper (resp.\ lower) half-plane incident to $x$ can be any non-negative integer, whereas for $\mathcal M$ the number of such arcs is always one.

Using the KMT strong coupling theorem for random walk and Brownian motion~\cite{kmt,zaitsev-kmt} and an elementary geometric argument, it was shown in~\cite[Theorem 2.1]{ghs-map-dist} that one can couple $(L,R)$ and $(\mcl L , \mcl R)$ so that the graph distances in their corresponding planar maps $\mathcal M$ and $\mcl G$ are close (actually,~\cite{ghs-map-dist} considers a more general class of pairs of random walks).

\begin{thm}[\cite{ghs-map-dist}] \label{thm-ghs-coupling} 
For each $\alpha > 0$, there exists $A = A(\alpha) > 0$ and a coupling of $(L,R)$ with $(\mcl L , \mcl R)$ such that the following is true with probability $1-O_n(n^{-\alpha})$.
Let $\mcl G$ be the mated-CRT map constructed from $(L,R)$ as in~\eqref{eqn-mcrt-def} and let $\mathcal M$ be the infinite planar map associated with the UIMS, constructed from $(\mcl L , \mcl R)$ as in Section \ref{sec-infinite}.  
For each $x,y \in [-n,n]\cap\BB Z$, 
\eqb \label{eqn-ghs-dist}
A^{-1} (\log n)^{-3} \op{dist}_{\mcl G[-n,n]}(x , y)  \leq \op{dist}_{\mathcal M[-n,n]}(x , y) \leq A (\log n)^3 \op{dist}_{\mcl G[-n,n]}(x , y) ,
\eqe 
where here we use the notations~\eqref{eqn-rpm-submap} and~\eqref{eqn-mcrt-submap}. 
\end{thm}

\subsection{Proof of lower bounds for graph distances} 
\label{sec-rpm-estimate}

In this section we prove Proposition~\ref{prop-rpm-estimate}.
We will deduce it from the combination of Theorem~\ref{thm-ghs-coupling} and the following analogous estimate for the mated-CRT map. 

\begin{prop} \label{prop-mcrt-estimate}
For each $\zeta \in (0,1)$ and each $p\in (0,1)$, there exists $C = C(p , \zeta) > 3$ such that for each large enough $n\in\BB N$, it holds with probability at least $p$ that the following is true:
\begin{enumerate}[$(a)$]
\item \label{item-mcrt-in} $\op{dist}_{\mcl G} \left( [-n,n]\cap\BB Z , \BB Z\setminus [-C n , C n] \right) \geq n^{1/d-\zeta}$. 
\item \label{item-mcrt-out} $\op{dist}_{\mcl G} \left( [-3 C n, 3 C n]\cap\BB Z , \BB Z\setminus [-C^2 n , C^2 n] \right) \geq n^{1/d-\zeta}$. 
\item \label{item-mcrt-out'} $\op{dist}_{\mcl G} \left( [-C^2 n, C^2 n]\cap\BB Z , \BB Z\setminus [-C^3 n , C^3 n] \right) \geq n^{1/d-\zeta}$. 
\item \label{item-mcrt-around} There are two paths $\Pi_1,\Pi_2$ in $\mcl G$ each going from a vertex of $[-n  , n ]\cap\BB Z$ to a vertex of $\BB Z\setminus [- C^2 n , C^2 n] $ which lie at $\mcl G$-graph distance at least $ n^{1/d-\zeta}$ from each other.
\end{enumerate}
\end{prop}

The proof of Proposition~\ref{prop-mcrt-estimate} is given in Section~\ref{sec-lqg}, using the relationship between the mated-CRT map and SLE-decorated LQG. 
We now deduce Proposition~\ref{prop-rpm-estimate} from Proposition~\ref{prop-mcrt-estimate}.

\begin{proof}[Proof of Proposition~\ref{prop-rpm-estimate}]
We divide the proof in three main steps.

\medskip

\noindent\textit{\underline{Step 1:} Regularity event.}
We first define a high-probability event which we will work on throughout the rest of the proof.
Let $A_0 > 0$ be as in Theorem~\ref{thm-ghs-coupling} with $\alpha = 1$, say, and let $A := 8 A_0$.  
By Theorem~\ref{thm-ghs-coupling} with $n^2$ in place of $n$, we can couple $\mathcal M$ and $\mcl G$ so that with probability tending to 1 as $n\rta\infty$, it holds for each $x,y\in [-n^2 ,n^2]\cap\BB Z$ that   
\eqb \label{eqn-use-ghs-coupling}
A^{-1} (\log n)^{-3} \op{dist}_{\mcl G [-n^2,n^2] }(x , y)  \leq \op{dist}_{M [-n^2,n^2] }(x , y) \leq A (\log n)^3 \op{dist}_{\mcl G [-n^2,n^2] }(x , y) .
\eqe

In order to make sure that we can compare distances in $\mathcal M$ and $\mcl G$ from points in $[-C^2 n , C^2 n] \cap\BB Z$ to the set $\BB Z \setminus [-C^2 n , C^2 n]$, we also need to impose a further regularity condition. 
By the adjacency conditions~\eqref{eqn-rw-map-def} and~\eqref{eqn-mcrt-def} for $\mathcal M$ and $\mcl G$ in terms of $(\mcl L  , \mcl R)$ and $(L,R)$, respectively, together with basic estimates for random walk and Brownian motion, for any fixed $C>1$ it holds with probability tending to 1 as $n\rta\infty$ that
\eqb \label{eqn-bdy-contained}
\text{$\forall$ edges $ \{x,y\} \in  \mcl M$ such that $x\in [-C^3 n ,C^3 n]\cap\BB Z$, we have $y\in [-n^2 ,n^2]$} ;
\eqe
and the same is true with $\mcl G$ in place of $\mathcal M$.
Note that the quantity $n^2$ is unimportant in~\eqref{eqn-bdy-contained}: the same would be true with $n^2$ replaced by any function of $n$ which goes to $\infty$ as $n\rta\infty$. 

We now take $C = C( 1 - (1-p)/2 , \zeta/2) > 0$ as in Proposition~\ref{prop-mcrt-estimate} with $1-(1-p)/2$ instead of $p$ and $\zeta/2$ instead of $\zeta$. 
By Proposition~\ref{prop-mcrt-estimate} and our above estimates, for each large $n\in\BB N$ it holds with probability at least $p$ that~\eqref{eqn-use-ghs-coupling} and~\eqref{eqn-bdy-contained} both hold and the numbered conditions in Proposition~\ref{prop-mcrt-estimate} hold with $\zeta/2$ in place of $\zeta$. Henceforth assume that this is the case. The rest of the argument is deterministic.
\medskip

\noindent\textit{\underline{Step 2:} Proofs of~\eqref{item-rpm-in}, ~\eqref{item-rpm-out}, and~\eqref{item-rpm-out'}.}
Consider a path $P$ in $\mathcal M$ from a point $x \in [-n,n]\cap\BB Z$ to a point of $\BB Z\setminus [-C n , C n]$.
By~\eqref{eqn-bdy-contained}, $P$ hits a vertex in $  \BB Z \cap ( [-n^2 , n^2] \setminus [-C n , C n] )$. 
Let $y$ be the first such vertex.
Then the segment $P'$ of $P$ between its starting point and the first time it hits $y$ is a path in $\mathcal M[-n^2,n^2]$.  
Using our above estimates, we now get that the length of $P'$ satisfies 
\allb \label{eqn-rpm-in0}
|P'| 
&\geq  A^{-1} (\log n)^{-3} \op{dist}_{\mcl G[-n^2 ,n^2]}(x , y)  \quad \text{(by~\eqref{eqn-use-ghs-coupling})} \notag\\ 
&\geq A^{-1} (\log n)^{-3} \op{dist}_{\mcl G} \left( [-n,n]\cap\BB Z , \BB Z\setminus [-C n , C n] \right) \quad \text{(choice of $x$ and $y$)} \notag\\ 
&\geq A^{-1} (\log n)^{-3} n^{1/d-\zeta/2} \quad \text{(\eqref{item-mcrt-in} of Proposition~\ref{prop-mcrt-estimate} with $\zeta/2$ instead of $\zeta$)} .
\alle 
Since $P'$ is a sub-path of $P$, we have $|P| \geq |P'|$. Taking the infimum over all $P$ now shows that $\op{dist}_{\mathcal M} \left( [-n,n]\cap\BB Z , \BB Z\setminus [-C n , C n] \right) $ is at least the right side of~\eqref{eqn-rpm-in0}, which is at least $n^{1/d-\zeta}$ if $n$ is large enough. Thus~\eqref{item-rpm-in} in the proposition statement holds.

The proofs of~\eqref{item-rpm-out} and~\eqref{item-rpm-out'} are identical to the proof of~\eqref{item-rpm-in}, except that we use~\eqref{item-mcrt-out} and~\eqref{item-mcrt-out'} from Proposition~\ref{prop-mcrt-estimate} instead of~\eqref{item-mcrt-in}. 
\medskip

\noindent\textit{\underline{Step 3:} Proof of~\eqref{item-rpm-around}.} 
Let $\wt\Pi_1$ and $\wt\Pi_2$ be the paths in $\mcl G$ as in \eqref{item-mcrt-around} of Proposition~\ref{prop-mcrt-estimate} with $\zeta/2$ instead of $\zeta$.
By possibly replacing $\wt\Pi_1$ and $\wt\Pi_2$ by sub-paths (which can only increase the graph distance between them), we can assume that each of these paths intersects $[-n  , n ]\cap\BB Z$ and $\BB Z\setminus [- C^2 n , C^2 n] $ only at its endpoints.

We view $\wt\Pi_1$ as a function $\wt\Pi_1 : [0,N]\cap\BB Z \rta \BB Z$. For $i=1,\dots,N$, we apply~\eqref{eqn-use-ghs-coupling} with $(\wt \Pi_1(i-1) , \wt \Pi_1(i))$ in place of $(x,y)$ to get $N$ new paths in $\mathcal M[-n^2 ,n^2]$ (the ones realizing $\op{dist}_{\mathcal M[-n^2,n^2]}(\wt \Pi_1(i-1) , \wt \Pi_1(i))$), each of length at most $A(\log n)^3$. Then we concatenate the $N$ new paths in $\mathcal M[-n^2,n^2]$. This results in a path $\Pi_1$ in $\mathcal M[-n^2 ,n^2]$ with the same endpoints as $\wt\Pi_1$ with the property that each point of $\Pi_1$ lies at $\mathcal M[-n^2,n^2]$-graph distance at most $A(\log n)^3$ from a point of $\wt \Pi_1$. 
We similarly construct a path $\Pi_2$ in $\mathcal M[-n^2 ,n^2]$ but started from $\wt\Pi_2$ instead of $\wt\Pi_1$.  

Since $\Pi_1,\Pi_2$ have the same endpoints as $\wt\Pi_1,\wt\Pi_2$, these paths each go from a vertex of $[-n,n]\cap\BB Z$ to a vertex of $\BB Z\setminus[-C^2 n , C^2 n]$. 
Furthermore, by the distance estimate in the preceding paragraph and the triangle inequality, the $\mathcal M[-n^2,n^2]$-graph distance from any point of $\Pi_1$ to any point of $\Pi_2$ is at least $\op{dist}_{\mathcal M[-n^2,n^2]} (\wt\Pi_1,\wt\Pi_2)  - 2 A(\log n)^3$. By this and~\eqref{eqn-use-ghs-coupling}, 
\allb \label{eqn-split-paths-compare}
\op{dist}_{\mathcal M[-C^3 n , C^3 n]} (\Pi_1, \Pi_2) 
&\geq \op{dist}_{\mathcal M[-n^2,n^2]} (\Pi_1, \Pi_2)  \notag\\
&\geq \op{dist}_{\mathcal M[-n^2,n^2]} (\wt\Pi_1,\wt\Pi_2)  - 2 A(\log n)^3   \notag\\
&\geq  A^{-1} (\log n)^{-3}   \op{dist}_{\mcl G} (\wt\Pi_1,\wt\Pi_2) - 2 A (\log n)^3   .
\alle 
By our initial choice of $\wt\Pi_1$ and $\wt\Pi_2$  the right side of~\eqref{eqn-split-paths-compare} is at least $A^{-1} (\log n)^{-3}n^{1/d - \zeta/2}  - 2 A (\log n)^3  $, which is at least $n^{1/d-\zeta}$ for each large enough $n\in\BB N$. 
\end{proof}

\subsection{Proof of Proposition~\ref{prop-map-diam}, Theorem~\ref{thm-macro}, and Theorem~\ref{thm-origin-loop}} 
\label{sec-macro-proof}
 
Combining Theorem~\ref{thm-loop-dichotomy} and Proposition~\ref{prop-rpm-estimate} leads to the following result.

\begin{prop} \label{prop-loop-dist}
For each $\zeta \in (0,1)$ and each sufficiently large $n\in\BB N$ (depending on $\zeta$), it holds with probability at least $1/10$ that the following is true.
There is a segment $\ol\ell$ of a loop or infinite path in $\Gamma$ which is contained in $[-  n ,   n] \cap \BB Z$ such that the $\mathcal M[-  n ,   n]$-graph-distance diameter of $\ol \ell$ and the $\mathcal M$-graph distance from $\ol \ell$ to $\BB Z\setminus [- n ,  n] $ are each at least $n^{1/d-\zeta}$. 
\end{prop}
\begin{proof}
Fix $\zeta \in (0,1)$. By Theorem~\ref{thm-loop-dichotomy}, for any choice of $C>1$, for each large enough $n\in\BB N$ it holds with probability at least $5-2\sqrt 6 > 1/10$ that at least one of~\ref{item-di-around} or~\ref{item-di-across} in the statement of Theorem~\ref{thm-loop-dichotomy} is satisfied.
By Proposition~\ref{prop-rpm-estimate}, there is a universal constant $C > 1$ so that for each large enough $n$, it holds with probability at least $1 - (5-2\sqrt 6 -1/10)$ that all four of the numbered conditions in the statement of Proposition~\ref{prop-rpm-estimate} are satisfied.
Henceforth assume that the events of Theorem~\ref{thm-loop-dichotomy} and Proposition~\ref{prop-rpm-estimate} occur (both with this same choice of $C$), which happens with probability at least $1/10$ if $n$ is large enough.  

We will show that 
\begin{enumerate} \label{eqn-loop-dist-show}
\item[$(*)$] there is a segment $\ol\ell$ of a loop or infinite path in $\Gamma$ which is contained in $[- C^3 n ,   C^3 n] \cap \BB Z$ such that the $\mathcal M[-  C^3 n ,  C^3 n]$-graph-distance diameter of $\ol \ell$ and the $\mathcal M$-graph distance from $\ol \ell$ to $\BB Z\setminus [- C^3 n ,  C^3 n]\cap\BB Z$ are each at least $n^{1/d-\zeta}$. 
\end{enumerate}
\newcommand{\refShow}{{(\hyperref[eqn-loop-dist-show]{$*$})}}
To extract the proposition statement from~\refShow, we can apply~\refShow\ with $\lfloor n/C^3\rfloor$ in place of $n$, then slightly shrink $\zeta$ in order to absorb a factor of $1/C^{3\zeta}$ into a power of $n$. 
 
To prove~\refShow, we will treat the two possible scenarios in Theorem~\ref{thm-loop-dichotomy} separately. 
First suppose that 
\begin{enumerate}[A.]
\item \label{item-di-around'} there is a loop $\ell$ in $\Gamma$ which disconnects $[-n,n]$ from $\infty$ and which hits a point of $[n , (2C+1)n + 3]\cap\BB Z$. 
\end{enumerate}
If $\ell$ is contained  in $[-C^2 n , C^2 n]\cap\BB Z$, then by planarity and since $\ell$ disconnects $[-n,n]$ from $\infty$, the loop $\ell$ must intersect every path in $\mathcal M$ from $[-n,n]\cap\BB Z$ to $\BB Z\setminus [-C^2 n , C^2 n]$. In particular, $\ell$ must intersect the paths $\Pi_1$ and $\Pi_2$ from~\eqref{item-rpm-around} of Proposition~\ref{prop-rpm-estimate}. Hence, the $\mathcal M[-C^3 n , C^3 n]$-graph-distance diameter of $\ell$ is at least $ n^{1/d-\zeta}$.
Furthermore, by~\eqref{item-rpm-out'} of Proposition~\ref{prop-rpm-estimate}, the $\mathcal M$-graph distance from $\ell$ to $\BB Z\setminus [-C^3 n , C^3 n]\cap\BB Z$ is at least $n^{1/d-\zeta}$. So, we can take $\ol\ell=\ell$. 

If $\ell$ is not contained in $[-C^2 n , C^2 n]\cap\BB Z$, then since $\ell$ hits $[n,(2C+1)n+3]\cap\BB Z$, there is a segment $\ol\ell $ of $\ell$ which is a path in $\mathcal M$ from a point of $[n , (2C+1) n + 3] \cap\BB Z \subset [-3 C n , 3 C n]\cap\BB Z$ to a point of $\BB Z\setminus [-C^2 n , C^2 n]$.
By possibly replacing $\ol\ell$ with a sub-path, we can assume that $\ol\ell$ is contained in $[-C^2 n , C^2 n] \cap\BB Z$ except for its terminal endpoint.

By~\eqref{item-rpm-out} of Proposition~\ref{prop-rpm-estimate}, the $\mathcal M[-C^3 n , C^3 n]$-graph-distance diameter of $\ol\ell$ is at least $n^{1/d-\zeta}$. 
By~\eqref{item-rpm-out'} of Proposition~\ref{prop-rpm-estimate}, the $\mathcal M$-graph distance from $\ol\ell$ to $[-C^3 n , C^3 n]\cap\BB Z$ is at least $n^{1/d-\zeta}$.
  
Next suppose that
\begin{enumerate}[A.] 
\setcounter{enumi}{1}
\item \label{item-di-across'} there is a loop or an infinite path in  $\Gamma$ which hits a point of each of $[-n,n]\cap\BB Z$ and $\BB Z\setminus [-C n , C n] $. 
\end{enumerate}
Let $\ol\ell$ be a segment of this loop or infinite path which is a path from a point of $[-n,n]\cap\BB Z$ to a point of $\BB Z\setminus [-C n , C n]$. By possibly replacing $\ol\ell$ with a sub-path, we can assume that $\ol\ell$ is contained in $[-C n , C n] \cap\BB Z$ except for its terminal endpoint.
Then~\eqref{item-rpm-in} of Proposition~\ref{prop-rpm-estimate} shows that the $\mathcal M[-C^3 n , C^3 n]$-graph-distance diameter of $\ol\ell$ is at least $n^{1/d-\zeta}$ and~\eqref{item-rpm-out'} of Proposition~\ref{prop-rpm-estimate} shows that the $\mathcal M$-graph distance from $\ol\ell$ to $[-C^3 n , C^3 n]\cap\BB Z$ is at least $n^{1/d-\zeta}$.
\end{proof}
 
Using Proposition~\ref{prop-loop-dist}, we obtain our lower tail bound for the diameter of the origin-containing loop in the UIMS.

\begin{proof}[Proof of Theorem~\ref{thm-origin-loop}]
Let $\zeta \in (0,1)$, which we will eventually send to zero. 
For $k\in\BB N$, let $X = X_k$ be sampled uniformly from $[-k^{d+\zeta} ,k^{d+\zeta}]\cap\BB Z$, independently from everything else, and let $\ell_X$ be the loop in $\Gamma$ which hits $X$. By the translation invariance of the law of $(\mathcal M,\Gamma)$, the translated planar map / loop pair $(\mathcal M-X,\Gamma-X)$ has the same law as $(\mathcal M,\Gamma)$. 
Hence, 
\eqb \label{eqn-uniform-loop}
\BB P\left[ \left(\text{$\mathcal M$-graph-distance diameter of $\ell_0$} \right) \geq k \right]
= \BB P\left[ \left(\text{$\mathcal M$-graph-distance diameter of $\ell_X$} \right) \geq k \right]   .
\eqe
We will now lower-bound the second quantity in~\eqref{eqn-uniform-loop}. 

By Proposition~\ref{prop-loop-dist} (with $\lfloor k^{d+\zeta} \rfloor$ in place of $n$ and with $\zeta$ possibly replaced by a smaller positive number), if $k$ is large enough then it holds with probability at least $1/10$ that there is a segment $\ol\ell$ of a loop or infinite path in $\Gamma$ which is contained in $[-  k^{d+\zeta}  ,   k^{d+\zeta}] \cap \BB Z$ such that the $\mathcal M[-  k^{d+\zeta},   k^{d+\zeta}]$-graph-distance diameter of $\ol \ell$ and the $\mathcal M$-graph distance from $\ol \ell$ to $\BB Z\setminus [- k^{d+\zeta} ,  k^{d+\zeta}] $ are each at least $k $. Let $E_k$ be the event that such an $\ol\ell$ exists. On $E_k$ let $\ol\ell$ be a path as in the definition of $E_k$, chosen in some manner which is measurable with respect to $\sigma(\mathcal M,\Gamma)$. 

We claim that if $E_k$ occurs, then the $\mathcal M$-graph-distance diameter of $\ol\ell$, and hence also the number of vertices of $\mathcal M$ hit by $\ol\ell$, are each at least $k $. 
Indeed, by definition, there are two vertices $x,y\in [-k^{d+\zeta},k^{d+\zeta}]\cap\BB Z$ hit by $\ol\ell$ which lie at $\mathcal M[-k^{d+\zeta},k^{d+\zeta}]$-graph distance at least $k $ from each other. Any path in $\mathcal M$ from $x$ to $y$ must either stay in $[-k^{d+\zeta},k^{d+\zeta}]\cap\BB Z$, in which case its length is at least $k $ by our choice of $x,y$; or must cross from $x$ to $\BB Z\setminus [-k^{d+\zeta},k^{d+\zeta}]$, in which case its length is also at least $k $ since the $\mathcal M$-graph distance from $\ol\ell$ to $\BB Z\setminus [-k^{d+\zeta},k^{d+\zeta}]$ is at least $k $. Taking the infimum over all such paths gives our claim.
 
Since $X$ is sampled uniformly from $[-k^{d+\zeta},k^{d+\zeta}]\cap\BB Z$, independently from $(\mathcal M,\Gamma)$, the preceding paragraph implies that
\eqb \label{eqn-cond-on-map}
\BB P\left[ X \in \ol \ell \,\middle|\, (\mathcal M,\Gamma) \right] \BB 1_{E_k}   \geq  \frac12 k^{-(d-1)-\zeta} \BB 1_{E_k}   .
\eqe
Since $\BB P[E_k] \geq 1/10$ and the $\mathcal M$-graph-distance diameter of $\ol\ell$ is at least $k $ on $E_k$, taking the expectations on both sides of~\eqref{eqn-cond-on-map} gives
\eqb \label{eqn-origin-loop0}
\BB P\left[ \left(\text{$\mathcal M$-graph-distance diameter of $\ell_X$} \right) \geq k \right] \geq \frac{1}{20} k^{-(d-1) - \zeta} .
\eqe
By~\eqref{eqn-uniform-loop}, \eqref{eqn-origin-loop0} implies~\eqref{eqn-origin-loop} upon sending $\zeta\rta 0$. 
\end{proof}

The following proposition is the analog of Theorem~\ref{thm-macro} for the UIMS. It is the main input in the proof of Theorem~\ref{thm-macro}. 
 
\begin{prop}  \label{prop-plane-macro}
For each $\zeta \in (0,1)$, there exists $\beta > 0$ and $a_0,a_1> 0$, depending on $\zeta$, such that for each $n\in\BB N$, it holds with probability at least $1 - a_0 e^{-a_1 n^\beta}$ that there is a segment of a loop or an infinite path in $\Gamma$ which hits only vertices of $[1,2n]\cap\BB Z$ and has $\mathcal M$-graph-distance diameter at least $n^{1/d-\zeta}$.
\end{prop}
\begin{proof} 
Let $\beta \in(0,1)$ to be chosen later, depending on $\zeta$. 
The idea of the proof is as follows. We will consider a collection of $\op{const} \times n^\beta$ disjoint sub-intervals of $[1,2n]$ of length $2\lfloor n^{1-\beta}\rfloor$.
We will then use independence to show that with high probability, the event of Proposition~\ref{prop-loop-dist} (with $\lfloor n^{1-\beta} \rfloor$ instead of $n$) occurs for at least one of these intervals. 

For $n \in \BB N$ and $x\in \BB Z$, let $G_n(x)$ be the event that the following is true:
\begin{enumerate}
\item[$(\ddagger)$] There is a segment $\ol\ell$ of a loop or infinite path in $\Gamma$ which is contained in $[x -  n , x +   n] \cap \BB Z$ such that the $\mathcal M[x - n , x +  n]$-graph-distance diameter of $\ol \ell$ and the $\mathcal M$-graph distance from $\ol \ell$ to $\BB Z\setminus [x - n , x + n] $ are each at least $n^{1/d-\beta}$. 
\end{enumerate} 
By the translation invariance of the law of $\mathcal M$, Proposition~\ref{prop-loop-dist} implies that for each large enough $n\in\BB N$, 
\eqb \label{eqn-dist-event-prob}
\BB P\left[ G_n(x) \right] \geq \frac{1}{10} ,\quad \forall x \in \BB Z  .
\eqe
Furthermore, the event $G_n(x)$ depends only on the set of edges of $\mathcal M$ between vertices in $[x-n,x+n]\cap\BB Z$ and the set of vertices in $[x-n,x+n]\cap\BB Z$ which are joined by edges of $\mathcal M$ to vertices not in $[x-n,x+n]\cap\BB Z$. This information is determined by the restricted, shifted walk $(\mcl L - \mcl L_x , \mcl R - \mcl R_x)|_{[x-n,x+n]\cap\BB Z}$. Consequently,
\eqb \label{eqn-dist-event-ind}
\text{$G_n(x)$ and $G_n(y)$ are independent if $|x-y| \geq 2n$.} 
\eqe

There is a constant $c = c(\beta)  > 0$ such that for each $n\in\BB N$, there is a deterministic set $X_n \subset [1,2n] \cap\BB Z$ of cardinality at least $c n^\beta$ such that 
\eqb
\text{the intervals $\left[ x - \lfloor n^{1-\beta}  \rfloor , x + \lfloor n^{1-\beta} \rfloor \right]$ for $x\in X_n$ are disjoint and contained in $[1,2n]$}.  
\eqe
By~\eqref{eqn-dist-event-prob} and~\eqref{eqn-dist-event-ind},
\eqb \label{eqn-dist-event-intersect}
\BB P\left[ \text{$G_{\lfloor n^{1-\beta} \rfloor}(x)$ occurs for at least one $x\in X_n$} \right] \geq 1 -  \left(\frac{9}{10}\right)^{c n^\beta} .
\eqe

On the other hand, if $G_{\lfloor n^{1-\beta} \rfloor}(x)$ occurs, then the segment $\ol\ell$ as in the definition of $G_{\lfloor n^{1-\beta}\rfloor}$ has $\mathcal M[x -  \lfloor n^{1-\beta}  \rfloor , x +   \lfloor n^{1-\beta}  \rfloor]$-graph-distance diameter at least $ \lfloor n^{1-\beta}  \rfloor^{1/d-\beta}$ and the $\mathcal M$-graph distance from $\ol \ell$ to $\BB Z\setminus [x -  \lfloor n^{1-\beta}  \rfloor , x +  \lfloor n^{1-\beta}  \rfloor]$ is at least $ \lfloor n^{1-\beta}  \rfloor^{1/d-\beta }$.
Hence, the $\mathcal M$-graph-distance diameter of $\ol\ell$ is at least $ \lfloor n^{1-\beta}  \rfloor^{1/d-\beta}$. 
 
We now choose $\beta$ to be small enough, depending on $\zeta$, so that 
\eqbn
(1-\beta)(1/d-\beta)  > 1/d-\zeta .
\eqen
Then~\eqref{eqn-dist-event-intersect} and the preceding paragraph give that if $n$ is large enough, then with probability at least $1-(9/10)^{c n^\beta}$, there is a segment of a loop in $\Gamma$ which intersects $[1,2n]\cap\BB Z$ and has $\mathcal M$-graph-distance diameter at least $n^{1/d-\zeta}$. This gives the proposition statement for an appropriate choice of $a_0,a_1 > 0$. 
\end{proof}

\begin{proof}[Proof of Theorem~\ref{thm-macro}]
Let $(\mathcal M,\Gamma)$ be the loop-decorated planar map associated with an infinite meandric system. 
For $n\in\BB N$, let $F_n$ be the event that there is no arc of the upper or lower arc diagram associated with $\mathcal M$ which has one endpoint in $[1,2n]\cap\BB Z$ and one endpoint not in $[1,2n]\cap\BB Z$.
Equivalently, by~\eqref{eqn-rw-inf-adjacency}, $F_n$ is the event that the encoding walks satisfy $\mcl L_{2n} = \mcl R_{2n} = 0$ and $\mcl L_x , \mcl R_x \geq 0$ for each $x\in [0,2n]\cap\BB Z$. 
By a standard random walk estimate, there is a universal constant $c>0$ such that 
\eqb \label{eqn-excursion-prob}
\BB P[F_n] \sim c n^{-3 } \quad \text{as $n\rta\infty$}.
\eqe

By the definition of $F_n$, if $F_n$ occurs, then no infinite path in $\Gamma$ can hit $[1,2n]\cap\BB Z$ and the set $\Gamma_n$ of loops of $\Gamma$ which hit $   [1,2n] \cap\BB Z$ is the same as the set of loops in $\Gamma$ which do not hit any vertices in $ \BB Z \setminus [0,2n]$.  
Moreover, the conditional law of $(\mcl L , \mcl R)|_{[0,2n]\cap\BB Z}$ given $F_n$ is that of a pair of independent uniform $2n$-step simple random walk excursions. 
By the discussion surrounding~\eqref{eqn-arc-walk}, this implies that the conditional law of the  loop-decorated planar map $( M[1,2n] , \Gamma_n)$ given $F_n$ is that of the planar map associated with a uniform meandric system of size $n$. Hence, it suffices to show that if we condition on $F_n$, then except on an event of conditional probability decaying faster than any negative power of $n$, there is a loop in $\Gamma_n$ which has $\mathcal M[1,2n]$-graph-distance diameter at least $n^{1/d-\zeta}$. 

By Proposition~\ref{prop-plane-macro} and~\eqref{eqn-excursion-prob}, if $\beta,a_0,a_1$ are as in Proposition~\ref{prop-plane-macro}, then it holds with conditional probability at least $1-a_0 c^{-1}  n^3 e^{-a_1 n^\beta}$ given $F_n$ that there is a segment of a loop or an infinite path in $\Gamma$ which hits only vertices of $[1,2n]\cap\BB Z$ and has $\mathcal M$-graph-distance diameter (and hence also $\mathcal M[1,2n]$-graph-distance diameter) at least $n^{1/d-\zeta}$. By the first sentence of the preceding paragraph, on $F_n$ this segment is in fact a segment of a loop $\ell \in \Gamma_n$. This loop $\ell$ has $\mathcal M[1,2n]$-graph-distance diameter at least $n^{1/d-\zeta}$. Since $a_0 c^{-1}  n^3 e^{-a_1 n^\beta} = O(n^{-p})$ for every $p >0$, this concludes the proof. 
\end{proof}

\begin{proof}[Proof of Proposition~\ref{prop-map-diam}]
Theorem~\ref{thm-macro} immediately implies that except on an event of probability decaying faster than any negative power of $n$, the graph-distance diameter of $\mathcal M_n$ is at least $n^{1/d-\zeta}$. 

To prove an upper bound for the graph-distance diameter of $\mathcal M_n$, we first apply~\cite[Theorem 1.15]{ghs-dist-exponent}, which tells us that there exists an exponent $\chi  >0$ such that for each $\zeta\in (0,1)$ and each $n \in \BB N$, it holds except on an event of probability decaying faster than any negative power of $n$ that the graph-distance diameter of $\mcl G[1,2n]$ is at most $n^{\chi +\zeta/2}$. It was shown in~\cite[Theorem 3.1]{gp-dla} that $\chi = 1/d$.

By combining the preceding paragraph with Theorem~\ref{thm-ghs-coupling}, we get that for each $\alpha > 0$, there exists $A = A(\alpha) > 0$ such that with probability at least $1-O_n(n^{-\alpha})$, the graph-distance diameter of $\mathcal M[1,2n]$ is at most $A(\log n)^3 n^{1/d + \zeta/2}$, which is at most $n^{1/d+\zeta}$ if $n$ is large enough. 
Since $\alpha$ can be made arbitrarily large, we get that except on an event of probability decaying faster than any negative power of $n$, the graph-distance diameter of $\mathcal M[1,2n]$ is at most $n^{1/d+\zeta}$. 

To transfer this from $\mathcal M[1,2n]$ to $\mathcal M_n$, we define the event $F_n$ as in the proof of Theorem~\ref{thm-macro}, condition on $F_n$, and apply~\eqref{eqn-excursion-prob}, exactly as in the proof of Theorem~\ref{thm-macro}. 
\end{proof}

\section{Estimate for the mated-CRT map via SLE and LQG}
\label{sec-lqg}

To complete the proofs of our main results, it remains to prove Proposition~\ref{prop-mcrt-estimate}. We will do this using SLE and LQG. 

\subsection{SLE/LQG description of the mated-CRT map}
\label{sec-lqg-mcrt}

Recall that we previously defined the mated-CRT map using Brownian motion in Section~\ref{sec-mcrt}.
In this subsection we will give the SLE/LQG description of the mated-CRT map, which comes from the results of~\cite{wedges}. 
We will not need many properties of the SLE/LQG objects involved, so we will not give detailed definitions. Instead, we give precise references. 

Let $h$ be the random generalized function on $\BB C$ associated with the $\sqrt 2$-quantum cone. 
The generalized function $h$ is a minor variant of the whole-plane Gaussian free field; see~\cite[Definition 4.10]{wedges} for a precise definition. 
One can associate with $h$ a random locally finite measure $\mu_h$ on $\BB C$, the \textbf{$\sqrt 2$-LQG measure}, which is a limit of regularized versions of $e^{\sqrt 2 h} \, dx\,dy$, where $dx\,dy$ denotes Lebesgue measure on $\BB C$~\cite{kahane,shef-kpz}. The measure $\mu_h$ assigns positive mass to every open subset of $\BB C$ and zero mass to every point but is mutually singular with respect to Lebesgue measure. See~\cite[Chapter 2]{bp-lqg-notes} for a detailed account of the construction and properties of $\mu_h$.

One can similarly associate with $h$ a random metric (distance function) $D_h$ on $\BB C$, the \textbf{$\sqrt 2$-LQG metric}~\cite{dddf-lfpp,gm-uniqueness}. 
The metric $D_h$ induces the same topology on $\BB C$ as the Euclidean metric, but the Hausdorff dimension $d$ of the metric space $(\BB C , D_h)$ is strictly larger than 2 (this is the same $d$ appearing in Proposition~\ref{prop-map-diam}). See~\cite{ddg-metric-survey} for a survey of results about $D_h$. 

The metric measure space $(\BB C , D_h , \mu_h)$ possesses a scale invariance property which will be important for our purposes:
\eqb \label{eqn-cone-scale}
\text{$(\BB C , D_h , \mu_h) \eqD (\BB C , r^{1/d} D_h , r \mu_h)$ as metric measure spaces $\forall r>0$} .
\eqe
In fact, one has the following slightly stronger property: for each $r> 0$, there is a random $\rho_r > 0$ such that
\eqb \label{eqn-cone-scale'}
 \left(  D_h , \mu_h \right) \eqD \left(  r^{1/d} D_h(\rho_r \cdot , \rho_r\cdot)  , r \mu_h(\rho_r\cdot) \right) ,\quad\forall r  > 0 .
\eqe
The scaling property~\eqref{eqn-cone-scale'} follows from the scaling property of $h$~\cite[Proposition 4.13(i)]{wedges} together with the fact that adding the constant $\frac{1}{\sqrt 2} \log r$ to $h$ results in scaling $\mu_h$ by $r$ and $D_h$ by $r^{-1/d}$, both of which are immediate from the constructions of $\mu_h$ and $D_h$ (see the proof of~\cite[Proposition 2.17]{gs-lqg-minkowski} for a more detailed explanation).

Whole-plane SLE$_8$ from $\infty$ to $\infty$ is a random space-filling curve $\eta$ which travels from $\infty$ to $\infty$ in $\BB C$. 
It can be thought of as a two-sided version of chordal SLE$_8$ (see~\cite[Footnote 4]{wedges} for a precise version of this statement).
For each $z\in\BB C$, a.s.\ $z$ is hit exactly once by $\eta$, but there exist zero-Lebesgue measure sets of points which are hit twice or three times. 

Now suppose that we sample $\eta$ independently from the random generalized function $h$ above, then re-parametrize $\eta$ so that
\eqb
\eta(0) = 0 \quad \text{and} \quad \mu_h(\eta([a,b])) = b-a ,\quad\forall a,b\in\BB R \: \text{with} \: a < b. 
\eqe
The law of $\eta$ is invariant under spatial scaling (this is immediate from the definition in~\cite[Footnote 4]{wedges}), so it follows from~\eqref{eqn-cone-scale'} that
\eqb \label{eqn-sle-scale}
 \left(  D_h , \mu_h  ,\eta \right) \eqD \left(  r^{1/d} D_h(\rho_r \cdot , \rho_r\cdot)  , r \mu_h(\rho_r\cdot)  , \rho_r^{-1} \eta( \cdot /r) \right) ,\quad \forall r>0   .
\eqe

The connection between the pair $(h,\eta)$ and the mated-CRT map $\mcl G$ comes by way of the following theorem, which is a consequence of~\cite[Theorems 1.9 and 8.18]{wedges}.

\begin{thm} \label{thm-mating}
With $h$ and $\eta$ as above, let $\mcl G$ be the graph with vertex set $\BB Z$, with two distinct vertices $x,y\in\BB Z$ joined by an edge if and only if
\eqb
\eta([x-1,x])\cap \eta([y-1,y]) \not=\emptyset. 
\eqe
Then $\mcl G$ has the same law (as a graph) as the mated-CRT map as defined in~\eqref{eqn-mcrt-def}. 
\end{thm}
 
The mated-CRT map has some double edges, but we do not worry about such edge multiplicity in Theorem~\ref{thm-mating} since in this section we are only interested in graph distances.

\begin{remark}
The results and proofs in this section all carry over verbatim if we replace $\sqrt 2$-LQG by $\gamma$-LQG for $\gamma\in (0,2)$ and SLE$_8$ by space-filling SLE$_\kappa$ for $\kappa = 16/\gamma^2$. In this setting, the corresponding mated-CRT map is constructed from a pair of correlated Brownian motions with correlation $-\cos(\pi\gamma^2/4)$, instead of a pair of independent Brownian motions as in Section~\ref{sec-mcrt}; and the value of $d$ depends on $\gamma$. 
\end{remark}

\subsection{Proof of lower bounds for mated-CRT graph distances}
\label{sec-lqg-proof}

Henceforth assume that we are in the setting of Theorem~\ref{thm-mating}. Our goal is to prove Proposition~\ref{prop-mcrt-estimate}. To this end, we first prove a lemma that allows us to compare $D_h$-distances and graph distances in $\mcl G$. 

For $n\in\BB N$ and $z,w \in \eta([-n,n])\subset \mathbb C$, we slightly abuse notation by writing
\eqb \label{eqn-mcrt-pt-dist}
\op{dist}_{\mcl G[-n,n]}(z,w) := \min\left\{ \op{dist}_{\mcl G[-n,n]}(x,y) : x,y\in [-n,n]\cap\BB Z,  z\in \eta([x-1,x]) , w\in \eta([y-1,y]) \right\} .
\eqe
Note that a.s.\ for Lebesgue-a.e.\ point $z$, the cell $\eta([x-1,x])$ containing $z$ is unique (since a.s.\ $z$ is hit exactly once by $\eta$), and that $\op{dist}_{\mcl G[-n,n]}(z,w) = 0$ if $z$ and $w$ belong to the same cell. 

\begin{lem} \label{lem-mcrt-lqg}
Fix $\zeta \in (0,1)$. It holds with polynomially high probability as $n\rta\infty$ that
\eqb \label{eqn-mcrt-lqg}
\op{dist}_{\mcl G[-n,n]}(z,w) \geq n^{-\zeta} D_{h} \left( z,w  \right)   - 1    ,\quad\forall z,w\in \eta([-n,n]) .
\eqe
\end{lem}
\begin{proof} 
By~\cite[Proposition 3.13]{gs-lqg-minkowski}, for each $p > 0$ and each $t \in \BB R$, the $p$th moment of the random variable $\sup_{u,v \in \eta([t-1,t])} D_{h}(u,v)$ is bounded above by a finite constant which depends only on $p$ (not on $t$).  
Therefore, we can apply Chebyshev's inequality (for large positive moments) followed by a union bound over all $x\in [-n,n] \cap \BB Z$ to get that with superpolynomially high probability as $n\rta\infty$,
\eqb \label{eqn-mcrt-lqg-sup}
\max_{x\in [-n,n]\cap\BB Z}  \sup_{u,v\in \eta([x-1,x])} D_{h}(u,v) \leq n^\zeta .
\eqe  
Henceforth assume that~\eqref{eqn-mcrt-lqg-sup} holds.

Let $z,w \in \eta([-n,n])$ and let $N := \op{dist}_{\mcl G[-n,n]}(z,w)$. 
By the description of $\mcl G$ in Theorem~\ref{thm-mating}, there exists $x_0, x_1 , \dots , x_N \in [-n,n]\cap\BB Z$ such that the union of the SLE$_8$ segments $\eta([x_j-1,x_j])$ for $j = 0,\dots,N$ contains a path from $z$ to $w$. By~\eqref{eqn-mcrt-lqg-sup}, the $D_h$-diameter of each of these cells is at most $n^\zeta$. By the triangle inequality,  
\eqbn
D_{h}(z,w) \leq n^\zeta (N+1) 
\eqen
which is~\eqref{eqn-mcrt-lqg}. 
\end{proof}

The following lemma is the main LQG estimate needed for the proof of Proposition~\ref{prop-mcrt-estimate}.

\begin{lem} \label{lem-lqg-estimate}
For each $t > 0$, it holds with probability tending to 1 as $C\rta\infty$, uniformly over the choice of $t$, that
\eqb \label{eqn-lqg-across}
D_h\left( \eta([-t,t]) , \bdy \eta([-C t , C t]) \right) \geq  t^{1/d}  .
\eqe 
Moreover, for each fixed $C >1$ it holds with probability tending to 1 as $\delta \rta 0$, uniformly over the choice of $t$, that 
\eqb \label{eqn-lqg-around}
D_h\left(  (-\infty, 0] \cap \ol{\eta\left( [-C t , C t] \setminus [-t,t] \right)} ,  [0,\infty) \cap \ol{\eta\left( [-C t , C t] \setminus [-t,t] \right) }    \right)  \geq \delta t^{1/d}  
\eqe
where here we identify $\BB R$ with $\BB R\times\{0\} \subset\BB C$. 
\end{lem}
\begin{proof}
By the scale invariance property~\eqref{eqn-sle-scale}, it suffices to prove both~\eqref{eqn-lqg-across} and~\eqref{eqn-lqg-around} in the case when $t = 1$. 
We start with~\eqref{eqn-lqg-across}. 
Since $\eta$ a.s.\ fills all of $\BB C$ and LQG metric balls of finite radius are a.s.\ compact, a.s.\ there exists some $C  >1$ such that 
\eqb \label{eqn-sle-swallow}
\left\{z\in\BB C \,: \, D_h(z, \eta([-1,1]) ) \leq 1 \right\} \subset \eta([-C  , C ]) .
\eqe
Hence,~\eqref{eqn-sle-swallow} holds with probability tending to 1 as $C\rta\infty$. This shows that~\eqref{eqn-lqg-across} with $t=1$ holds with probability tending to 1 as $C \rta \infty$.

Since a.s.\ 0 is contained in the interior of $\eta([-1,1])$, for any fixed $C>1$ the compact sets 
\eqbn
 (-\infty, 0] \cap \ol{\eta\left( [-C   , C  ] \setminus [-1,1] \right)}  \quad \text{and} \quad  [0,\infty) \cap \ol{\eta\left( [-C  , C ] \setminus [-1,1] \right) }
\eqen
lie at positive Euclidean distance from each other. Moreover, since $D_h$ induces the Euclidean topology on $\BB C$, the two compact sets above lie at positive  $D_h$ distance from each other. Hence~\eqref{eqn-lqg-around} with $t =1$ holds with probability tending to 1 as $\delta \rta 0$.
\end{proof}

\begin{figure}[ht!]
	\begin{center}
		\includegraphics[width=0.75\textwidth]{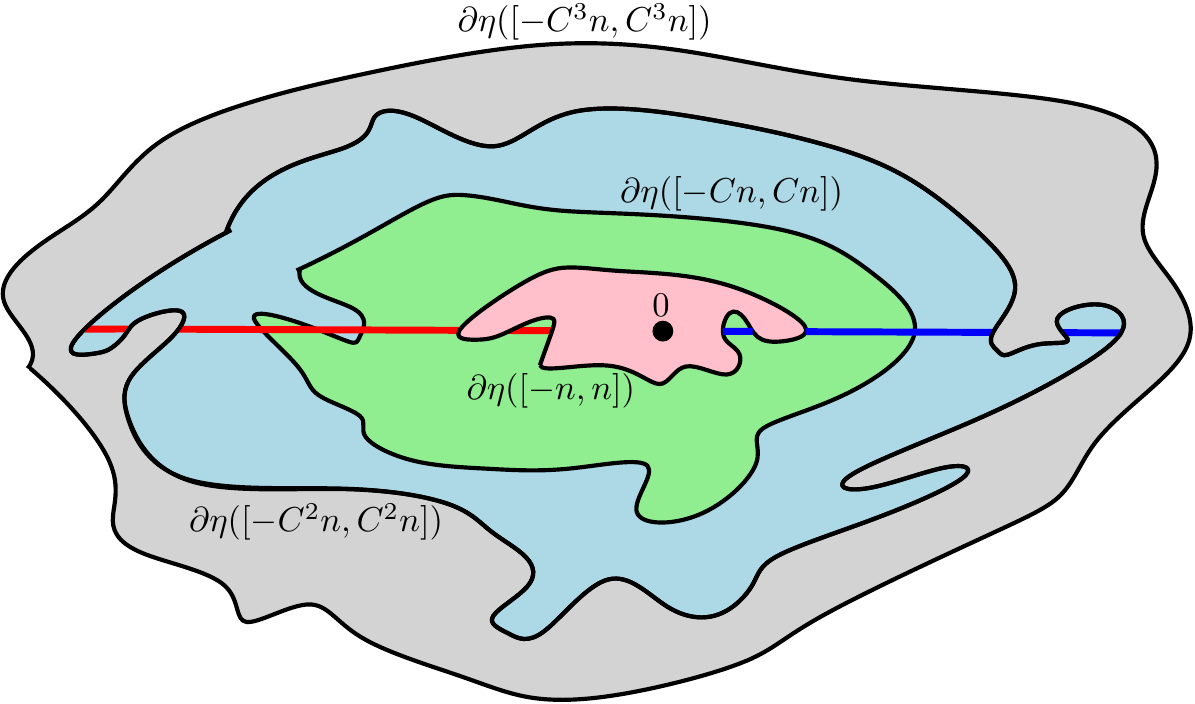}  
		\caption{\label{fig-lqg-estimate} Illustration of the SLE/LQG event used in the proof of Proposition~\ref{prop-mcrt-estimate}. On the event, the LQG distance between the inner and outer boundaries of the green, light-blue, and gray annular regions, as well as the LQG distance between the blue and red sets, are all bounded below. We note that the blue and red sets are not connected, but each of these sets contains a connected set whose closure intersects each of $\eta([-n,n])$ and $\ol{\BB C\setminus \eta([-C^2 n , C^2 n])}$. 
		}
	\end{center}
	\vspace{-3ex}
\end{figure}

\begin{proof}[Proof of Proposition~\ref{prop-mcrt-estimate}]
See Figure~\ref{fig-lqg-estimate} for an illustration.
By Lemma~\ref{lem-lqg-estimate}, we can find $C >1$ such that for each $t > 0$,~\eqref{eqn-lqg-across} holds with probability at least $1-(1-p)/5$. 
Applying this with $t = n$, $t = C n$, and $t = C^2 n$ shows that with probability at least $1-3(1-p)/5$,
\allb \label{eqn-use-lqg-across}
&D_h\left( \eta([-n,n]) , \bdy \eta([-C n , C n]) \right) \geq  n^{1/d}   ,\quad 
D_h\left( \eta([-C n, C n]) , \bdy \eta([-C^2 n , C^2 n]) \right) \geq  (C n)^{1/d}   ,\quad \text{and} \notag\\
&D_h\left( \eta([-C^2 n, C^2 n]) , \bdy \eta([-C^3 n , C^3 n]) \right) \geq  (C^2 n)^{1/d}  . 
\alle 
By~\eqref{eqn-lqg-around} of Lemma~\ref{lem-lqg-estimate} (applied with $C^2$ instead of $C$ and $t =n$), we can find $\delta \in (0,1)$ (depending on $C$) such that with probability at least $1-(1-p)/5$, 
\eqb \label{eqn-use-lqg-around}
D_h\left(  (-\infty, 0] \cap \ol{\eta\left( [-C^2 n  , C^2 n ] \setminus [- n ,  n] \right)} ,  [0,\infty) \cap \ol{\eta\left( [-C^2 n , C^2 n] \setminus [-n,n] \right)    } \right)  \geq \delta n^{1/d} .
\eqe
By Lemma~\ref{lem-mcrt-lqg} (applied with $C^3 n$ in place of $n$ and $\zeta/2$ in place of $\zeta$), for each large enough $n\in\BB N$ it holds with probability at least $1-(1-p)/5$ that
\eqb \label{eqn-use-mcrt-lqg}
\op{dist}_{\mcl G[-C^3 n, C^3 n]}(z,w) \geq n^{-\zeta/2} D_{h} \left( z,w  \right)   - 1    ,\quad\forall z,w\in \eta([-C^3 n, C^3 n]) .
\eqe

Henceforth assume that~\eqref{eqn-use-lqg-across}, \eqref{eqn-use-lqg-around}, and \eqref{eqn-use-mcrt-lqg} all occur, which happens with probability at least $p$ if $n$ is large enough. 
Recall that two vertices $x,y\in \BB Z$ are joined by an edge of $\mcl G$ if and only if $\eta([x-1,x])\cap\eta([y-1,y]) \not=\emptyset$. 
Hence each path in $\mcl G$ from $[-n,n]\cap\BB Z$ to $ \BB Z\setminus [-C n , C n] $ has a sub-path which is contained in $[-C n , C n]\cap\BB Z$ and which goes from $[-n,n]\cap\BB Z$ to a vertex $x\in \BB Z$ whose corresponding cell $\eta([x-1,x])$ intersects $\bdy\eta([-C n ,Cn])$. 
Therefore, the first inequality in~\eqref{eqn-use-lqg-across} together with~\eqref{eqn-use-mcrt-lqg} implies~\eqref{item-mcrt-in} in the lemma statement (provided $n$ is large enough that $n^{1/d-\zeta/2}-1\geq n^{1/d-\zeta}$). 
We similarly obtain~\eqref{item-mcrt-out} and~\eqref{item-mcrt-out'} from the second and third inequalities in~\eqref{eqn-use-lqg-across}.

To get~\eqref{item-mcrt-around}, let $\Pi_1^0$ (resp.\ $\Pi_2^0$) be the set of $x\in\BB Z$ such that the cell $\eta([x-1,x])$ intersects $(-\infty,0]\cap \ol{\eta\left( [-C^2 n  , C^2 n ] \setminus [- n ,  n] \right)}$  (resp.\ $[0,\infty) \cap \ol{\eta\left( [-C^2 n  , C^2 n ] \setminus [- n ,  n] \right)}$). 
Then each of $\Pi_1^0$ and $\Pi_2^0$ contains a connected subset of $\BB Z$ which intersects both $[-n,n]\cap\BB Z$ and $\BB Z\setminus [-C^2 n , C^2 ]$. So, we can find a path $\Pi_1$ (resp.\ $\Pi_2$) in $\mcl G$ from $[-n,n]\cap\BB Z$ to $\BB Z\setminus [-C^2 n , C^2 ]$ which is contained in $\Pi_1^0$ (resp.\ $\Pi_2^0$) and which is contained in $[-C^2 n , C^2 n]\cap\BB Z$ except for its terminal endpoint.  

By~\eqref{eqn-use-lqg-around} and~\eqref{eqn-use-mcrt-lqg}, the $\mcl G[-C^3 n , C^3 n]$-graph distance between $\Pi_1$ and $\Pi_2$ is at least $\delta n^{1/d-\zeta/2} - 1$, which is at least $n^{1/d-\zeta}$ if $n$ is large enough. 
Furthermore, by~\eqref{item-mcrt-out'} from the proposition statement (which we have already proven), the $\mcl G$-graph distance from each of $\Pi_1$ and $\Pi_2$ to $\BB Z\setminus [-C n^3 , C n^3]$ is at least $n^{1/d-\zeta}$. Hence, any path in $\mcl G$ from $\Pi_1$ to $\Pi_2$ has to have length at least $n^{1/d-\zeta}$. 
This gives~\eqref{item-mcrt-around}.
\end{proof}

\section{Proofs for the UIHPMS}
\label{sec-half-plane-proof}

This section is organized as follows. In Section~\ref{sec-uihpms-equiv}, we prove the equivalence of our two definitions of the UIHPMS (Lemma~\ref{lem-uihpms-equiv}). 
The rest of the section is devoted to the proofs of Theorem~\ref{thm-infinite-path-half} and Proposition~\ref{prop-alternating}. 
We start out in Section~\ref{sec-aperiodic} by introducing a special sequence of boundary points for the UIHPMS, $\{H_m\}_{m\in\BB Z}\subset\{J_k\}_{k\in\BB Z}$. 
The points $\{H_m\}_{m\in\BB Z}$ may be viewed as marked points of a minor variant of the UIHPMS that satisfies an invariance property under translation by $H_m$ as in~\eqref{eqn-bdy-translate}, but (crucially) without the constraint that $m$ is even (Lemma~\ref{lem-good-translate}). This property will be important for our purposes since we will eventually need to use some parity arguments.

Sections~\ref{sec-path-finiteness} and~\ref{sec-infinite-arcs} constitute the core part of the argument. In Section~\ref{sec-path-finiteness}, we prove that a.s.\ there are no semi-infinite paths in the UIHPMS that start at one of the special boundary points $H_m$ and never hit any other special boundary point (Lemma~\ref{lem-path-finiteness}). This is done using a Burton-Keane style argument, combined with the existence of certain special times for the pair of encoding walks $(\mcl L ,|\mcl R|)$ (Lemma~\ref{lem-block}). 

The non-existence of such semi-infinite paths allows us to define a perfect matching on $\BB Z$ by saying that $m\in\BB Z$ is matched to $m'\in\BB Z$ if and only if the path $P_m$ of arcs started from $H_m$ hits $H_{m'}$ before hitting any other special boundary point (see Definition~\ref{def-bdy-path} for a precise definition of this path). In Section~\ref{sec-infinite-arcs}, we prove a general lemma for ergodic perfect matchings on $\BB Z$ (Lemma~\ref{lem-ergodic-matching}) which in our setting implies that there are infinitely many paths between good boundary points which disconnect 0 from $\infty$ in the UIHPMS (Lemma~\ref{lem-infinite-arcs}). As explained in Section~\ref{sec-uihpms-proof}, these paths act as ``shields'' which cannot be crossed by any infinite path in the UIHPMS. This allows us to prove Theorem~\ref{thm-infinite-path-half} and Proposition~\ref{prop-alternating}.

Throughout this section, we assume that the arcs of the UIMS and the UIHPMS are drawn in such a way that each arc joining $x < y$ is contained in $[x,y] \times \BB R$. With this convention, it is easy to see from either of the two definitions of the UIHPMS from Section~\ref{sec-half-plane} that for each boundary point $J_k$  there is no arc which crosses the downward ray $\{J_k\} \times (-\infty,0]$.

\subsection{Equivalence of the definitions of the UIHPMS}
\label{sec-uihpms-equiv}

Recall the two definitions of the UIHPMS from Section~\ref{sec-half-plane}, one by cutting the UIMS and one in terms of a simple random walk and an independent reflected random walk. 
We now prove that the two definitions are equivalent. 

We use a discrete version of celebrated L\'evy's theorem~\cite{simons1983discrete}, which we now recall.
Let $\mcl X=(\mcl X_n)_{n\geq0}$ be a (one-sided) simple random walk with $\mcl X_0=0$ and $M^{\mcl X}$ be the running minimum process of $\mcl X$ from time $0$, that is 
\eqb \label{eqn-runmin}
M^{\mcl X}_n = \min_{k\in [0,n]} \mcl X_k, \qquad \text{for } n\ge 0.
\eqe
Then 
\eqb \label{eqn-discrete-levy}
    (\mcl X - M^{\mcl X}, -M^{\mcl X}) \stackrel{d}{=} (| \mcl X|-\textbf{1}_{\{\mcl X>0\}}, \ell^{\mcl X})
\eqe
where $\ell^{\mcl X}_n$ denotes the number of times that $\mcl X$ crosses $1/2$ during the time interval $[0,n]$, for all $n \geq 0$.

\begin{proof}[Proof of Lemma~\ref{lem-uihpms-equiv}]
Let $(\mcl L , \mcl R)$ be the pair of independent two-sided simple random walks on $\BB Z$ used to construct the UIMS as in Section~\ref{sec-infinite}. We assume that $\mcl L_0 = \mcl R_0 = 0$. 
We first analyze the cutting description of the UIHPMS. 
For each integer $k\in\BB Z_{>0}$, let $T_k$ be the smallest time $x \in \BB Z_{>0}$ such that $\mcl R_x = -k$. Also, let $T_{-k}$ be the largest time $x \in \BB Z_{<0}$ such that $\mcl R_x = -k$. 
By~\eqref{eqn-rw-inf-adjacency}, for each $k \in\BB Z_{>0}$ the point $T_k$ is joined to $T_{-k} + 1$ by a lower arc of the UIMS . Furthermore, arcs of this type are the only ones which cross $\{1/2\} \times (-\infty, 0]$. Therefore, the UIHPMS under the cutting definition is obtained from the UIMS by first removing the lower arcs joining $T_k$ and $T_{-k}+1$ for $k \in\BB Z_{>0}$, then adding lower arcs joining $T_{2k-1}$ and $T_{2k}$ for each $k \in\BB Z_{>0}$; and lastly joining $T_{2k+1}+1$ and $T_{2k} + 1$ for each $k \in\BB Z_{<0}$. 

Let $\wt{\mcl R}$ be the random walk  
on $\BB Z $ obtained from $\mcl R$ by replacing each of the downward steps $\mcl R_{T_{2k-1}} - \mcl R_{T_{2k-1}-1}$ for $k\in\BB Z_{>0}$ by an upward step, and replacing each of the upward steps $\mcl R_{T_{2k+1}+1} - \mcl R_{T_{2k+1}}$ for $k\in\BB Z_{<0}$ by a downward step (and otherwise leaving the steps of $\mcl R$ unchanged). Then the cutting description of the UIHPMS is equivalent to the random walk description~\eqref{eqn-rw-inf-adjacency} with $\wt{\mcl R}$ in place of $\mcl R$. Therefore, it suffices to show that $\wt{\mcl R} \eqD |\mcl R|$. 

Let $M^{\mcl R}$ and $\ell^{\mcl R}$ be as in~\eqref{eqn-runmin}~and~\eqref{eqn-discrete-levy}, with $\mcl X = \mcl R\mid_{[0,\infty)\cap\BB Z}$. Then for $x \in\BB Z_{\geq 0}$, we have
\eqb
\wt{\mcl R}_x = \mcl R_x - M_x^{\mcl R} + \textbf{1}_{\{ \text{$M_x^{\mcl R}$ is odd} \}}  .
\eqe
On the other hand, the number $\ell_x^{\mcl R}$ of crossings of $1/2$ by $\mcl R$ during the time interval $[0,x]$ is odd if and only if $\mcl R_x > 0$, so
\eqb
|\mcl R|_x = |\mcl R|_x -  \textbf{1}_{\{ \mcl R_x > 0 \}}  +    \textbf{1}_{\{ \text{$\ell_x^{\mcl R}$ is odd} \}}.
\eqe
By combining the preceding two identities with~\eqref{eqn-discrete-levy}, we get that $\wt{\mcl R}\mid_{[0,\infty)\cap\BB Z} \eqD |\mcl R| \mid_{[0,\infty)\cap\BB Z}$. We similarly get the desired equality in law for negative times.
\end{proof}

\subsection{Aperiodicity via good boundary points}
\label{sec-aperiodic}
 
In this subsection we address the following technical point. 
The law of the UIHPMS is only invariant under \emph{even} translations along the boundary, i.e., translations by $J_{2k}$ for $k\in\BB Z$; see~\eqref{eqn-bdy-translate}. 
In one step of our proof of Theorem~\ref{thm-infinite-path-half} (Section~\ref{sec-infinite-arcs}) we will use an ergodicity argument which requires us to know that a certain event for boundary points of odd index has the same probability as the corresponding event for boundary points of even index. 
For this reason, we need to introduce a variant of the UIHPMS where we have translation invariance for \emph{all} boundary points, not just even boundary points. The idea is to replace the reflected random walk $|\mcl R|$ by a walk which has some constant steps at height 0. This is similar to how one can make a Markov chain aperiodic by introducing constant steps. 

\begin{figure}[ht!]
	\begin{center}
		\begin{minipage}[c]{.8\textwidth}
			\includegraphics[width=\textwidth]{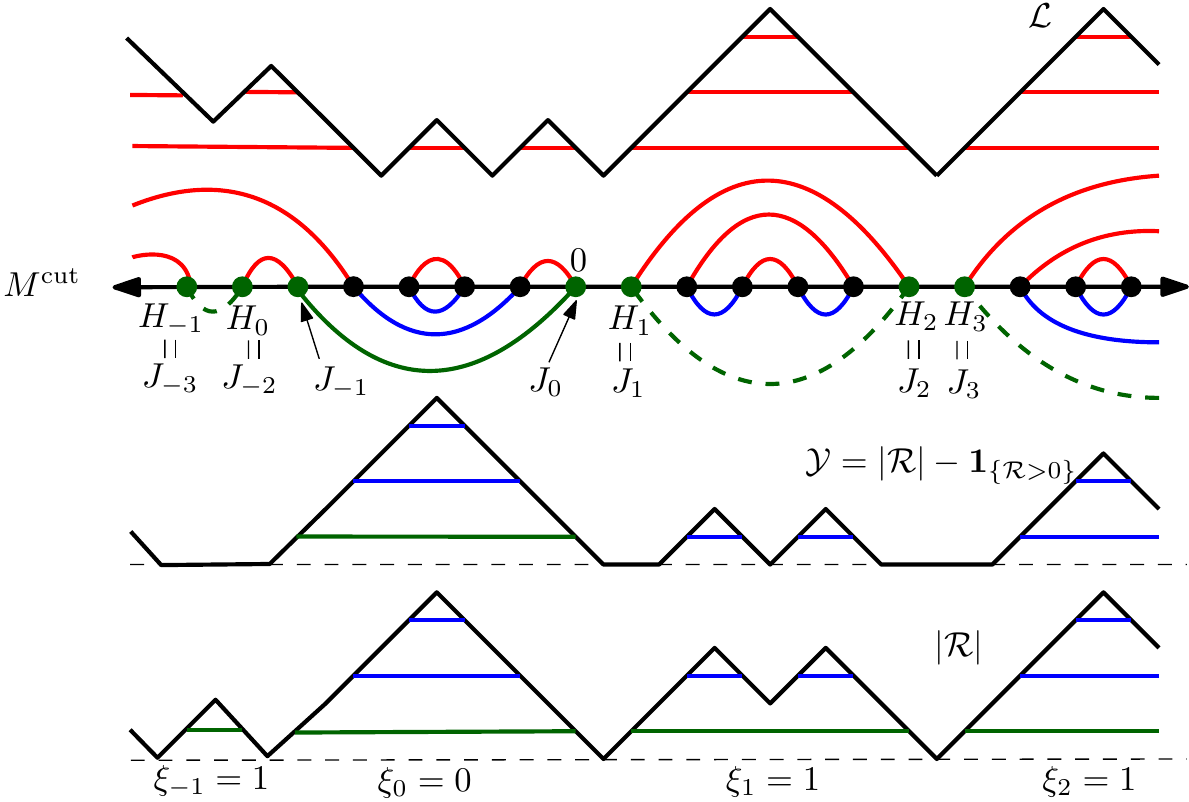}
		\end{minipage}
		\caption{\label{fig-good-bdy-def} Illustration of the variant $\mcl M^{\op{cut}}$ of the UIHPMS $\mcl M'$ constructed using the walk $\mcl Y$ of~\eqref{eqn-good-walk}. Dashed arcs are the ones which are part of the original UIHPMS $\mcl M'$ but not part of $\mcl M^{\op{cut}}$. The graph of the reflected random walk $|\mcl R|$ used in the construction of the UIHPMS is shown at the bottom of the figure. Below each excursion of $|\mcl R|$ we have also shown the outcomes of the i.i.d.\ Bernoulli$(1/2)$ random variables $\{\xi_k\}_{k\in\BB Z}$.
		}
	\end{center}
	\vspace{-3ex}
\end{figure}

Let $\mcl R$ be the two-sided simple random walk on $\BB Z$ such that $|\mcl R|$ is used to construct the UIHPMS. 
We define the modified walk 
\eqb \label{eqn-good-walk}
\mcl Y_x := |\mcl R|_x - \textbf{1}_{\{\mcl R_x > 0\}}. 
\eqe
For $m \in\BB Z$, let
\eqb \label{eqn-good-bdy-def}
  H_m
:= \left(\text{$m$th smallest $x\in\BB Z$ such that $\mcl Y_x = \mcl Y_{x-1} = 0$}\right) ,
\eqe
with the numbering chosen so that $0 \in [ H_0  ,  H_1-1]\cap\BB Z$. 

Recall that we denote by $J_k$ (resp.\ $J_{-k+1}$) the $k$th positive (resp.\ non-positive) boundary point of the UIHPMS. In particular, for each $k\in\BB Z$, $J_{2k}=J_{2k+1}-1$ and $\mcl R_{J_{2k}}=\mcl R_{J_{2k+1}-1}=0$; and these are all the zeros of $\mcl R$. 
Conditional on $|\mcl R|$, the signs of the excursions $\{\mcl R|_{[J_{2k-2}  , J_{2k } ]}\}_{k\in\BB Z}$ of $|\mcl R|$ away from 0 are i.i.d.\ Bernoulli$(1/2)$ random variables. Therefore, $\mcl Y$ can equivalently be constructed from $|\mcl R|$ as follows. Start with a collection $\{\xi_k\}_{k\in\BB Z}$ of i.i.d.\ Bernoulli$(1/2)$ random variables independent from $|\mcl R|$. Then, for each $k\in\BB Z$ such that $\xi_k = 1$, replace the steps $|\mcl R|_{J_{2k}} - |\mcl R|_{J_{2k}-1}$ and $|\mcl R|_{J_{2k-1}} - |\mcl R|_{J_{2k-1}-1}$ by constant steps. 
In particular, 
\eqb \label{eqn-good-bdy-compare}
\{  H_m \}_{m\in\BB Z} \subset \{J_k\}_{k \in \BB Z }\quad \text{and} \quad \BB P\left[J_k \in \{  H_m \}_{m\in\BB Z} \,|\, |\mcl R| \right] = \frac12 ,\quad\forall k \in \BB Z.
\eqe 
We call $H_m$ the $m$th \textbf{good boundary point} of the UIHPMS. See Figure~\ref{fig-good-bdy-def} for an illustration. 

Let $(\mcl M^{\op{cut}} , \Gamma^{\op{cut}})$ be the infinite planar map decorated by a collection of loops and paths which is defined exactly as in~\eqref{eqn-rw-inf-adjacency} but with $\mcl Y$ in place of $\mcl R$, with the convention that if $\mcl Y_x = \mcl Y_{x-1} = 0$, then there is no lower arc incident to $x\in\BB Z$. The discussion just above implies that $(\mcl M^{\op{cut}} , \Gamma^{\op{cut}})$ can be obtained from the UIHPMS by removing each lower arc joining boundary points $J_{2k-1}$ and $J_{2k}$ for $k\in\BB Z$ such that $\xi_k = 1$. The good boundary points are exactly these points $J_{2k-1}$ and $J_{2k}$ for $k\in\BB Z$ such that $\xi_k = 1$. In particular, they are not determined by the UIHPMS.

One can also see from the proof of the equivalence of the two definitions of the UIHPMS (Lemma~\ref{lem-uihpms-equiv}) that $(\mcl M^{\op{cut}} , \Gamma^{\op{cut}})$ has the same law as the decorated planar map obtained from the UIMS by removing all of the lower arcs which intersect $\{1/2\} \times (-\infty,0]$. Moreover, in this interpretation, the good boundary points are just the endpoints of such removed arcs. 

The following lemma is our main reason for introducing the objects in this subsection. It will be used in the proof of Lemma~\ref{lem-infinite-arcs} below.

\begin{lem} \label{lem-good-translate}
Let $\wh{\mcl Y}$ be sampled from the conditional law of $\mcl Y$ given the event 
\eqb 
\{ \mcl Y_0 = \mcl Y_{-1}  = 0 \} = \{H_0 = 0\} .
\eqe 
Let $(\wh{\mcl M}^{\op{cut}} , \wh\Gamma^{\op{cut}})$ be defined in the same manner as $(\mcl M^{\op{cut}} , \Gamma^{\op{cut}})$ with $\wh{\mcl Y}$ in place of $\mcl Y$ and let $\{\wh H_m\}_{m\in \BB Z}$ be as in~\eqref{eqn-good-bdy-def} with $\wh{\mcl Y}$ in place of $\mcl Y$. 
Then
\eqb \label{eqn-good-translate}
\left( \wh{\mcl M}^{\op{cut}} - \wh H_m , \wh\Gamma^{\op{cut}} - \wh H_m \right) 
\eqD \left( \wh{\mcl M}^{\op{cut}}  , \wh\Gamma^{\op{cut}}   \right) ,\quad\forall m \in \BB Z.
\eqe 
Furthermore, for each positive integer $r\in\BB Z_{>0}$, each event which is invariant under translations of the form~\eqref{eqn-good-translate} with $m$ restricted to lie in $r\BB Z$ has probability zero or one. 
\end{lem} 
\begin{proof}
Since $(\wh{\mcl M}^{\op{cut}} , \wh\Gamma^{\op{cut}})$ is constructed from $\wh{\mcl Y}$ and an independent two-sided simple random walk $\mcl L$ in the manner of~\eqref{eqn-rw-inf-adjacency}, it suffices to show that
\eqb \label{eqn-good-translate-walk}
\wh{\mcl Y}_{\wh H_m +\, \cdot} \eqD \wh{\mcl Y}_{\cdot} \, ,\quad\forall m \in \BB Z,
\eqe
and that any event which is invariant under translations of the form~\eqref{eqn-good-translate-walk} with $m$ restricted to lie in $r\BB Z$ has probability zero or one. 

To this end, we first consider the unconditioned random walk $\mcl Y=|\mcl R| - \textbf{1}_{\{\mcl R > 0\}}$ from~\eqref{eqn-good-walk}. 
Let $\{\tau_j\}_{j\in\BB Z}$ be the times such that $\mcl Y_x = 0$, numbered from left to right in such a way that $\tau_0 = 0$. 

Let $(\mcl X , M^{\mcl X})$ be a one-sided simple random walk on $\BB Z$ started from $0$ at time $0$ and its running minimum process, as in~\eqref{eqn-runmin}.
Also let $\mu$ be the law of $\mcl X - M^{\mcl X}$ stopped at the first positive time when it reaches zero. 
A sample from $\mu$ can be produced as follows: Flip a fair coin. If the coin comes up heads, we run a simple random walk started from 0, conditioned on the event that its first step is upward, until the first positive time when it reaches zero. If the coin comes up tails, we instead take the path which takes one step from 0 to 0. 
From this description, we get that the law of a path sampled from $\mu$ is invariant under time reversal. 

By the discrete version of L\'evy's theorem~\eqref{eqn-discrete-levy}, the process $\mcl Y |_{[0,\infty)\cap\BB Z}$ has the same law as $\mcl X -  M^{\mcl X}$. By the strong Markov property of $\mcl X$, we get that the excursions $\mcl Y|_{[\tau_{j-1} , \tau_j] \cap\BB Z}$ for $j \geq 1$ are i.i.d.\ samples from $\mu$. 
The law of $\mcl Y$ is invariant under time reversal, so the increments $\mcl Y|_{[\tau_{j-1} , \tau_j]\cap\BB Z}$ for $j \leq 0$ are also i.i.d., and have the same law as the time reversal of a sample from $\mu$. 
From the time reversal symmetry of $\mu$, we get that the law of the whole collection of increments $\mcl Y|_{[\tau_{j-1} , \tau_j]\cap\BB Z}$ for $j \in\BB Z$ are i.i.d.\ samples from $\mu$.

The points $H_m$ for $m\in\BB Z$ coincide precisely with the times $\tau_j$ for which $\tau_j =  \tau_{j-1} + 1 $. 
From this and the above description of the increments $\mcl Y|_{[\tau_{j-1} , \tau_j]\cap\BB Z}$, we infer that the increments $\wh{\mcl Y}|_{[\wh H_{m-1}  , \wh H_m] \cap\BB Z}$ for $m\in\BB Z$ are i.i.d.\ samples from the law of $\mcl Y |_{[0,\infty)\cap\BB Z}$ stopped at the first positive time $x$ such that $\mcl Y_x = \mcl Y_{x-1} = 0$. 
From this and the fact that $\wh{\mcl Y}$ is sampled from the conditional law of $\mcl Y$ given the event $\{\mcl Y_0 = \mcl Y_{-1}  = 0 \}$, the translation invariance property~\eqref{eqn-good-translate-walk} is immediate. Furthermore, the desired ergodicity property for $\wh{\mcl Y}$ follows from the zero-one law for translation invariant events depending on a sequence of i.i.d.\ random variables, applied to the i.i.d.\ random variables $\wh{\mcl Y}|_{[\wh H_{r n  }  , \wh H_{r(n+1)}] \cap\BB Z}$ for $n\in\BB Z$. 

We remark that one could also prove~\eqref{eqn-good-translate-walk} by showing that if $U_m$ is sampled uniformly from $[-m,m]\cap\BB Z$, then the law of $ \mcl Y_{H_{U_m}+\cdot}$ converges as $m\to\infty$ to the law of $\wh{\mcl Y}$. We will not give the details. 
\end{proof}


\subsection{No semi-infinite paths started from the boundary}
\label{sec-path-finiteness}

Define the UIHPMS and its left-to-right ordered sequence of boundary points $\{J_k\}_{k\in\BB Z} \subset \BB Z$ as in Section~\ref{sec-half-plane}.
Also define the left-to-right ordered sequence of good boundary points $\{ H_m\}_{m\in\BB Z}$  as in~\eqref{eqn-good-bdy-def}.

\begin{defn} \label{def-bdy-path} 
For each good boundary point $H_m$ with $m\in \BB Z$, let $P_m$ be the unique directed path of arcs in the UIHPMS starting from $H_m$, following the arc in the \emph{upper} half-plane incident to $H_m$, and ending at the first good boundary point other than $H_m$ which is hit by the path, if it exists; otherwise let $P_m$ be the whole semi-infinite path started from $H_m$. 
We call $P_m$ the \textbf{boundary path} started from $H_m$.
\end{defn}

By definition, the boundary path $P_m$ hits good boundary points of the UIHPMS only at its endpoints. However, it is in principle possible that for some values of $m$, the path $P_m$ never hits another good boundary point other than $H_m$, in which case it is semi-infinite.
The first step in the proof of Theorem~\ref{thm-infinite-path-half} is to rule this event out. 

\begin{lem}\label{lem-path-finiteness} 
Almost surely, the boundary path $P_m$ is finite for all $m \in \BB Z$.
\end{lem}

The proof of Lemma~\ref{lem-path-finiteness} proceeds as follows. By even translation invariance~\eqref{eqn-bdy-translate}, the probability that $J_k$ is a good boundary point (i.e., $J_k = H_m$ for some $m\in\BB Z$) and $P_m$ is semi-infinite depends only on the parity of $k$. If this probability is positive for even values of $k$, say, then by the ergodic theorem a.s.\ $P_{ m}$ is semi-infinite for a positive fraction of the indices $m \in \BB Z$. We will show that this cannot be the case by a Burton-Keane style argument. Roughly speaking, we will use certain special times for the encoding walks (Lemma~\ref{lem-block}) to argue that there is not enough ``room'' for there to be a positive density of values of $m$ for which $P_{ m}$ is semi-infinite. We start with some bounds on the probability that a bridge has at least some number of crossings of $1/2$.

\begin{lem}\label{lem-walk-estimate}
    For a simple random walk $\mcl X$ on $\BB Z$ with $\mcl X_0=0$, recall that $\ell^{\mcl X}_n$ denotes the number of times that $\mcl X$ crosses $1/2$ during the time interval $[0,n]$. For $n, k>0$, we have
    \eqb \label{eqn-walk-estimate}
    \frac{2^{2k-1}\binom{2n-2k}{n}}{\binom{2n}{n}}\le
    \BB P\left[\ell^{\mcl X}_{2n}\ge 2k \;\middle|\; \mcl X_{2n}=0\right]
    \le \frac{2^{k}\binom{2n-k}{n}}{\binom{2n}{n}}.
    \eqe 
    As a particular consequence, for each $\ep>0$ and each integer $n > 0$,
    \eqb\label{eqn-estimate-limit}
    C_1\leq \BB P\left[\ell^{\mcl X}_{2n}\ge \ep \sqrt{n} \;\middle|\; \mcl X_{2n}=0\right]\leq C_2 \quad\text{and}\quad C_1 \leq \BB P\left[\ell^{\mcl X}_{2n}\ge \ep \sqrt{n} \;\middle|\; (\mcl X_{2n-1},\mcl X_{2n})=(1,0)\right] \leq C_2,
    \eqe
    for constants $C_1 , C_2 \in (0,1)$ only depending on $\ep$.
\end{lem}

\begin{proof}
    We call a simple random walk path starting and ending at $0$ an \emph{excursion} if it does not hit $0$ except at its two endpoints (we allow excursions to be positive or negative). If $\mcl X_{2n} = 0$, we can decompose the path of $\mcl X\mid_{[0,2n] \cap \BB Z}$ into several excursions.
    The number $\ell^{\mcl X}_{2n}$ is exactly twice the number of positive excursions. Hence, if $\mcl X_{2n}=0$ and $\ell^{\mcl X}_{2n}\ge 2k$, then the path of $\mcl X\mid_{[0,2n] \cap \BB Z}$ can be decomposed into at least $k$ excursions. Therefore, we compute $\BB P\left[\mcl X\mid_{[0,2n] \cap \BB Z} \text{has at least $k$ excursions}\;\middle|\; \mcl X_{2n}=0\right]$ in the next paragraph to obtain an upper bound for $\BB P\left[\ell^{\mcl X}_{2n}\ge 2k \;\middle|\; \mcl X_{2n}=0\right]$.

    Assume that the path of $\mcl X\mid_{[0,2n]\cap \BB Z}$ with $\mcl X_{2n}=0$ can be decomposed into at least $k$ excursions. By making the last $k$ excursions all positive (by flipping if necessary) and removing the last downward step of each of these $k$ excursions, we obtain a simple walk from $0$ to $k$ with $2n-k$ steps, i.e. with $n$ upward steps and $n-k$ downward steps. The total number of such walks is $\binom{2n-k}{n}$. In fact, this defines a $2^{k}$-to-$1$ map, as the modified positive excursions can be recovered from the resulting walk from $0$ to $k$ (the last visit of $j\in[1,k]\cap\BB Z$ corresponds to where a downward step was removed), and there are $2^{k}$ ways to assign a sign for each excursion to obtain $\mcl X\mid_{[0,2n]}$. As there are $\binom{2n}{n}$ possible walks from $0$ to $0$ with $2n$ steps, the upper bound in~\eqref{eqn-walk-estimate} follows.
    
   
 The lower bound of~\eqref{eqn-walk-estimate} is obtained via a similar argument. It suffices to notice that a simple walk from $0$ to $2k$ with $2n-2k$ steps can be turned (as done above) in a walk from 0 to 0 with $2n$ steps and at least $k$ positive excursions among the last $2k$ ones in at least $2^{2k-1}$ ways (by symmetry). 
	
	The first bound in~\eqref{eqn-estimate-limit} is obtained by setting $k = \lfloor \ep\sqrt n/2 \rfloor$ in~\eqref{eqn-walk-estimate} and applying Stirling's formula.
    To add the extra condition $\mcl X_{2n-1}=1$ in the second bound in~\eqref{eqn-estimate-limit}, we just fix the last excursion to be positive in the excursion decomposition defined above, which gives a similar bound.
\end{proof}

Let $\mcl L$ and $\mcl R$ be the independent two-sided simple random walks on $\mathbb Z$ with $\mcl L_0 = \mcl R_0 = 0$ used in the definition of the UIHPMS from~\eqref{eqn-rw-inf-adjacency-half}.   
We call $x\in \BB Z_{>0}$ an \textbf{upper block} if $x$ is linked to some point in $(-\infty, 0]\cap \BB Z$ by an arc above the real line. If $x$ is an upper block, no arc above the real line connects $(0,x)\cap \BB Z$ and $(x,\infty)\cap \BB Z$; otherwise such an arc would have to cross the upper arc incident to $x$.
In terms of the random walk description, $x\in \BB Z_{>0}$ is an upper block if and only if $\mcl L\mid_{[0,x]\cap \BB Z}$ attains its \emph{unique} minimum value at time $x$.

Recall that for $k>0$ we denoted by $J_k$ the $k$th positive boundary point of the UIHPMS. We call $x \in 2\BB Z_{>0}$ a \textbf{block} if $x$ is an upper block and $x=J_{2k}$ for some $k>0$. That is, $x \in 2\BB Z_{>0}$ is a block if and only if $|\mcl R|_x = 0$ and $\mcl L\mid_{[0,x]\cap\BB Z}$ attains its unique minimum value at time $x$. See Figure~\ref{fig-block} for an example. Our motivation for the definition of a block is that if $x$ is a block, then no path of arcs started at a point in $(0,x)\cap\BB Z$ can cross the vertical line $\{x\} \times \BB R$ without first hitting $(-\infty,0] \cap\BB Z$.

\begin{figure}[ht!]
	\begin{center}
		\begin{minipage}[c]{.7\textwidth}
			\includegraphics[width=\textwidth]{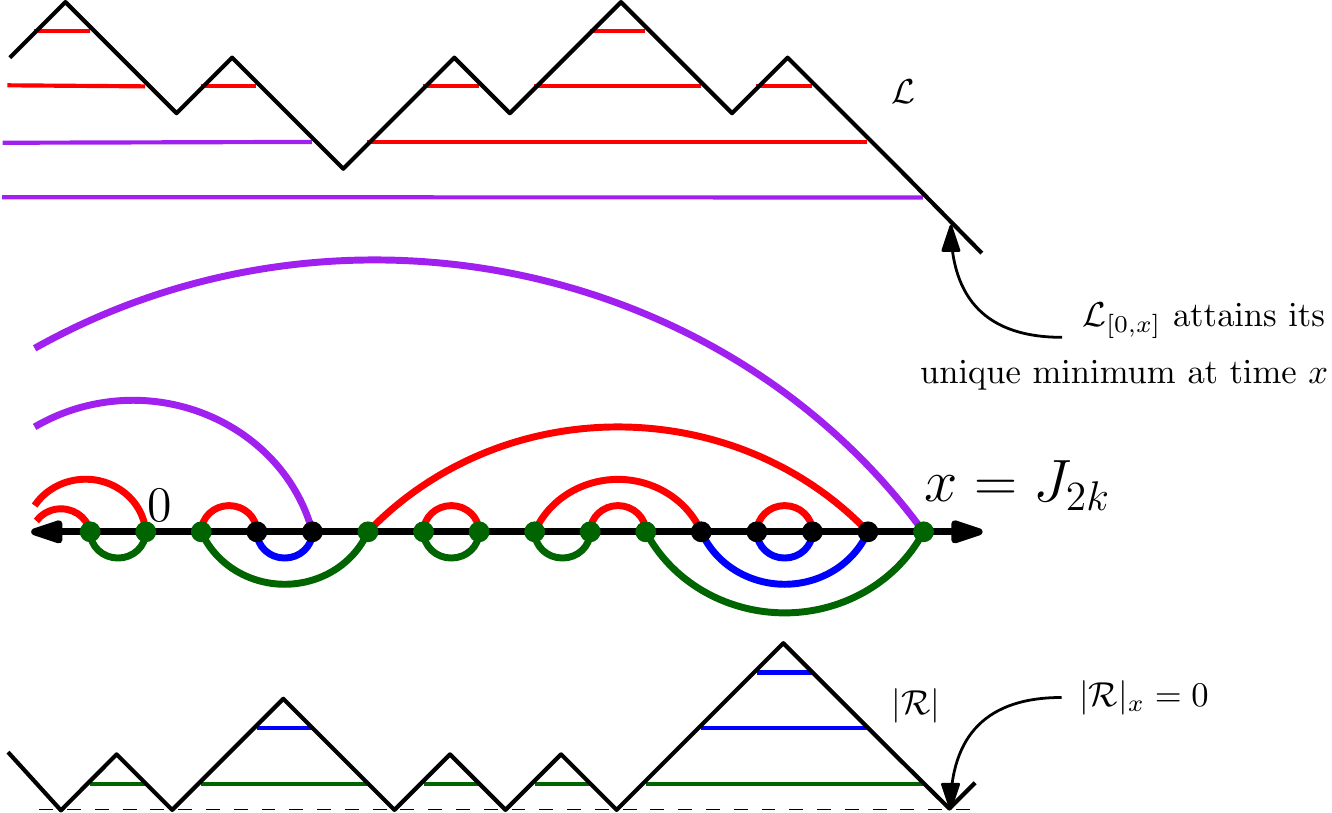}
		\end{minipage}
		\caption{\label{fig-block} An illustration of a block $x=J_{2k}$ with some upper blocks and boundary points in a subset of the UIHPMS. The right endpoints of the purple arcs are upper blocks, while the red arcs are not incident to any upper blocks. Boundary points are the endpoints of the lower green arcs. Lemma~\ref{lem-block} describes a special type of block such that the number of purple arcs is much fewer than the number of green arcs. 
		}
	\end{center}
	\vspace{-3ex}
\end{figure}

\begin{lem}\label{lem-block}
    Fix $\ep>0$. Almost surely, there are infinitely many $k > 0$ such that $J_{2k}$ is a block and there are at most $\ep k$ upper blocks in $(0, J_{2k}]\cap \BB Z$.
\end{lem}

\begin{proof}
    In light of L\'evy's theorem~\eqref{eqn-discrete-levy}, consider another two-sided simple random walk $\wt{\mcl L}\eqD \mcl L$ such that
    $n>0$ is an upper block of $\mcl L$ if and only if $\wt{\mcl L}$ crosses $1/2$ in between times $n-1$ and $n$. Hence, the number of upper blocks in $(0,2n]\cap\BB Z$ is equal to $\ell_{2n}^{\wt{\mcl L}}$, where $\ell$ is defined in~\eqref{eqn-discrete-levy}. Also, the number of zeros of $\mcl R$ in $(0,2n]\cap \BB Z$ is at least $\ell^{\mcl R}_{2n}/2$ because at most two crossings of $1/2$ can correspond to the same zero of $\mcl R$. 
    Therefore, it is enough to prove that a.s.\ there are infinitely many integers $n>0$ such that
    \eqb\label{eqn-local-time-n}
        (\wt{\mcl L}_{2n-1}, \wt{\mcl L}_{2n})=(1,0),\quad \mcl R_{2n}=0,\quad \ell^{\mcl R}_{2n}\ge 2\sqrt{n} \quad\text{ and }\quad \ell^{\wt{\mcl L}}_{2n}< \ep\sqrt{n}.
    \eqe
     Indeed, the first two conditions guarantees that $2n$ is a block and so $2n=J_{2k}$ for some $k>0$, the third condition guarantees that there are at least $\sqrt{n}$ boundary points in $(0,J_{2k}]$ and so $2k\geq \sqrt{n}$, and the fourth condition guarantees that there are at most $\ep\sqrt{n}$ upper blocks in $(0,J_{2k}]$.

    Let $A_n$ be the event that \eqref{eqn-local-time-n} holds.
    Using that $\BB P[(\wt{\mcl L}_{2n-1},\wt{\mcl L}_{2n},\mcl R_{2n})=(1,0,0)]\sim\frac{1}{2\pi n}$ for large $n$, the estimates in~\eqref{eqn-estimate-limit} implies that
    $\BB P\left[A_n \right] \ge \frac{C_3}{n}$
    for some constant $C_3>0$ only depending on $\ep$. In particular, $\sum_{n=1}^\infty \BB P\left[A_n \right]=\infty$.
    We now want to apply
    the Kochen-Stone theorem~\cite{kochen1964note}, which asserts that if $\sum_{n=1}^\infty \BB P\left[A_n \right]=\infty$ and
    \eqb\label{eq:suff-bor}
    \liminf_{N\to \infty} \frac{\sum_{n_1,n_2=1}^N \BB P[A_{n_1}\cap A_{n_2}]}{(\sum_{n=1}^N \BB P[A_{n}])^2} < \infty,
    \eqe
    then $\BB P[A_n\text{ infinitely often}]>0$.
    For $n_2 > n_1>0$, we have
    \eqbn
    \BB P[A_{n_1}\cap A_{n_2}] \le \BB P[(\wt{\mcl L}_{2n_1},\wt{\mcl L}_{2n_2}, \mcl R_{2n_1}, \mcl R_{2n_2})=(0,0,0,0)] \le \frac{C_4}{n_1(n_2-n_1)},
    \eqen
    and similarly $\BB P[A_{n}] \le \frac{C_4}{n}$ for $n>0$, with some constant $C_4$.
    It follows that
    \eqbn
    \sum_{n_1,n_2=1}^N \BB P[A_{n_1}\cap A_{n_2}]\le 2\sum_{1\le n_1\le n_2\le N} \BB P[A_{n_1}\cap A_{n_2}] \le 2C_4\left(\sum_{n_1=1}^N\frac{1}{n_1}+\sum_{n_1,n_2=1}^N \frac{1}{n_1n_2}\right) \le C_5 \log^2(N),
    \eqen
    for another constant $C_5$. Combined with our above lower bound for $\BB P[A_n]$, this implies \eqref{eq:suff-bor}.
    Therefore,
    $\BB P[A_n\text{ infinitely often}]>0$, and Kolmogorov's zero-one law assures that $A_n$ happens infinitely often, almost surely.
\end{proof}

\begin{proof}[Proof of Lemma~\ref{lem-path-finiteness}]
We prove the lemma via a Burton-Keane type argument. Since the UIHPMS is invariant under even translations along the boundary~\eqref{eqn-bdy-translate}, we first look at the case of good boundary points of the form $J_{2k}$ for $k\in \BB Z$. The odd case will be treated analogously at the end of the proof.  
{For $k \in\BB Z$, let $E_k$ be the event that $J_{2k} $ is a good boundary point, i.e., $J_{2k} = H_m$ for some $m\in\BB Z$, and the boundary path $P_m$ is semi-infinite, that is $P_m$ does not hit another good boundary point other than $H_m$.
By the translation invariance of the UIHPMS~\eqref{eqn-bdy-translate} and the fact that the set of $k\in\BB Z$ such that $J_{2k}$ is good is independent from the UIHPMS (see the discussion just below~\eqref{eqn-good-bdy-compare}), $p := \BB P[E_k]$ does not depend on $k$.} We need to show that $p = 0$.

Assume for contradiction that $p>0$.
By the Birkhoff ergodic theorem, 
\eqb  \label{eqn-birkhoff}
    \lim_{k\to \infty}\frac{1}{k}\#\left\{ k'\in [1,k]\cap \BB Z :  E_{k'}\text{ occurs} \right\} = p.
\eqe 
By Lemma~\ref{lem-block} with $\ep=p/4$ and \eqref{eqn-birkhoff}, almost surely, there exist arbitrarily large values of $k\in\BB Z_{>0}$ such that: 
\begin{enumerate}[(i)]
    \item $J_{2k}$ is a block;
    \item there are at most $pk/4$ upper blocks in $(0,J_{2k})\cap\BB Z$;
    \item $\#\{k'\in [1,k]\cap \BB Z :  E_{k'}\text{ occurs}\} > pk/2$.
\end{enumerate} 

Fix $k\in\BB Z_{>0}$ with the above three properties.
{For each $k' \in (0,k)\cap\BB Z$ such that $E_{k'}$ occurs, let $m(k') \in \BB Z$ such that $H_{m(k')} = J_{2k'}$.}
Each semi-infinite path $P_{m(k')}$ for $k'\in (0,k)\cap \BB Z$ such that $E_{k'}$ occurs does not intersect any other such semi-infinite path.
Since $J_{2k}$ is a block, there is no arc connecting $(0,J_{2k} )\cap \BB Z$ and $(J_{2k} ,\infty)\cap \BB Z$. Therefore, each semi-infinite path $P_{m(k')}$ for $k' \in (0,k) \cap\BB Z$ must exit the interval $(0,J_{2k} )\cap \BB Z$ via some arc joining $(0,J_{2k})\cap \BB Z$ to $(-\infty,0]\cap \BB Z$. This arc must be above the real line, since by the construction of the UIHPMS there are no arcs below the real line that cross the ray $\{1/2\} \times (-\infty,0]$. In other words, the right endpoint of such an arc is an upper block. It follows that the number of upper blocks in $(0,J_{2k})\cap \BB Z$ is at least the number of semi-infinite paths started from points in $(0,J_{2k})\cap\BB Z$, which is at least $pk/2$. But, we have chosen $k$ so that the number of upper blocks in $(0,J_{2k})\cap \BB Z$ is at most $pk/4$. This yields a contradiction, and thus $p=0$.

Exactly the same argument shows that a.s.\ the event $E_k'$ that $J_{2k+1} = H_m$ for some $m\in\BB Z$ and $P_m$ is semi-infinite does not occur for any $k\in\BB Z$. This finishes the proof.
\end{proof}

\subsection{Infinitely many paths separating the origin from \texorpdfstring{$\infty$}{infinity}}
\label{sec-infinite-arcs}

Lemma~\ref{lem-path-finiteness} implies that a.s.\ there are no semi-infinite paths in the UIHPMS which start at a good boundary point and never hit another good boundary point. The following lemma will allow us in Section~\ref{sec-uihpms-proof} to rule out more general types of infinite paths.

\begin{lem} \label{lem-infinite-arcs}
Almost surely, there exist infinitely many $m \in\BB Z_{>0}$ such that the boundary path $P_m$ ends at $H_{m'}$ for some $m' \in\BB Z_{\le 0}$.
\end{lem}

We will deduce Lemma~\ref{lem-infinite-arcs} from a general result (Lemma~\ref{lem-ergodic-matching}) for random perfect matchings.

\begin{defn}
    A \textbf{perfect matching} of $\BB Z$ is a function $\phi : \BB Z\rta \BB Z$ such that $\phi(k) \not= k$ and $\phi(\phi(k)) = k$ for each $k\in\BB Z$. We say that $j,k \in\BB Z$ are \textbf{matched} if $\phi(j) = k$, equivalently, $\phi(k) =  j$. We say that $\phi$ is \textbf{non-crossing} if there do not exist integers $j_1,j_2$ such that $j_1  < j_2 < \phi(j_1) < \phi(j_2)$. 
    \end{defn}

    \begin{defn}\label{defn:ergodic}
    A random perfect matching $\phi$ is \textbf{stationary} if $\phi(\cdot +   m) \eqD \phi(\cdot)$ for each $m\in\BB Z$.
    A stationary perfect matching is \textbf{strongly ergodic} if any event which is invariant under shifts of the form $\phi(\cdot) \mapsto \phi(\cdot +2m)$ for all $m\in\BB Z$ has probability zero or one. 
    \end{defn}
    
    \begin{lem} \label{lem-ergodic-matching}
    Let $\phi$ be a strongly ergodic random perfect matching of $\BB Z$.
    Almost surely, for each $j \in\BB Z$ there exists $ k \geq j$ such that $\phi(k) \leq j -1$. If $\phi$ is non-crossing, then a.s.\ for each $j \in\BB Z$ there are infinitely many integers $ k \geq j$ such that $\phi(k) \leq j -1$.
    \end{lem}

    \begin{proof}[Proof of Lemma~\ref{lem-ergodic-matching}]
	We say that $j \in\BB Z$ is \textbf{exposed} if there does not exist $k \geq j$ such that $\phi(k) \leq j-1$. If we represent the matching by an arc diagram (with each arc matching $x$ and $\phi(x)$ plotted in the region $[x\wedge \phi(x) , x\vee \phi(x)]\times[0,\infty)$) then $j$ being exposed is equivalent to the condition that there are no arcs crossing $\{j-1/2\}\times[0,\infty)$.

    We claim that a.s.\ there are no exposed integers. By the stationarity of $\phi$, the probability that $j\in\BB Z$ is exposed does not depend on $j$. By the Birkhoff ergodic theorem, if any integer has a positive probability to be exposed, then a.s.\ there are infinitely many exposed integers.
    Hence we just need to show that the probability that there are infinitely many exposed integers is zero.
    
    Suppose that $j_1 \leq j_2$ are both exposed. By the definition of exposed, every integer in the set $\{j_1,\dots,j_2-1\}$ must have its match in $\{j_1,\dots,j_2-1\}$. 
    Hence $\# \{j_1,\dots,j_2-1\}$ must be even, so $j_1$ and $j_2$ must have the same parity. 
    Consequently, a.s.\ either every exposed integer is even or every exposed integer is odd. {By stationarity,} the probabilities of the events
    \eqb \label{eqn-parity-events}
    \{\text{$\exists$ infinitely many even exposed integers}\} \quad \text{and} \quad 
    \{\text{$\exists$ infinitely many odd exposed integers}\} 
    \eqe
    are the same. The two events in~\eqref{eqn-parity-events} are invariant under translations of the form $\phi\mapsto \phi(\cdot+2m)$ for $m\in\BB Z$, so by strong ergodicity each of these events has probability zero or one. Since we have just seen that these two events are disjoint, they must each have probability zero. Hence a.s.\ there are no exposed integers.
    
    Now assume that $\phi$ is non-crossing and consider a $j\in\BB Z$ with the property that there are only finitely many integers $ k \geq j$ such that $\phi(k) \leq j -1$.
    We will show that there exists an exposed integer. This, combined with our earlier result, will then imply that no such $j$ exists.
    
    To construct an exposed integer, let $k_*$ be the largest $k \geq j$ such that $\phi(k) \leq j-1$. 
    We claim that $k_*+1$ is exposed. Indeed, suppose for contradiction that there exists $k\geq k_*+1$ with $\phi(k) \leq k_*$. 
    By the maximality of $k_*$, we must have $\phi(k) \geq j$. Since $\phi(k_*) \leq j-1$, we have $\phi(k) \not= k_*$ and so $\phi(k) \leq k_*-1$. 
    Therefore, 
    \eqbn
    \phi(k_*) \leq j-1 < j \leq \phi(k)   < k_* < k,
    \eqen
    which contradicts the non-crossing condition.
    \end{proof}
     
{
    \begin{proof}[Proof of Lemma~\ref{lem-infinite-arcs}]
    We divide the proof in three main steps.
    
    \medskip
    
    \noindent\textit{\underline{Step 1:} Reducing to the setting of Lemma~\ref{lem-good-translate}.} 
    Let $(\mcl M^{\op{cut}} , \Gamma^{\op{cut}})$ be the variant of the UIHPMS where we remove the lower arcs joining the good boundary points $\{H_m\}_{m\in\BB Z}$, as in Section~\ref{sec-aperiodic}. By Definition~\ref{def-bdy-path}, the boundary path $P_m$ for $m\in\BB Z$ does not traverse any lower arcs in the UIHPMS joining pairs of good boundary points.  
    Hence, the definition of $P_m$ and the statement of the lemma are unaffected if we replace the UIHPMS by $(\mcl M^{\op{cut}} , \Gamma^{\op{cut}})$. 
    
    Let $(\wh{\mcl M}^{\op{cut}} ,\wh\Gamma^{\op{cut}})$ be sampled from the law of $(\mcl M^{\op{cut}} , \Gamma^{\op{cut}})$ conditioned on the event that $H_0 = 0$, as in Lemma~\ref{lem-good-translate}. Define the good boundary points $\{\wh H_m\}_{m\in\BB Z}$ and the boundary paths $\{\wh P_m\}_{m\in\BB Z}$ with $(\wh{\mcl M}^{\op{cut}} ,\wh\Gamma^{\op{cut}})$ in place of the original UIHPMS. 
    
	We claim that it suffices to show that a.s.\
	\begin{enumerate}
	\item[$(\boxdot)$]\text{$\forall m_0 \in \BB Z$, $\exists$ infinitely many $m > m_0$ s.t.\ the path $\wh P_m$ in $\wh{\mcl M}^{\op{cut}}$ ends at $\wh H_{m'}$ for some $m' \leq m_0$.}\label{eqn-infinite-arcs-show}
	\end{enumerate}
	\newcommand{\refShowSq}{{(\hyperref[eqn-infinite-arcs-show]{$\boxdot$})}}
	Indeed, the law of $(\wh{\mcl M}^{\op{cut}} ,\wh\Gamma^{\op{cut}})$ is absolutely continuous with respect to the law of $( {\mcl M}^{\op{cut}} , \Gamma^{\op{cut}})$.
	So~\refShowSq\ implies that with positive probability, the event~\refShowSq\ holds with $\mcl M^{\op{cut}}$, $P_m$ and $H_{m'}$ in place of $\wh{\mcl M}^{\op{cut}}$, $\wh P_m$ and $\wh H_{m'}$. This event depends on the UIHPMS in a manner which is invariant by translations of the form~\eqref{eqn-bdy-translate}. By the zero-one law for translation invariant events, we get that this event in fact has probability one. 
	
\medskip
 
\noindent\textit{\underline{Step 2:} Constructing a random perfect matching.}
    Recall that $\wh P_m$ is defined as in Definition~\ref{def-bdy-path} with $(\wh{\mcl M}^{\op{cut}} ,\wh\Gamma^{\op{cut}})$ in place of the UIHPMS. 
    Define $\phi : \BB Z \to \BB Z$ by the condition that $\phi(m) =m'$ if and only if $\wh P_m$ ends at $\wh H_{m'}$. 
    By Lemma~\ref{lem-path-finiteness} and the absolutely continuity of $(\wh{\mcl M}^{\op{cut}} ,\wh\Gamma^{\op{cut}})$ with respect to $( {\mcl M}^{\op{cut}} , \Gamma^{\op{cut}})$, a.s.\ each $\wh P_m$ is a finite path ending at a good boundary point. Hence a.s.\ $\phi$ is well-defined. We will prove~\refShowSq\ by applying Lemma~\ref{lem-ergodic-matching} to $\phi$. 
    
    We first check that $\phi$ is a perfect matching. 
    By the definition of $(\wh{\mcl M}^{\op{cut}} ,\wh\Gamma^{\op{cut}})$, none of the good boundary points $\wh H_m$ for $m\in\BB Z$ is incident to a lower arc of $\wh{\mcl M}^{\op{cut}}$. Hence, $\wh P_m$ traverses an upper arc immediately before hitting $\wh H_{\phi(m)}$. Therefore, $\wh P_{\phi(m)}$ is the time reversal of $\wh P_m$ and so $\phi(\phi(m)) = m$ and $\phi(m) \neq m$. Thus $\phi$ is a perfect matching of $\BB Z$. 
\medskip

\begin{figure}[ht!]
    \begin{center} 
        \includegraphics[width=0.6\textwidth]{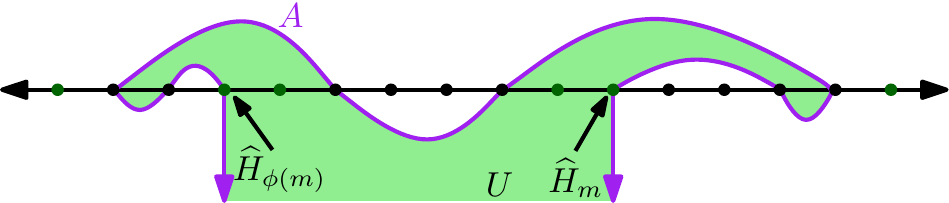} 
    \caption{\label{fig-A-U} The sets $A$ and $U$ used in Step 3 of the proof of Lemma~\ref{lem-infinite-arcs}. Good boundary points are shown in green, and arcs not traversed by $\wh P_m \subset A$ are not shown. The key property of $A$ is that it is not crossed by any arc of $ \wh{\mcl M}^{\op{cut}}$. }
    \end{center}
    \vspace{-3ex}
\end{figure}

\noindent\textit{\underline{Step 3:} Non-crossing, stationarity, and ergodicity.}
    We next show that $\phi$ is non-crossing.  
    Fix $m \in\BB Z$. We are going to show that each point in $(m\wedge\phi(m), m\vee \phi(m)) \cap\BB Z $ is matched with a point in $(m\wedge\phi(m), m\vee \phi(m)) \cap\BB Z$. See Figure~\ref{fig-A-U} for an illustration. Let
    \eqb\label{eqn-path-region}
    A := \wh P_m \cup \left(  \{\wh H_m\} \times (-\infty,0] \right)  \cup \left( \{\wh H_{\phi(m)}\} \times (-\infty,0] \right)  .
    \eqe 
    By the definition of boundary points, for each $m'\in \BB Z$ no arc of $\wh{\mcl M}^{\op{cut}}$ can cross the infinite ray $\{\wh H_{m'}\} \times (-\infty,0]$. This implies in particular that $A$ is the trace of a bi-infinite simple path in $\BB C$ from $\infty$ to $\infty$. By the Jordan curve theorem $A$ separates $\BB C$ into exactly two open connected components. Let $U$ be the open connected component which contains an unbounded subset of the semi-infinite box $(\wh H_{\phi(m)}\wedge \wh H_m,\wh H_{\phi(m)}\vee \wh H_m) \times (-\infty,0)$.
    
    If $m' \in \BB Z\setminus \{m,\phi(m)\}$, then the infinite ray $\{\wh H_{m'}\} \times (-\infty,0]$ intersects $U$ (resp.\ $\BB C\setminus \ol U$) provided $m' \in (m\wedge\phi(m), m\vee \phi(m)) \cap\BB Z $ (resp.\ $m' \in \BB Z\setminus [m\wedge\phi(m), m\vee \phi(m)]  $). This ray cannot cross $A$ (by definition of boundary points), so must be entirely contained in either $U$ or $\BB C\setminus \ol U$. 
    Hence
    \eqbn
    U \cap \{\wh H_{m'}\}_{m' \in \BB Z} = \{\wh H_{m'}\}_{ m' \in (m\wedge\phi(m), m\vee \phi(m)) \cap\BB Z } .
    \eqen
    Since the paths $\wh P_{m'}$ for $m'\in \BB Z\setminus \{m,\phi(m)\}$ cannot cross $A$, this implies the desired property.
      
    By Lemma~\ref{lem-good-translate}, $\phi$ is stationary and strongly ergodic in the sense of Definition~\ref{defn:ergodic}. Therefore, we can apply Lemma~\ref{lem-ergodic-matching} to get that a.s.~\refShowSq\ holds.  
    \end{proof}
}

\subsection{Proofs of Theorem~\ref{thm-infinite-path-half} and Proposition~\ref{prop-alternating}}
\label{sec-uihpms-proof}

We are now ready to prove our main theorem (Theorem~\ref{thm-infinite-path-half}) for the UIHPMS. The key idea is that the paths $P_m$ as in Lemma~\ref{lem-infinite-arcs} act as ``shields'' which an infinite path in the UIHPMS cannot cross.

\begin{proof}[Proof of Theorem~\ref{thm-infinite-path-half}]
Fix $n\in\BB N$. We will show that a.s.\ there is no bi-infinite path of arcs in the UIHPMS which intersects $[-n,n]$. Sending $n\to \infty$ shows that a.s.\ there is no bi-infinite path of arcs in the UIHPMS. See Figure~\ref{fig-infinite-path-half} for an illustration.

\begin{figure}[ht!]
	\begin{center} 
		\includegraphics[width=0.6\textwidth]{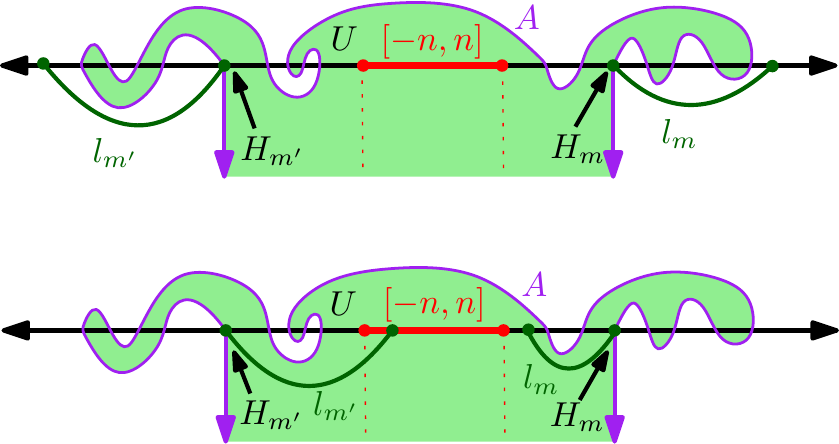} 
		\caption{\label{fig-infinite-path-half} Illustration of the proof of Theorem~\ref{thm-infinite-path-half}. The numbers $m \in \BB Z_{>0}$ and $m' \in \BB Z_{<0}$ are chosen so that the boundary path $P_m$ ends at $H_{m'}$ and is disjoint from $[-n,n] \times (-\infty,0]$. This implies that any bi-infinite path of arcs which hits a point of $[-n,n] \cap\BB Z$ must enter the domain $U$. This is possible only through the lower arcs $l_{m'}$ and $l_m$ of the UIHPMS incident to $H_{m'}$ and $H_m$. We show that, regardless of the configuration of these arcs, the bi-infinite path cannot enter $U$ (the two possible configurations are shown in the upper and lower panels of the figure) concluding that there is no bi-infinite path. }
	\end{center}
	\vspace{-3ex}
\end{figure}

  By, e.g., the definition of the UIHPMS by cutting, there are only finitely many arcs of the UIHPMS which intersect the semi-infinite box $[-n,n] \times (-\infty, 0]$ (recall that we are assuming that each arc joining $x < y$ is draw in such a way that it is contained in $[x,y] \times \BB R$). 
  If $m,m' \in \BB Z$ are chosen so that $P_m$ does not end at $H_{m'}$, then the paths $P_m$ and $P_{m'}$ are disjoint. 
  Therefore, there are only finitely many $m\in\BB Z$ such that $P_m$ traverses an arc of the UIHPMS which intersects $[-n,n]\times (-\infty,0]$.  
  From this and Lemma~\ref{lem-infinite-arcs}, we get that a.s.\ there exists a large enough $m\in\BB Z_{>0}$ such that $P_m$ ends at $H_{m'}$ for some $m'\in\BB Z_{<0}$ and $P_m$ is disjoint from $[-n,n]\times (-\infty,0]$.  

    Define $A$ and $U$ exactly as in the proof of Lemma~\ref{lem-infinite-arcs} with $P_m$ instead of $\wh P_m$, see the paragraph including~\eqref{eqn-path-region}. The condition that $P_m$ is disjoint from $[-n,n]\times (-\infty,0]$ implies that $[-n,n] \subset U$.
    There are only finitely many arcs of the UIHPMS which intersect $U$, so a bi-infinite path of arcs $\gamma$, if it exists, can spend only a finite amount of time in $\ol U$.
    Since $\gamma$ is bi-infinite, it follows that if $\gamma$ hits a point of $[-n,n]\cap\BB Z\subset U$, then there are at least two times when $\gamma$ traverses an arc which joins a point of $\ol U \cap \BB Z$ and a point of $\BB Z\setminus \ol U$. 
    
    By the definition of $A$ and the fact that no lower arcs of the UIHPMS cross the infinite rays $\{H_m\}\times (-\infty,0]$ and $\{H_{m'}\} \times (-\infty,0]$, the only arcs of the UIHPMS which intersect $A$ are the arcs traversed by $P_m$ and the lower arcs incident to $H_{m'}$ and $H_m$. 
    Call these two lower arcs $l_{m'}$ and $l_m$, respectively. 
    Since $P_m\subset A = \bdy U$, these two lower arcs $l_{m'}$ and $l_m$ are the only arcs of the UIHPMS which can possibly join a point of  $\ol U \cap \BB Z$ and a point of $\BB Z\setminus \ol U$. 
    Hence, if $\gamma$ intersects $[-n,n]\cap\BB Z$ then each of $l_{m'}$ and $l_m$ joins $\ol U \cap \BB Z$ and a point of $\BB Z\setminus \ol U$, i.e., $l_m$ joins $H_m$ to a point of $\BB Z\setminus \ol U$ and similarly for $l_{m'}$ (see the first panel in Figure~\ref{fig-infinite-path-half}).
    
    But, $l_{m'}$ and $l_m$ are also the only two arcs of the UIHPMS which can possibly join a point of $U \cap\BB Z$ to a point of $\BB Z\setminus U$ (see the second panel in Figure~\ref{fig-infinite-path-half}). But we have just seen that if $\gamma$ intersects $[-n,n]\cap\BB Z$, then neither $l_{m'}$ nor $l_m$ has an endpoint in $U\cap\BB Z$. Hence, in this case there are no arcs joining a point of $U\cap\BB Z$ to a point of $\BB Z\setminus U$, so a.s.\ there is no bi-infinite path of arcs in the UIHPMS which intersects $[-n,n] \cap\BB Z\subset U$.  
\end{proof}

Using the same idea, we prove that there is a unique infinite path in the PIHPMS.

\begin{proof}[Proof of Proposition~\ref{prop-alternating}] 
By the definition of the PIHPMS given in~\eqref{eqn-sle-rw}, we can couple the PIHPMS with the UIHPMS in such a way that the upper arcs for both infinite meandric systems are the same, the lower arcs joining non-boundary points for both meandric systems are the same, and the sequence of boundary points $\{J_k\}_{k\in\BB Z}$ for both meandric systems is the same.
{Hence we can identify the good boundary points $\{H_m\}_{m\in\BB Z}$ with boundary points of the PIHPMS.} This implies that the boundary paths $P_m$ for $m\in\BB Z$ (Definition~\ref{def-bdy-path}) are unaffected if we replace the UIHPMS by the PIHPMS.

Recall that the path of arcs $\gamma^\circ$ started from $0$ in the PIHPMS is always semi-infinite because $\gamma^\circ$ cannot finish forming a loop as no arc in the lower half plane is incident to $0$ by construction (note that $\gamma^\circ$ is different from $P_0$ since it does not stop when it hits a boundary point). 

By Lemma~\ref{lem-infinite-arcs}, a.s.\ there exist infinitely many $m\in \BB Z_{>0}$ such that $P_m$ ends at $H_{m'}$ for some $m'\in \BB Z_{\leq 0}$. Since $\gamma^\circ$ is semi-infinite, the same argument as in the proof of Theorem~\ref{thm-infinite-path-half} shows that $\gamma^\circ$ must hit either $H_m$ or $H_{m'}$ for each such $m$, and hence $\gamma^\circ$ must traverse $P_m$ or $P_{m'}$ (the time reversal of $P_m$) for each such $m$. This shows that $\gamma^\circ$ hits infinitely many points in each of $\{H_m\}_{m < 0} \subset \{J_k\}_{k<0}$ and $\{H_m\}_{m > 0} \subset \{J_k\}_{k > 0}$, almost surely.
The same argument as in the proof of Theorem~\ref{thm-infinite-path-half} also shows that any other infinite path of arcs in the PIHPMS must also traverse infinitely many of the paths $P_m$ for $m \in \BB Z_{>0}$ as above.
Therefore, any other infinite path must share a portion of arcs with $\gamma^\circ$. But this is possible only if such infinite path is a subpath of $\gamma^\circ$. So $\gamma^\circ$ is the unique maximal infinite path of arcs in the PIHMPS, i.e, the unique infinite path in $\Gamma^\circ$.
\end{proof}

\section{Justification for Conjectures~\ref{conj-system} and~\ref{conj-largest}}
\label{sec-justification}

\subsection{\texorpdfstring{SLE$_8$ on $\sqrt 2$-LQG}{SLE on LQG} via mating of trees} 
\label{sec-mating}

Let $\frk S_n$ be a uniform meandric system of size $n$ and let $(\mathcal M_n,P_n,\Gamma_n)$ be the associated planar map decorated by a Hamiltonian path (corresponding to the real line) and a collection of loops, as in the discussion at the beginning of Section~\ref{sec-conj}. We have already seen in Sections~\ref{sec-mated-crt} and~\ref{sec-lqg} that the infinite-volume analog $(\mathcal M,P)$ of $(\mathcal M_n,P_n)$ is closely connected to $\sqrt 2$-LQG decorated by SLE$_8$. More precisely, due to the encoding of the UIMS by random walks (Section~\ref{sec-infinite}) and the convergence of random walk to Brownian motion, $(\mathcal M,P)$ is connected to the mated-CRT map $\mcl G$, equipped with the left-right ordering of its vertices (Section~\ref{sec-mated-crt}). Furthermore, the mated-CRT map $\mcl G$ is closely connected to $\sqrt 2$-LQG decorated by SLE$_8$ due to the results of~\cite{wedges} (Theorem~\ref{thm-mating}). This strongly suggests that the scaling limit of $(\mathcal M_n,P_n)$ should be given by $\sqrt 2$-LQG decorated by SLE$_8$. 

In fact, the random walk excursions which encode the upper and lower arc diagrams of $(\mathcal M_n,P_n)$ converge under appropriate scaling to two independent Brownian excursions. This can already be interpreted as a convergence statement for $(\mathcal M_n,P_n)$ toward $\sqrt 2$-LQG decorated by SLE$_8$ in the so-called \textbf{peanosphere sense}. See~\cite{ghs-mating-survey} for a survey of this type of convergence and the results that can be proven using it. 

To justify Conjecture~\ref{conj-system}, we still need to explain why $\Gamma_n$ should converge to CLE$_6$. To do this, we will give a physics argument which also provides an alternative heuristic for the scaling limit of $(\mathcal M_n,P_n)$. 

\subsection{Physics argument} 
\label{sec-physics}

We will now explain why Conjecture~\ref{conj-system} can be viewed as a special case of Conjecture~\cite[Conjecture 6.2]{bgs-meander}, which in turn is derived from physics heuristics in~\cite{dgg-meander}, building on~\cite{JK-FLP2,JK-DPL,DCN-FPL-exact,JZJ-FPL-exact}. 

\smallskip

Let $n,m > 0$ and let $G$ be a finite 4-regular graph. The \textbf{fully packed $O(n\times m)$ loop model} on $G$ is the probability measure on pairs of collections of loops (simple cycles) $\Gamma_1, \Gamma_2$ in $G$ with the following properties:
\begin{itemize}
\item Each vertex of $G$ is visited by exactly one loop in each of $\Gamma_1$ and $\Gamma_2$.
\item Each edge of $G$ is visited by either exactly one loop in $\Gamma_1$ or exactly one loop in $\Gamma_2$ (but not both).
\end{itemize}
The probability of each possible configuration $(\Gamma_1,\Gamma_2)$ is proportional to $n^{\# \Gamma_1} m^{\# \Gamma_2}$. 

\smallskip

Suppose now that $G$ is a planar map, instead of just a graph. This in particular means that we have a canonical cyclic ordering of the four edges incident to each vertex of $G$. The \textbf{crossing fully packed $O(n\times m)$ loop model} on $G$ is the variant of the $O(n\times m)$ loop model where we impose the additional constraint that a loop of $\Gamma_1$ and a loop of $\Gamma_2$ cross at each vertex of $G$. Equivalently, the edges incident to each vertex, in cyclic order, must alternate between edges of $\Gamma_1$ and edges of $\Gamma_2$. 

\smallskip

One can also define the (crossing) fully packed $O(n\times m)$ loop model when one or both of $n$ or $m$ is equal to zero. In the case when, e.g., $n = 0$, one considers pairs $(\Gamma_1,\Gamma_2)$ as above subject to the constraint that $\Gamma_1$ consists of a single loop (necessarily a Hamiltonian cycle), and the probability of each configuration is proportional to $m^{\# \Gamma_2}$. 

The relevance of the (crossing) fully packed $O(n\times m)$ loop model to the present paper is that meandric systems are equivalent to the crossing fully packed $O(0\times 1)$ loop model on a planar map, as we now explain. Let $\frk S_n$ be a meandric system of size $n$ and consider the corresponding decorated planar map $(\mathcal M_n,P_n,\Gamma_n)$. 
Then $\mathcal M_n$ is a 4-regular planar map with $2n$ vertices, $P_n$ is a Hamiltonian cycle on $\mathcal M_n$, $\Gamma_n$ is a collection of loops on $\mathcal M_n$ such that each vertex is visited by exactly one loop, and each edge is visited by either the Hamiltonian cycle or one of the loops (but not both). Moreover, at each vertex of $\mathcal M_n$, the Hamiltonian cycle $P_n$ is crossed by a loop of $\Gamma_n$ (this corresponds to the condition that loops of $\frk S_n$ do not touch $\BB R$ without crossing it in Definition~\ref{def-system}). 

Conversely, every 4-regular planar map decorated by a Hamiltonian path and a collection of loops which satisfy the above conditions gives rise to a meandric system: just choose a planar embedding of the map into $\BB R^2\cup\{\infty\}$ under which the Hamiltonian cycle is mapped to the real line.

Hence, the set of possibilities for $(\mathcal M_n,P_n,\Gamma_n)$ is exactly the same as the set of possibilities for a 4-regular planar map with $2n$ edges decorated by a realization of the crossing fully packed $O(0\times 1)$ loop model. In particular, if $\frk S_n$ is sampled uniformly, then $\mathcal M_n$ is a sample from the set of 4-regular planar maps with $2n$ vertices, weighted by the number of possible crossing fully packed $O(0\times 1)$ loop model configurations on the map. Moreover, the conditional law of $(P_n,\Gamma_n)$ given $\mathcal M_n$ is that of the fully packed $O(0\times 1)$ loop model on $\mathcal M_n$. 

Jacobsen and Kondev~\cite{JK-FLP2,JK-DPL} gave a prediction for the scaling limit of the fully packed $O(n\times m)$ loop model on $\BB Z^2$ in terms of a conformal field theory whose central charge is an explicit function of $n$ and $m$, which was later verified (at a physics level of rigor) in~\cite{DCN-FPL-exact,JZJ-FPL-exact}. Building on this, Di Franceso, Golinelli, and Guitter~\cite[Section 3]{dgg-meander} gave a prediction for the scaling limit of the crossing fully packed $O(n\times m)$ loop model on a random planar map. Section 6 of~\cite{bgs-meander} reviews the arguments leading to these predictions and translates the predictions into the language of SLE and LQG. In particular, from \cite[Conjecture 6.2]{bgs-meander} for $n=0$ and $m=1$, we get that the scaling limit of random planar maps decorated by the crossing fully packed $O(0\times 1)$ loop model should be given by $\sqrt 2$-LQG decorated by SLE$_8$ and CLE$_6$. Since this decorated random planar map model is equivalent to a uniform meandric system (as discussed just above), this leads to Conjecture~\ref{conj-system}. 

\subsection{Predictions for exponents via KPZ} 
\label{sec-kpz}

We now explain how Conjecture~\ref{conj-system} leads to predictions for various exponents associated with meandric systems (including Conjecture~\ref{conj-largest}). In this and the next section, we assume that the reader has some familiarity with SLE and LQG.

Consider a $\gamma$-LQG sphere, represented by a metric and a measure on $\BB C$. 
If $X\subset\BB C$ is a random set sampled independently from this metric and measure, then we can define the Hausdorff dimensions $\Delta_0$ and $\Delta_\gamma$ of $X$ with respect to the Euclidean and $\gamma$-LQG metrics, respectively. We re-scale $\Delta_\gamma$ by the reciprocal of the dimension of the whole space, so that $\Delta_\gamma \in [0,1]$. The \textbf{Knizhnik-Polyakov-Zamolodchikov (KPZ) formula}~\cite{kpz-scaling} states that a.s.\ 
\eqb \label{eqn-kpz}
\Delta_0 = \left(2+\frac{\gamma^2}{2} \right) \Delta_\gamma - \frac{\gamma^2}{2} \Delta_\gamma^2 .
\eqe
See~\cite{benjamini-schramm-cascades,shef-kpz, rhodes-vargas-log-kpz,bjrv-gmt-duality,shef-renormalization,aru-kpz,grv-kpz,ghs-dist-exponent,ghm-kpz,wedges,gwynne-miller-char,gp-kpz} for various rigorous versions of the KPZ formula.

Now consider $(\mathcal M_n,P_n,\Gamma_n)$ as above and for $k\in\BB N$ let $\ell_n^k$ be the $k$th largest loop in $\Gamma_n$, i.e., the one with the $k$th largest number of vertices. Conjecture~\ref{conj-system} tells us that if $k$ is fixed, then as $n\rta\infty$, the loop $\ell_n^k$ should converge to the $k$th {largest loop in a} CLE$_6$ on an independent $\sqrt 2$-LQG sphere, i.e., the one with the $k$th largest $\sqrt 2$-LQG length. The Euclidean dimension of {the $k$th largest loop in a} CLE$_6$  is $\Delta_0 = 7/4$, the same as the dimension of an SLE$_6$ curve~\cite{beffara-dim}. By~\eqref{eqn-kpz}, the (re-scaled) dimension of {the $k$th largest loop in a CLE$_6$} with respect to the $\sqrt 2$-LQG metric is
\eqb \label{eqn-loop-exponent-kpz} 
\alpha := \Delta_{\sqrt 2 } = \frac12 \left(3 - \sqrt 2 \right) \approx 0.7929.
\eqe
This means that the number of $\sqrt 2$-LQG balls of $\sqrt 2$-LQG mass $\ep$ needed to cover the $k$th largest loop in a CLE$_6$ is about $\ep^{-\alpha}$. The analog of this in the discrete setting says that 
\eqb \label{eqn-largest-loop}
\BB E\left[\# \{\text{vertices in $\ell_n^k$}\}\right] \approx n^\alpha ,\quad \text{with $\alpha$ as in~\eqref{eqn-loop-exponent-kpz}}. 
\eqe
This gives Conjecture~\ref{conj-largest}. 

Using similar techniques to the ones above, we can also derive predictions for other exponents associated with uniform meandric systems. 
For example, let $\ell_n^1$ be the largest loop in $\Gamma_n$ and let $\op{Cross}_n$ be the number of times that $\ell_n^1$ crosses the vertical line $\{n\} \times \BB R$, i.e., the number of arcs of $\ell_n^1$ which disconnect $n$ from $\infty$ on either of the two sides of the horizontal line. To estimate the growth rate of $\op{Cross}_n$ as $n\rta\infty$, we look at the continuum analog of the set of crossing points. Consider an SLE$_8$ curve $\eta$ from $\infty$ to $\infty$ and an independent whole-plane CLE$_6$. Assume that $\eta$ is parametrized by $\sqrt 2$-LQG area with respect to an independent unit area quantum sphere, so that $\eta : [0,1]\rta \BB R$. 
According to Conjecture~\ref{conj-system}, the continuum analog of the set of crossing points in the definition of $\op{Cross}_n$ is the intersection of $\eta([0,1/2]) \cap \eta([1/2,1])$ with the largest (in the sense of $\sqrt 2$-LQG length) loop in the CLE$_6$.

The set $\eta([0,1/2]) \cap \eta([1/2,1])$ is the union of two SLE$_2$-type curves~\cite[Footnote 4]{wedges}. Therefore, the Euclidean Hausdorff dimension of the intersection of this set with the largest loop in the CLE$_6$ should be equal to the Euclidean Hausdorff dimension of an SLE$_2$ curve intersected with an independent SLE$_6$ curve. Using~\cite{beffara-dim}, on the event that the intersection is non-empty, its Euclidean Hausdorff dimension should be
\eqb
\Delta_0 = 2 - (2-5/4) - (2-7/4) = 1. 
\eqe
Plugging this into~\eqref{eqn-kpz} (with $\gamma=\sqrt 2$) as above, we arrive at the prediction
\eqb \label{eqn-crossing-exponent}
\BB E\left[ \op{Cross}_n \right] \approx n^\nu , \quad \text{where} \quad \nu = \frac12 (3-\sqrt 5) \approx 0.3820.
\eqe

\subsection{Simulations}
\label{sec-sim}

We give details on our numerical simulations in this section.

A sample of a (finite) uniform meandric system of size $n$ can be obtained from a sample of two independent simple random walks of size $2n$ started from $0$, conditioned to stay non-negative and to end at $0$ using~\eqref{eqn-arc-walk}. These conditioned walks $\mcl X_i$ can be sampled quickly from the conditional probability $\BB P\left[\mcl X_{i}-\mcl X_{i-1}=1 \,|\, \mcl X_{1}, \dots, \mcl X_{i-1}\right]$, which boils down to counting the number of simple walk bridges conditioned to be non-negative. Those numbers can be explicitly computed by the reflection property of simple random walks in any case. The corresponding meandric system of size $n$ is then constructed by~\eqref{eqn-arc-walk} as a combinatorial graph and analyzed by JuliaGraphs\cite{Graphs2021}.

Our first simulation computes the median sizes (numbers of vertices) of the 5 largest loops from 100 uniformly sampled meandric systems of size $n$, for increasing values of $n$. The result in Figure~\ref{fig-plots} (Left) suggests that the constant $\alpha\approx 0.7929$ in Conjecture~\ref{conj-largest} is correct.
This is distinguishable from the value $0.8$ guessed in~\cite{kargin-meander-system} because the 95\% confidence interval for the largest loop exponent is $0.7929\pm 0.0022$. 
The size of $k$th largest loop can be written as $C_k n^\alpha$ where $C_k$ is some random variable. In Figure~\ref{fig-plots} (Left), the $y$-intercept is related to $\log(C_k)$ for each $k$. One implication of Conjecture~\ref{conj-system} is that the law of $C_{k+1}/C_{k}$ is tight for each $k$.

Our second simulation computes the medians of $\op{Cross}_n$ (defined in Section~\ref{sec-kpz}) from 100 uniformly sampled meandric systems of size $n$, for increasing values of $n$. The result in Figure~\ref{fig-plots} (Right) suggests that the KPZ prediction~\eqref{eqn-crossing-exponent} is correct. This is strong evidence that our conjectured limiting objects SLE$_8$ and CLE$_6$ are independent, which is not trivial from the definition of meandric systems. See the discussion after Conjecture~\ref{conj-system}.

\begin{figure}[ht!]
    \begin{center}
    \includegraphics[width=0.48\textwidth]{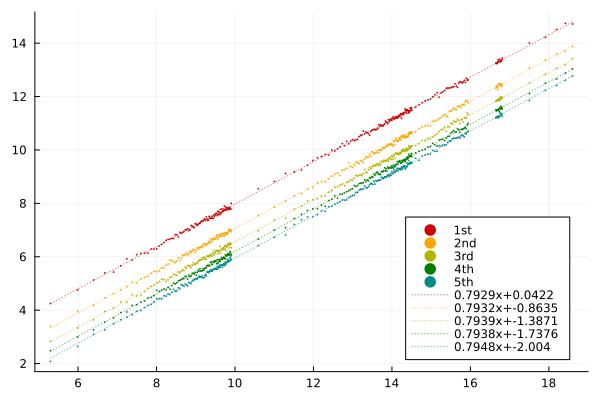}
    \includegraphics[width=0.48\textwidth]{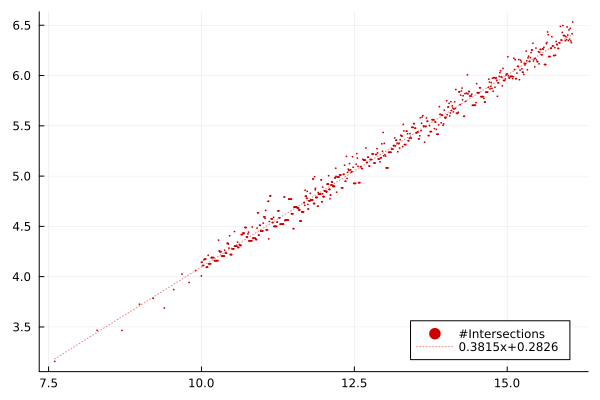}  
    \caption{\label{fig-plots}\textbf{Left:} The plot of $\log(|\ell^{k}_n|)$ versus $\log(2n)$, where $|\ell^{k}_n|$ is the number of vertices of the $k$th largest loop of $\frk S_n$. The 95\% confidence intervals for the slopes are $0.7929\pm 0.0022, 0.7932\pm 0.0021, 0.7939\pm 0.0021, 0.7938\pm 0.0020, 0.7948\pm 0.0019$, respectively, which all include our conjectured value $(3-\sqrt2)/2\approx 0.7929.$ \textbf{Right:} The plot of $\log(|\op{Cross}_n|)$ versus $\log(2n)$, where $\op{Cross}_n$ is defined in Section~\ref{sec-kpz}. {The 95\% confidence interval for the slope is $0.3815\pm 0.0015$, which includes our conjectured value $\frac12 (3-\sqrt5)\approx 0.3820$.}
    }
    \end{center}
    \vspace{-3ex}
\end{figure}

We now give the details of the simulations in Figure~\ref{fig-cle-sim}.
There, we also visualize how the arcs of arc diagrams (which are usually embedded as smooth circular arcs in $\BB R^2$ throughout the paper) associated with a uniformly sampled meandric system may look like fractal loops of a CLE in a proper embedding. For this purpose, we consider a uniform meandric system \emph{with boundary} (as defined below; its infinite-volume version is defined in Section~\ref{sec-half-plane}) and embed the underlying planar map via the Tutte embedding~\cite{gms-tutte}, a.k.a.\ harmonic embedding, so that it conjecturally approximates the $\sqrt{2}$-LQG wedge on the disk. We use a finite meandric system with boundary because Tutte embeddings are easier to define for planar maps with boundary.


The \textbf{uniform meandric system with boundary} can be constructed as in~\eqref{eqn-arc-walk} from two independent conditioned simple random walks with $2n$ steps started from $0$, both conditioned to stay non-negative, but one ends at $0$ as before while the other ends at $2\lfloor n^{1/2}\rfloor$.
These conditioned random walks $\mcl X_i$ can be also sampled from the corresponding conditional probability $\BB P\left[\mcl X_{i}-\mcl X_{i-1}=1 \,|\, \mcl X_{1}, \dots, \mcl X_{i-1}\right]$ exactly as before. Then we may use~\eqref{eqn-arc-walk} to construct arc diagrams, but for the second sampled walk, there are $2\lfloor n^{1/2}\rfloor$ unmatched points called \emph{(good) boundary points}, c.f.~Section~\ref{sec-half-plane} and Section~\ref{sec-aperiodic}.
With these boundary points, we may apply the Tutte embedding.

A path of arcs started from any of the boundary points ends at another boundary point because boundary points have degree 1, whereas non-boundary points have degree 2.
Thus it defines a \emph{boundary path}, c.f.\ Definition~\ref{def-bdy-path}.
We also link each pair of boundary points successively to form loops, c.f. Figure~\ref{fig-cutting} (Left) for the corresponding procedure in the UIHPMS. After coloring loops according to their sizes, we finally obtain the pictures shown in Figure~\ref{fig-cle-sim}.
This construction is analogous to the random walk construction of the UIHPMS defined in Section~\ref{sec-half-plane}.\footnote{Simple walks with $2n$ steps started from $0$ conditioned to stay positive and end at $2\lfloor n^{1/2}\rfloor$ are in bijection with reflected simple walks with $2n+2\lfloor n^{1/2}\rfloor$ steps started from $0$ conditioned to have exactly $2\lfloor n^{1/2}\rfloor$ zeros (after time zero) and to end at $0$, from the bijection in the proof of Lemma~\ref{lem-walk-estimate}. With the latter random walk, an arc diagram can be constructed from~\eqref{eqn-rw-inf-adjacency-half}, but no extra linking is necessary as there are no unmatched points. See also Figure~\ref{fig-cutting} (Right).}

Instead of linking every successive pair of boundary points as above, we may exclude two marked points (specifically, we marked $0$ and another point close to $n$ in our simulations).
As each point other than these two marked points is incident to two arcs, there is always a path from one marked point to the other.
We expect this path approximates a chordal SLE$_6$ curve between two marked points. See the discussion prior to Conjecture~\ref{conj-sle}. The simulation in Figure~\ref{fig-sle-sim} does not contradict any known qualitative behaviors of SLE$_6$.

We also note that there are other embeddings, e.g., circle packing, Riemann uniformization, or Smith embedding (square tiling) under which planar maps without boundary should converge to LQG.
However, the representation of paths and loops on these embeddings may not be as straightforward as in the Tutte embedding.

Every simulation in this section can also be done for random meandric systems encoded by correlated random walks. For example, for $\gamma=1$ (correlation $-\cos(\pi\gamma^2/4) = -1/\sqrt 2$), the simulation suggests the size of the largest loop grows like $n^\alpha$ where $\alpha$ is in the 95\% confidence interval $0.8379\pm 0.0100$. This is consistent with the prediction $\frac12 (5-\sqrt{11})\approx 0.8417$ coming from KPZ if we assume that the loops converge to CLE$_6$. See Remark~\ref{remark-correlated} for further discussion.

\bibliography{cibib, extra}
\bibliographystyle{hmralphaabbrv}

\noindent\textbf{Data availability statement:}
Data sharing is not applicable to this article as no new data were created or analyzed in this study.

\medskip

\noindent\textbf{Conflict of interest:}
The authors have no competing interests to declare that are relevant to the content of this article.

\end{document}